\documentclass[10pt]{amsart}

\usepackage{comment, amsmath, amssymb, amsgen, amsthm, amscd, xspace, color, epsfig, float, epic, multirow, array, lscape}
\usepackage[hypertex]{hyperref}
\usepackage{extarrows}

\makeatletter
\newcommand{\rmnum}[1]{\romannumeral #1}
\newcommand{\Rmnum}[1]{\expandafter\@slowromancap\romannumeral #1@}

\newcommand{\vf}{\varphi}

\newcommand{\Ext}{\operatorname{Ext}}

\newcommand{\Hom}{\operatorname{Hom}}

\newcommand{\rank}{\mathrm{rank}}

\newcommand{\sgn}{\mathrm{sgn}}
\newcommand{\htt}{\mathrm{ht}}

\newcommand{\Ree}{\mathrm{Re}}

\newcommand{\ra}{\rightarrow}

\newcommand{\frb}{\mathfrak{b}}

\newcommand{\frg}{\mathfrak{g}}
\newcommand{\frh}{\mathfrak{h}}

\newcommand{\frl}{\mathfrak{l}}

\newcommand{\frp}{\mathfrak{p}}

\newcommand{\fru}{\mathfrak{u}}

\newcommand{\bbC}{\mathbb{C}}

\newcommand{\bbQ}{\mathbb{Q}}
\newcommand{\bbR}{\mathbb{R}}

\newcommand{\bbZ}{\mathbb{Z}}
\newcommand{\caA}{\mathcal{A}}

\newcommand{\caO}{\mathcal{O}}

\newtheorem{theorem}[equation]{Theorem}
\newtheorem{cor}[equation]{Corollary}
\newtheorem{prop}[equation]{Proposition}
\newtheorem{lemma}[equation]{Lemma}

\theoremstyle{remark}
\newtheorem{remark}[equation]{Remark}

\theoremstyle{definition}
\newtheorem{definition}[equation]{Definition}

\newtheorem{example}[equation]{Example}

\numberwithin{equation}{section}
\begin{document}

\title[Jantzen coefficients]{Jantzen coefficients and simplicity of generalized Verma modules}

\author{Wei Xiao}
\address[Xiao]{College of Mathematics and statistics, Shenzhen Key Laboratory of Advanced Machine Learning and Applications, Shenzhen University,
Shenzhen, 518060, Guangdong, P. R. China}
\email{xiaow@szu.edu.cn}

\author{Ailin Zhang*}
\address[Zhang]{College of Mathematics and statistics, Shenzhen Key Laboratory of Advanced Machine Learning and Applications, Shenzhen University,
Shenzhen, 518060, Guangdong, P. R. China}
\email{az304@szu.edu.cn}

\thanks{The first author is supported by the National Science Foundation of China (Grant No. 11701381) and Guangdong Natural Science Foundation (Grant No. 2017A030310138). The second author is supported by the National Science Foundation of China (Grant No. 11504246).}

\thanks{*Corresponding author}



\bigskip

\begin{abstract}
 The main purpose of this paper is to establish new tools in the study of $\caO^\frp$. We introduce the Jantzen coefficients of generalized Verma modules. They come from the Jantzen's simplicity criteria for generalized Verma modules and have a deep relation with the structure of $\caO^\frp$. We develop a reduction process to compute those coefficients by considering basic generalized Verma modules. These modules are induced from maximal parabolic subalgebras and have maximal nontrivial singularity. The classification of such modules is also obtained in this paper. As the first application of our results, we give a refinement of Jantzen's simplicity criteria for generalized Verma modules.
\end{abstract}


\maketitle

%
%
\section{Introduction}
%
%

Many interesting representations of Lie groups and Lie algebras can be studied through the category $\caO$ introduced by Bernstein-Gelfand-Gelfand \cite{BGG}. Rocha \cite{R} initiated the study of the category $\caO^\frp$ as a natural generalization of $\caO$. The Koszul duality \cite{So, BGS, B}, which reveals a hidden parabolic-singular duality of blocks in $\caO^\frp$, makes $\caO^\frp$ an interesting object of study in its own right.


The Jantzen coefficient comes from the Jantzen's simplicity criteria for generalized Verma modules of $\caO^\frp$. It is closely related to Jantzen filtration \cite{HX} and leading coefficients (the $\mu$-functions \cite{KL}) of Kazhdan-Lusztig polynomials \cite{X2}. One goal of this paper is to establish necessary results about it for studying related topics, such as simplicity criteria of generalized Verma modules \cite{J, He, HKZ, BX1, BX2}, homomorphism between generalized Verma modules \cite{Bo, BC, BEJ, L1, L2, M1, M2, M3, X1} and representation types of the blocks of $\caO^\frp$ \cite{BN, P2} and so on.

The Jantzen coefficient turns out to be quite a useful tool in the study of $\caO^\frp$, with which we get a refinement of Jantzen's simplicity criteria in this paper. In \cite{HXZ, X2}, it is used to completely solve the open problem about blocks of $\caO^\frp$ (see \cite{ES, BN, Br, P1} or \S9.15 in \cite{H3}). Combined it with a generalized Jantzen sum formula \cite{HX}, we are able to efficiently determine radical filtrations of many generalized Verma modules. We will also use it to give representation types of blocks in $\caO^\frp$\cite{XZ}.


Now we describe our main idea and related results of this paper. Let $\frg$ be a complex semisimple Lie algebra. Suppose that $\frp\supset\frb\supset\frh$ is a standard parabolic subalgebra of $\frg$ containing a fixed Borel subalgebra $\frb$ and a fixed Cartan subalgebra $\frh$. Let $\Phi$ be the root system of $(\frg, \frh)$ with the positive system $\Phi^+$ and the simple system $\Delta$ corresponding to $\frb$. Let $W$ be the Weyl group of $\Phi$. Note that the parabolic subalgebra $\frp=\frl\oplus\fru$ is determined by a subset $I\subset\Delta$, where $\frl$ is the Levi subalgebra and $\fru$ is the nilpotent radical. In particular, the subsystem $\Phi_I$ generated by $I$ is the root system of $(\frl, \frh)$. Set
\[
\Lambda_I^+:=\{\lambda\in\frh^*\ |\ \langle\lambda, \alpha^\vee\rangle\in\bbZ^{>0}\ \mbox{for all}\ \alpha\in I\},
\]
where $\langle-,-\rangle$ is the bilinear form on $\frh^*$ induced from the Killing form and $\alpha^\vee$ is the coroot of $\alpha$. Let $\rho$ be the half sum of positive roots. For $\lambda\in\Lambda_I^+$, the generalized Verma module is defined by
\[
M_I(\lambda):=U(\frg)\otimes_{U(\frp)}F(\lambda-\rho),
\]
where $F(\lambda-\rho)$ is a finite dimensional simple $\frp$-modules of highest weight $\lambda-\rho$. For convenience, we also use the notation $M(\lambda, \Phi_I, \Phi)$ for $M_I(\lambda)$ when we need to deal with generalized Verma modules associated with different root systems. In this paper, $\caO^\frp$ is the category of all finitely generated $\frg$-modules $M$ that are semisimple as $\frl$-modules and locally $\frp$-finite. If $I=\emptyset$, then $\caO^\frp=\caO^\frb$ is the usual BGG category $\caO$ and $M(\lambda):=M_\emptyset(\lambda)$ is the Verma module with highest weight $\lambda-\rho$. Let $K(\caO^\frp)$ be the Grothendieck group of the category $\caO^\frp$ with $[M]\in K(\caO^\frp)$ for $M\in\caO^\frp$. For $\lambda\in\frh^*$, denote
\[
\theta(\lambda)=\sum_{w\in W_I}(-1)^{\ell(w)}[M(w\lambda)],
\]
where $W_I$ is the Weyl group of $\Phi_I$. Then $\theta(\lambda)=M_I(\lambda)$ for $\lambda\in\Lambda_I^+$ (see for example Proposition 9.6 in \cite{H3}, keeping in mind the notation difference).

We start with the Jantzen's simplicity criterion \cite{J} which was widely used in representation theory (e.g., \cite{EHW, M1}). For $\lambda\in\frh^*$, set
\[
\Psi_\lambda^+:=\{\beta\in\Phi^+\backslash\Phi_I\ |\ \langle\lambda, \beta^\vee\rangle\in\bbZ^{>0}\}.
\]

\smallskip
\noindent{\bf Jantzen's simplicity criterion} Let $\lambda\in\Lambda_I^+$. The $\frg$-module $M_I(\lambda)$ is simple if and only if
\begin{equation}\label{inteq1}
\sum_{\beta\in\Psi_\lambda^+}\theta(s_\beta\lambda)=0.
\end{equation}

Since the set $\Psi_\lambda^+$ might be too large and the above sum formula might be too complicated, in practice there are a lot of simplifications of the criterion for special cases \cite{J, Ku, He, HKZ, BX2}, while a general approach is lacking.

To overcome this, we need another result of Jantzen. If $\Phi'$ is a subsystem of $\Phi$, there exists a unique weight $\lambda|_{\Phi'}$ in the subspace $\bbC\Phi'$ so that $\langle\lambda|_{\Phi'}, \alpha^\vee\rangle=\langle\lambda, \alpha^\vee\rangle$ for all $\alpha\in\Phi'$. Set
\[
\Phi_{\beta,1}:=(\bbQ\Phi_I+\bbQ\beta)\cap\Phi.
\]
The following reduction result can be found in \cite{J}.

\smallskip
\noindent{\bf Parabolic reduction (Jantzen)} The $\frg$-module $M_I(\lambda)$ is simple if and only if $M(\lambda|_{\Phi_{\beta, 1}}, \Phi_I, \Phi_{\beta, 1})$ is simple for all $\beta\in\Psi_\lambda^+$.
\smallskip

In other words, the simplicity of $M_I(\lambda)$ can be determined by generalized Verma modules associated with smaller simple root systems. In fact, these modules are induced from maximal parabolic subalgebras (since $\rank \Phi_{\beta, 1}=\rank\Phi_I+1$) of corresponding simple Lie subalgebras. But there are still too many of such modules to investigate.

Although the parabolic reduction is seldom used in practice, it provides deep reduction idea. Inspired by the Koszul duality and related parabolic-singular duality, we believe that there should be a corresponding singular reduction. In fact, when $\lambda\in\frh^*$ is fixed, set
\[
\Phi_{\beta,2}:=(\bbQ\Phi_\lambda+\bbQ\beta)\cap\Phi,
\]
where $\Phi_\lambda=\{\alpha\in\Phi\mid\langle\lambda, \alpha\rangle=0\}$. We formulate and prove the following result.

\smallskip
\noindent{\bf Singular reduction} (Lemma \ref{redllem4}) The $\frg$-module $M_I(\lambda)$ is simple if and only if $M(\lambda|_{\Phi_{\beta, 2}}, \Phi_I\cap\Phi_{\beta, 2}, \Phi_{\beta, 2})$ is simple for all $\beta\in\Psi_\lambda^+$.
\smallskip

The proof of the singular reduction is a little laborious. For $\lambda\in\Lambda_I^+$, set
\[
\Psi_\lambda^{++}:=\{\beta\in\Psi_\lambda^+\mid\langle s_\beta\lambda, \alpha\rangle\neq0\ \mbox{for all}\ \alpha\in\Phi_I\}.
\]
Then $\theta(s_\beta\lambda)$ is nonzero if and only if $\beta\in\Psi_\lambda^{++}$ (Proposition \ref{jansprop1}). In this case, it is easy to see that the isotropic group of $s_\beta\lambda$ under the action of $W_I$ is trivial. There exists $w\in W_I$ such that $\mu=w s_\beta\lambda\in\Lambda_I^+$. Therefore $\theta(s_\beta\lambda)=(-1)^{\ell(w)}[M_I(\mu)]$, where $\ell(-)$ is the length function on $W$. Note that $[M_I(\mu)]$ ($\mu\in\Lambda_I^+$) form a basis of $K(\caO^\frp)$.

\smallskip
\noindent{\bf Definition} (Definition \ref{sjdef1})
Write
\begin{equation}\label{inteq2}
\sum_{\beta\in\Psi_\lambda^+}\theta(s_\beta\lambda)=\sum_{\lambda>\mu\in\Lambda_I^+}c(\lambda, \mu)[M_I(\mu)].
\end{equation}
The coefficient $c(\lambda, \mu)$ is called the \emph{Jantzen coefficient} associated with $(\lambda, \mu)$. Here $\lambda>\mu$ means $\Hom(M(\mu), M(\lambda))\neq0$.
\smallskip

Fix $\lambda\in\Lambda_I^+$. The coefficients $c(\lambda, \mu)$ are nonzero for only finitely many $\mu\in\Lambda_I^+$. Jantzen's simplicity criterion implies that $M_I(\lambda)$ is simple if and only if all the Jantzen coefficients $c(\lambda, \mu)=0$.

The key feature of Jantzen coefficients is that they possess several invariant properties. These invariant properties are formulated and proved in section 4 (Lemma \ref{invlem1}-\ref{invlem4} and Lemma \ref{invlem5}-\ref{invlem6}). They form the framework of our theory of Jantzen coefficient and make it a useful tool. In particular, the singular reduction can be proved by the singular invariance lemma (see Lemma \ref{invlem4}) of Jantzen coefficients.

The situation have been greatly improved after the singular reduction is settled. For $\lambda\in\frh^*$, define
\[
\Phi_{[\lambda]}:=\{\alpha\in\Phi\mid\langle\lambda, \alpha^\vee\rangle\in\bbZ\}.
\]
Let $\Phi_{\beta, 0}$ be the irreducible component of $\Phi$ with $\beta\in\Phi_{\beta, 0}$. We can construct the following sequences of subsystems for $\beta\in\Psi_\lambda^+$:
\[
\Phi_{[\lambda]}=\Phi_0(\beta)\supset\Phi_1(\beta)\supset\Phi_2(\beta)\supset\cdots\supset\Phi_k(\beta)\supset\cdots
\]
such that for $i\in\bbZ^{\geq0}$,
\[
\Phi_{i+1}(\beta)
=\left\{\begin{aligned}
&(\Phi_{i}(\beta))_{\beta,0}\qquad\qquad\qquad\quad\ \mbox{if}\ i\equiv0\ (\mathrm{mod}4);\\
&(\Phi_{i}(\beta))_{\beta,1}\qquad\qquad\qquad\quad\ \mbox{if}\ i\equiv1\ (\mathrm{mod}4);\\
&(\Phi_{i}(\beta))_{\beta,0}\qquad\qquad\qquad\quad\ \mbox{if}\ i\equiv2\ (\mathrm{mod}4);\\
&(\Phi_{i}(\beta))_{\beta,2}\qquad\qquad\qquad\quad\ \mbox{if}\ i\equiv3\ (\mathrm{mod}4).
\end{aligned}
\right.
\]
The sequence is stationary, that is, there exists $k\geq0$ such that $\Phi_k(\beta)=\Phi_{k+1}(\beta)=\cdots$. Denote $\Phi(\beta)=\Phi_k(\beta)$. Then $\rank\Phi_I\cap\Phi(\beta)=\rank\Phi_\lambda\cap\Phi(\beta)=\rank\Phi(\beta)-1$. We get a generalized Verma module $M(\lambda|_{\Phi(\beta)}, \Phi_I\cap\Phi(\beta), \Phi(\beta))$ of irreducible subsystem $\Phi(\beta)$ with integral highest weight. It is induced from a maximal parabolic subalgebra of $\Phi(\beta)$ and has maximal nontrivial singularity. Such a module is called a \emph{basic generalized Verma module}. The following result is an immediate consequence of this reduction process.

\smallskip
\noindent{\bf Theorem A} (Theorem \ref{redsthm1}) The $\frg$-module $M_I(\lambda)$ is simple if and only if the basic generalized Verma module $M(\lambda|_{\Phi(\beta)}, \Phi_I\cap \Phi(\beta), \Phi(\beta))$ is simple for all $\beta\in\Psi_\lambda^+$.
\smallskip

This makes the classification of basic generalized Verma modules necessary for our reasoning. Suppose that $M_I(\lambda)$ is a basic generalized Verma module. Since $\lambda$ is integral, there exists a unique dominant weight $\overline\lambda\in W\lambda$ with $\Phi_{\overline\lambda}=\Phi_J$ for some $J\subset\Delta$. Moreover, $\rank\Phi_J=\rank\Phi_\lambda=\rank\Phi-1$. We can write $\lambda=w\overline\lambda$, where $w$ is contained in
\[
{}^IW^J=\{w\in W\mid \ell(xwy)=\ell(x)+\ell(w)+\ell(y)\ \mbox{for any}\ x\in W_I, y\in W_J\}.
\]
We can assume that $I=\Delta\backslash\{\alpha_i\}$ and $J=\Delta\backslash\{\alpha_j\}$ for some $i, j\in\{1, \cdots, n\}$, where $\Delta=\{\alpha_1, \cdots, \alpha_n\}$ are standard simple roots of $\Phi$ (\cite{H1}, \S11.4). The triple $(\Phi, i, j)$ is called a \emph{basic system}. The classification of basic systems are given as follows.

\smallskip
\noindent{\bf Theorem B} (Theorem \ref{bthm1}) The basic system $(\Phi, i, j)$ must be one of the following cases.
\begin{itemize}
\item [(1)] $(A_1, 1, 1)$, $(A_2, 1, 1)$, $(A_2, 1, 2)$, $(A_2, 2, 1)$, $(A_2, 2, 2)$, $(A_3, 2, 2)$;

\item [(2)] $(B_2, 1, 1)$, $(B_2, 1, 2)$, $(B_2, 2, 1)$, $(B_2, 2, 2)$, $(B_3, 2, 2)$, $(B_3, 2, 3)$, $(B_3, 3, 2)$, $(B_4, 3, 3)$;

\item [(3)] $(C_2, 1, 1)$, $(C_2, 1, 2)$, $(C_2, 2, 1)$, $(C_2, 2, 2)$, $(C_3, 2, 2)$,  $(C_3, 2, 3)$, $(C_3, 3, 2)$, $(C_4, 3, 3)$;

\item [(4)] $(D_4, 2, 2)$, $(D_5, 3, 3)$;

\item [(5)] $(E_6, 4, 4)$, $(E_7, 4, 4)$, $(E_7, 4, 5)$, $(E_7, 5, 4)$, $(E_8, 3, 4)$, $(E_8, 4, 3)$, $(E_8, 4, 4)$, $(E_8, 4, 5)$, $(E_8, 5, 4)$, $(E_8, 5, 5)$;

\item [(6)] $(F_4, 2, 2)$, $(F_4, 2, 3)$, $(F_4, 3, 2)$, $(F_4, 3, 3)$;

\item [(7)] $(G_2, 1, 1)$, $(G_2, 1, 2)$, $(G_2, 2, 1)$, $(G_2, 2, 2)$.
\end{itemize}
\smallskip

Based on this classification, we can find all the basic generalized Verma modules of each basic system (see \S5.2). With a generalized Jantzen sum formula, the radical filtration of all the basic generalized Verma modules are given in \cite{HX}. The basic systems and basic generalized Verma modules turn out to be quite important in the study of $\caO^\frp$. For example, one might wonder what makes modules of the other classical types behave so different from those of type $A$. Part of this are affected by the basic systems $(B_2, 1, 2)$, $(B_2, 2, 1)$, $(C_2, 1, 2)$, $(C_2, 2, 1)$ and $(D_4, 2, 2)$ which are semisimple systems contain more than one simple module. In fact, whenever one has a speculation about $\caO^\frp$, it seems natural to check these extreme cases first.

The reduction process and the classification of basic generalized Verma modules provide us a new simplicity criterion (Theorem A) for generalized Verma modules. The calculation of Jantzen coefficients could bring further refinement. All the nonzero Jantzen coefficients of basic generalized Verma modules are obtained in this paper (Theorem \ref{nzj}). These coefficients show that most basic systems are semisimple. In view of the invariant properties (Lemma \ref{invlem1}-\ref{invlem4}), we have the following result.

\smallskip
\noindent{\bf Theorem C} (Theorem \ref{njthm1}) Jantzen coefficient $|c(\lambda, \mu)|\leq1$ unless $\Phi=E_7, E_8$. In these exceptional cases, $|c(\lambda, \mu)|\leq2$.
\smallskip

Let $\Phi=B_n$, $C_n$ or $D_n$. Suppose that $\Delta\backslash I=\{\alpha_{q_1}, \dots, \alpha_{q_{m-1}}\}$ with
\[
0<q_1<\cdots< q_{m-1}\leq n.
\]
Write $\lambda=(\lambda_1, \cdots, \lambda_n)$ for any $\lambda\in\frh^*$. The following theorem shows when Jantzen coefficients for classical Lie algebras are vanished.

\smallskip
\noindent{\bf Theorem D} (Theorem \ref{rcthm1}) Let $\Phi$ be a classical root system and $\lambda, \mu\in\Lambda_I^+$. Suppose that $\mu=ws_\beta\lambda$ for some $w\in W_I$ and $\beta\in\Psi_\lambda^{++}$. Then $c(\lambda, \mu)=0$ if and only if one of the following conditions is satisfied.
\begin{itemize}
  \item [(1)] $\Phi=B_n$ $($resp. $C_n)$, $\beta=e_i$ $($resp. $2e_i)$ or $e_i+e_j$ for $q_{s-1}<i\leq q_s\leq q_{m-1}<j\leq n$ and $1\leq s<m$. Moreover, $\lambda_i=\lambda_j\in\frac{1}{2}\bbZ^{>0}$ $($resp. $\bbZ^{>0}$$)$ and $\lambda_k\neq 0$, $-\lambda_i$ for $q_{s-1}<k\leq q_s$.
  \item [(2)] $\Phi=B_n$ $($resp. $C_n)$, $\beta=e_i$ $($resp. $2e_i)$ or $e_i+e_j$ for $q_{s-1}<i<j\leq q_s$ and $1\leq s<m$. Moreover, $\lambda_i\in\bbZ^{>0}$, $\lambda_j=0$, $\lambda_k\neq-\lambda_i$ for $q_{s-1}<k\leq q_s$ and $\lambda_l\neq\lambda_i$ for $q_{m-1}<l\leq n$.
  \item [(3)] $\Phi=D_n$, $\beta=e_i+e_j$ or $e_i+e_k$ for $q_{s-1}<i<j\leq q_s\leq q_{m-1}<k<n$ and $1\leq s<m$. Moreover, $\lambda_i=\lambda_k\in\bbZ^{>0}$, $\lambda_j=\lambda_n=0$ and $\lambda_l\neq-\lambda_i$ for $q_{s-1}<l\leq q_s$.
\end{itemize}
\smallskip

Combined with Theorem C, all the Jantzen coefficients are determined for classical root systems up to a sign. Theorem D plays essential role in the problem of blocks for $\caO^\frp$ \cite{X2}. Recall that Jantzen's simplicity criteria implies that a generalized Verma modules is simple if and only if all its Jantzen coefficients are vanished. In the relatively easy case of type $A$, Theorem D recovers a result of Jantzen (\cite[Satz 4]{J}, Theorem \ref{jansthm3}). Explicit simplicity criteria for classical root systems $B_n$, $C_n$ and $D_n$ are described in the last section (Theorem \ref{simcthm1}-\ref{simcthm3}).

\begin{example}\label{simcex1}
Let $\Phi=B_8$ and $\Delta\backslash I=\{e_2-e_3, e_5-e_6\}$ (so $q_1=2$, $q_2=5$ and $m=3$). Choose
\[
\lambda=(2, 1\ |\ 2, -1, -3\ |\ 4, 2, 1\ )\in\Lambda_I^+,
\]
where $I$ separates the weight into three segments (different segments are divided by the vertical lines). If we apply original criterion (\ref{inteq1}), it will be quite time consuming to determine the set $\Psi_\lambda^+$ and calculate the corresponding sum formula.

On the other hand, it is easy to check that $\beta=e_5+e_6\in\Psi_\lambda^{++}$. Theorem D implies that $c(\lambda, \mu)\neq0$, where $\mu=ws_\beta\lambda\in\Lambda_I^+$ for some $w\in W_I$ . Therefore $M_I(\lambda)$ is not simple in view of (\ref{inteq2}).
\end{example}

This paper is organized as follows. In section 2, we provide necessary notations and definitions; The reduction process is built in section 3, while several invariant properties of Jantzen coefficients and the singular reduction are proved in section 4. The classification of basic systems and basic generalized Verma modules are obtained in section 5. We give the Jantzen coefficients for basic generalized Verma modules and posets for basic systems in section 6. In the last section, we calculate the Jantzen coefficients and get a refinement of Jantzen's simplicity criteria for classical Lie algebras.

%
%
\section{Notations and definitions}
%
%

\subsection{General notations}

We adopt several notations in \cite{H3}. Let $\frg$ be a complex semisimple Lie algebra with a fixed Cartan subalgebra $\frh$ contained in a Borel subalgebra $\frb$. Let $\Phi\subset\frh^*$ be the root system of $(\frg, \frh)$ with a positive system $\Phi^+$ and a simple system $\Delta\subset\Phi^+$ corresponding to $\frb$. Denote by $\frg_\alpha$ the root subspace of $\frg$ associated with $\alpha\in\Phi$. Note that every subset $I\subset\Delta$ generates a subsystem $\Phi_I\subset\Phi$ with a positive root system $\Phi^+_I:=\Phi_I\cap\Phi^+$. Denote by $W$ (resp. $W_I$) the Weyl group of $\Phi$ (resp. $\Phi_I$) with the longest element $w_0$ (resp. $w_I$). Let $\ell(-)$ be the length function on $W$. It can be viewed as the length function on $W_I$ via restriction. The action of $W$ on $\frh^*$ is given by $s_\alpha\lambda=\lambda-\langle\lambda, \alpha^\vee\rangle\alpha$ for $\alpha\in\Phi$ and $\lambda\in\frh^*$. Here $\langle-, -\rangle$ is the bilinear form on $\frh^*$ induced from the Killing form and $\alpha^\vee:=2\alpha/\langle\alpha, \alpha\rangle$ is the coroot of $\alpha$.

We say $\lambda\in\frh^*$ is \emph{regular} (resp. $\Phi_I$-\emph{regular}) if $\langle\lambda, \alpha^\vee\rangle\neq0$ for all roots $\alpha\in\Phi$ (resp. $\alpha\in\Phi_I$). Otherwise we say $\lambda$ is \emph{singular} (resp. $\Phi_I$-\emph{singular}). We say $\lambda$ is \emph{integral} (resp. $\Phi_I$-\emph{integral}) if $\langle\lambda, \alpha^\vee\rangle\in\bbZ$ for all $\alpha\in\Phi$ (resp. $\alpha\in \Phi_I$). An integral weight $\lambda\in\frh^*$ is \emph{dominant} (resp. \emph{anti-dominant}) if $\langle\lambda, \alpha^\vee\rangle\in\bbZ^{\geq0}$ (resp. $\langle\lambda, \alpha^\vee\rangle\in\bbZ^{\leq0}$) for all $\alpha\in\Delta$. When $\lambda$ is integral, there exists a unique dominant integral weight $\overline\lambda$ in the orbit $W\lambda$ such that $\lambda=w\overline\lambda$ for some $w\in W$. Then $\underline\lambda=w_0\overline\lambda$ is the unique anti-dominant weight in $W\lambda$.

\subsection{Category $\caO^\frp$}

Let $\frl_I:=\frh\oplus\sum_{\alpha\in\Phi_I}\frg_\alpha$ be a Levi subalgebra and $\fru_I:=\bigoplus_{\alpha\in\Phi^+\backslash\Phi_I^+}\frg_\alpha$ be a nilpotent radical of $\frg$. We obtain a standard parabolic subalgebra $\frp_I:=\frl_I\oplus\fru_I$ of $\frg$. For simplicity, we frequently drop the subscript when $I$ is fixed. For $\frp=\frp_I$, the category $\caO^\frp$ is the category of all finitely generated $\frg$-modules $M$ that are semisimple as $\frl_I$-modules and locally $\frp_I$-finite. In particular, $\caO^\frb$ is the usual Bernstein-Gelfand-Gelfand category $\caO$. Put
\[
\Lambda_I^+:=\{\lambda\in\frh^*\ |\ \langle\lambda, \alpha^\vee\rangle\in\bbZ^{>0}\ \mbox{for all}\ \alpha\in I\}.
\]
Set $\rho:=\frac{1}{2}\sum_{\alpha\in\Phi^+}\alpha$. For $\lambda\in\Lambda_I^+$, the \emph{generalized Verma module} is defined by
\[
M_I(\lambda):=U(\frg)\otimes_{U(\frp_I)}F(\lambda-\rho),
\]
where $F(\lambda-\rho)$ is a finite dimensional simple $\frl_I$-modules of highest weight $\lambda-\rho$, and has trivial $\fru_I$-actions viewed as a $\frp_I$-module. Thus $M_I(\lambda)\in\caO^\frp$ and $M(\lambda):=M_\emptyset(\lambda)$ is the Verma module with highest weight $\lambda-\rho$.  Let $L(\lambda)$ be the simple quotient of $M(\lambda)$.  These highest weight modules have the same infinitesimal character $\chi_\lambda$, where $\chi_\lambda$ is an homomorphism from the center $Z(\frg)$ of $U(\frg)$ to $\bbC$ such that $z\cdot v=\chi_\lambda(z)v$ for all $z\in Z(\frg)$ and $v\in M(\lambda)$. Moreover, $\chi_\lambda=\chi_\mu$ when $\mu\in W\lambda$. Denote by $\caO^\frp_\lambda$ the full subcategory of $\caO^\frp$ containing modules $M$ on which $z-\chi_\lambda(z)$ acts as locally nilpotent operator for all $z\in Z(\frg)$.



\subsection{The posets ${}^IW^J$.} For $\lambda\in\frh^*$, set
\[
\Phi_\lambda:=\{\beta\in\Phi\mid\langle\lambda,\beta\rangle=0\}.
\]
It is obvious that $\Phi_\lambda$ is a subsystem of $\Phi$. Define
\[
{}^IW=:\{w\in W\mid \ell(s_\alpha w)=\ell(w)+1\ \mbox{for all}\ \alpha\in I\}.
\]
When $\lambda$ is integral, recall that $\overline\lambda$ is the unique dominant weight in $W\lambda$. Set
\[
J=\{\alpha\in\Delta\mid\langle\overline\lambda, \alpha\rangle=0\}.
\]
Then $\Phi_\lambda=w\Phi_J\simeq\Phi_J=\Phi_{\overline\lambda}$ and $W_J=\{w\in W\mid w\overline\lambda=\overline\lambda\}$. Put
\[
{}^IW^J=\{w\in{}^IW\mid \ell(w)+1=\ell(ws_\alpha)\ \mbox{and}\ ws_\alpha\in {}^IW\ \mbox{for all}\ \alpha\in J\}.
\]
Every integral weight $\lambda\in\Lambda_I^+$ can be uniquely written in the form $\lambda=w\overline\lambda$ for some $w\in{}^IW^J$. Another parametrization is also widely used: denote $J'=-w_0J$ and $w'=w_Iww_{J}w_0$. Then $J'\subset\Delta$ and $w'\in{}^IW^{J'}$. So
\[
\lambda=w\overline\lambda=ww_J\overline\lambda=w_Iw_Iww_Jw_0w_0\overline\lambda
=w_Iw'\underline\lambda.
\]
Although the first parametrization is more convenient in this paper, we will always be aware of such differences in the cited results.


\subsection{Other conventions}
For any subsystem $\Phi'\subset\Phi$, denote by $W(\Phi')$ the subgroup of $W$ generated by reflections $s_\alpha$ with $\alpha\in\Phi'$. Then $W(\Phi)=W$ and $W(\Phi_I)=W_I$. We will frequently use the notation $M(\lambda, \Phi_I, \Phi)=M_I(\lambda)$ when we need to deal with generalized Verma modules associated with different root systems at the same time. We use similar convention for other notations, for example, $\Lambda^+(\Phi_I, \Phi)=\Lambda_I^+$, $\caO(\lambda, \Phi_I, \Phi)=\caO_{\lambda}^{\frp_I}$.

%
%
\section{Jantzen's simplicity criteria and reduction lemmas}
%
%
In this section, we will first recall Jantzen's simplicity criteria of generalized Verma modules and then give four reduction lemmas. Three of these reduction results are already known. The last reduction lemma, which we called the singular reduction, combined with the others, builds a reduction process on related root systems. With this, the simplicity problem of generalized Verma modules can be reduced to similar problem of some very special modules (so called the basic generalized Verma modules).

\subsection{Jantzen's simplicity criteria}
\begin{definition}\label{jansdef1}
For $\lambda\in\frh^*$, define
\[
\theta(\lambda)=\sum_{w\in W_I}(-1)^{\ell(w)}[M(w\lambda)].
\]
\end{definition}
One has $\theta(\lambda)=[M_I(\lambda)]$ for $\lambda\in\Lambda_I^+$ (see Proposition 9.6 in \cite{H3}).

\begin{prop}[\cite{J, M1, Ku}]\label{jansprop1}
Let $\lambda\in\frh^*$.
\begin{itemize}
\item [(1)] $\theta(w\lambda)=(-1)^{\ell(w)}\theta(\lambda)$ for $w\in W_I$.
\item [(2)] If $\langle \lambda,\alpha\rangle=0$ for some $\alpha\in\Phi_I$, then $\theta(\lambda)=0$.
\item [(3)] If $\langle \lambda,\alpha\rangle\in\bbZ\backslash\{0\}$ for all $\alpha\in\Phi_I$, there exists $w\in W_I$ so that $w\lambda\in\Lambda_I^+$ and $\theta(\lambda)=(-1)^{\ell(w)}[M_I(w\lambda)]$.
\end{itemize}
\end{prop}

Define the following sets of roots for $\lambda\in\frh^*$:
\[
\begin{aligned}
\Psi_\lambda^+:&=\{\beta\in\Phi^+\backslash\Phi_I\ |\ \langle\lambda, \beta^\vee\rangle\in\bbZ^{>0}\};\\
\Psi_\lambda^{++}:&=\{\beta\in\Psi_\lambda^+\ |\ \langle s_\beta\lambda,\alpha\rangle\neq0\ \mbox{for all}\ \alpha\in\Phi_I\}.
\end{aligned}
\]

\begin{theorem}[{\cite[Corollar 1]{J}}] \label{jansthm1}
Let $\lambda\in \Lambda_{I}^+$. Then $M_I(\lambda)$ is simple if and only if
\begin{equation}\label{janseq1}
\sum_{\beta\in\Psi_\lambda^+}\theta(s_\beta\lambda)=0.
\end{equation}
\end{theorem}

With Proposition \ref{jansprop1}, we can rewrite Theorem \ref{jansthm1}.

\begin{cor}\label{janscor1}
Let $\lambda\in \Lambda_{I}^+$. Then $M_I(\lambda)$ is simple if and only if
\begin{equation*}
\sum_{\beta\in\Psi_\lambda^{++}}\theta(s_\beta\lambda)=0.
\end{equation*}
\end{cor}

\subsection{Reduction lemmas} We will present four reduction lemmas in this subsection. Three of them can be found in \cite{J}. The last one will be proved in the next section. As mentioned in 2.4, we write $M_I(\lambda)=M(\lambda, \Phi_I, \Phi)$ when we need to deal with generalized Verma modules associated with different pairs $(\Phi_I, \Phi)$ in this paper. In a similar spirit, we write $W=W(\Phi)$ and $\caO_\lambda^{\frp_I}=\caO(\lambda, \Phi_I, \Phi)$ if needed. For $\lambda\in\frh^*$, define
\[
\begin{aligned}
\Phi_{[\lambda]}:=&\{\alpha\in\Phi\mid\langle\lambda, \alpha^\vee\rangle\in\bbZ\}.\\
W_{[\lambda]}:=&\{w\in W\mid w\lambda-\lambda\in\sum_{\alpha\in\Delta}\bbZ\alpha\}.
\end{aligned}
\]

\begin{theorem}[{\cite[Theorem 3.4]{H3}}]\label{redlthm1}
Let $\lambda\in\frh^*$. Then
\begin{itemize}
  \item [(1)] $\Phi_{[\lambda]}$ is a subsystem of $\Phi$.
  \item [(2)] $W_{[\lambda]}$ is the Weyl group of $\Phi_{[\lambda]}$, that is, $W_{[\lambda]}=W(\Phi_{[\lambda]})$.
\end{itemize}
\end{theorem}

If $\lambda\in\Lambda_I^+$, then $\Phi_I\subset\Phi_{[\lambda]}$ and $I\subset\Delta_{[\lambda]}$, where $\Delta_{[\lambda]}$ is the simple system corresponding to $\Phi_{[\lambda]}^+=\Phi_{[\lambda]}\cap\Phi^+$. If $\Phi'$ is a subsystem of $\Phi$, there exists a unique weight $\lambda|_{\Phi'}$ in the subspace $\bbC\Phi'$ so that
\begin{equation}\label{reseq1}
\langle\lambda|_{\Phi'}, \alpha^\vee\rangle=\langle\lambda, \alpha^\vee\rangle
\end{equation}
for all $\alpha\in\Phi'$.

\begin{lemma}[Integral reduction]\label{redllem1}
Let $\lambda\in\Lambda_I^+$. Then $M_I(\lambda)$ is simple if and only if  $M(\lambda|_{\Phi_{[\lambda]}}, \Phi_I, \Phi_{[\lambda]})$ is simple $($in the category $\caO(\lambda|_{\Phi_{[\lambda]}}, \Phi_I, \Phi_{[\lambda]}))$.
\end{lemma}

As a consequence of Theorem \ref{jansthm1}, this result was given in \cite{J} (see the remark after Corollar 4). It also can be deduced from the category equivalence obtained in \cite{So}. If $\beta\in\Phi$, let $\Phi_{\beta, 0}$ be the irreducible component of $\Phi$ with $\beta\in\Phi_{\beta, 0}$. Theorem \ref{jansthm1} also implies the following result (see the remark after Corollar 4 in \cite{J}).

\begin{lemma}[Irreducible reduction]\label{redllem2}
Let $\lambda\in\Lambda_I^+$. Then $M_I(\lambda)$ is simple if and only if  $M(\lambda|_{\Phi_{\beta,0}}, \Phi_I\cap\Phi_{\beta,0}, \Phi_{\beta,0})$ is simple $($in $\caO(\lambda|_{\Phi_{\beta,0}}, \Phi_I\cap\Phi_{\beta,0}, \Phi_{\beta,0})$$)$ for all $\beta\in\Psi_\lambda^+$.
\end{lemma}

Define the following two subsystems of $\Phi$:
\[
\begin{aligned}
\Phi_{\beta,1}:=&(\bbQ\Phi_I+\bbQ\beta)\cap\Phi;\\
\Phi_{\beta,2}:=&(\bbQ \Phi_\lambda+\bbQ\beta)\cap\Phi.
\end{aligned}
\]

The following result is a consequence of Satz 3 in \cite{J}.

\begin{lemma}[Parabolic reduction]\label{redllem3}
Let $\lambda\in\Lambda_I^+$. Then $M_I(\lambda)$ is simple if and only if $M(\lambda|_{\Phi_{\beta, 1}}, \Phi_I, \Phi_{\beta, 1})$ is simple $($in $\caO(\lambda|_{\Phi_{\beta, 1}}, \Phi_I, \Phi_{\beta, 1})$$)$ for all $\beta\in\Psi_\lambda^+$.
\end{lemma}

The last reduction lemma about singularity of $\lambda$ will be proved in the next section.

\begin{lemma}[Singular reduction]\label{redllem4}
Let $\lambda\in\Lambda_I^+$. Then $M_I(\lambda)$ is simple if and only if $M(\lambda|_{\Phi_{\beta, 2}}, \Phi_I\cap \Phi_{\beta, 2}, \Phi_{\beta, 2})$ is simple $($in $\caO(\lambda|_{\Phi_{\beta, 2}}, \Phi_I\cap \Phi_{\beta, 2}, \Phi_{\beta, 2})$$)$ for all $\beta\in\Psi_\lambda^+$.
\end{lemma}

\begin{remark}\label{redrmk1}
If $\beta\in\Phi\backslash\Phi_I$, it was pointed out in \cite{J} that there exists exactly one $\gamma\in\Phi_{\beta, 1}^+:=\Phi_{\beta, 1}\cap\Phi^+$, so that $I\cup\{\gamma\}$ is a simple system of $\Phi_{\beta, 1}$. Thus $\rank\Phi_{\beta,1}=\rank\Phi_I+1$. If $\lambda\in\Lambda_I^+$ is integral and $\beta\in\Phi\backslash\Phi_\lambda$, there exists $J\subset\Delta$ and $w\in{}^IW^J$ so that $\Phi_\lambda=w\Phi_J$. We can also find $\nu\in w^{-1}\Phi_{\beta, 2}\cap\Phi^+=(\bbQ \Phi_J+\bbQ w^{-1}\beta)\cap\Phi^+$ such that $J\cup\{\nu\}$ is simple system of $w^{-1}\Phi_{\beta, 2}$. Thus $\rank\Phi_{\beta,2}=\rank\Phi_J+1=\rank \Phi_\lambda+1$.
\end{remark}

\subsection{Reduction process} With the above reduction lemmas, we can build a reduction process to verify the simplicity of generalized Verma modules. Fix $\lambda\in\Lambda_I^+$. For $\beta\in\Psi_\lambda^+$, choose a chain of subsystems
\begin{equation*}
\Phi_{[\lambda]}=\Phi_0(\beta)\supset\Phi_1(\beta)\supset\Phi_2(\beta)\supset\cdots\supset\Phi_m(\beta)\supset\cdots
\end{equation*}
such that for $i\in\bbZ^{>0}$,
\begin{equation}\label{redseq1}
\Phi_{i+1}(\beta)
=\left\{\begin{aligned}
&(\Phi_{i}(\beta))_{\beta,0}\qquad\qquad\qquad\quad\ \mbox{if}\ i\equiv0\ (\mathrm{mod}4);\\
&(\Phi_{i}(\beta))_{\beta,1}\qquad\qquad\qquad\quad\ \mbox{if}\ i\equiv1\ (\mathrm{mod}4);\\
&(\Phi_{i}(\beta))_{\beta,0}\qquad\qquad\qquad\quad\ \mbox{if}\ i\equiv2\ (\mathrm{mod}4);\\
&(\Phi_{i}(\beta))_{\beta,2}\qquad\qquad\qquad\quad\ \mbox{if}\ i\equiv3\ (\mathrm{mod}4).
\end{aligned}
\right.
\end{equation}
If $\Phi_i(\beta)$ is not irreducible, then $\rank\ \Phi_{i+2}(\beta)<\rank\ \Phi_{i}(\beta)$. If $\rank(\Phi_I\cap\Phi_i(\beta))<\rank\ \Phi_i(\beta)-1$ or $\rank(\Phi_\lambda\cap\Phi_i(\beta))<\rank\ \Phi_i(\beta)-1$, then $\rank\ \Phi_{i+4}(\beta)<\rank\ \Phi_{i}(\beta)$, we can eventually get $\Phi_k(\beta)=\Phi_{k+1}(\beta)=\cdots$ for some $k\in\bbZ^{>0}$. Denote $\Phi(\beta)=\Phi_k(\beta)$. The following result is evident.

\begin{lemma}\label{redslem1}
Let $\lambda\in\Lambda_I^+$ and $\beta\in\Psi_\lambda^+$. Then $\Phi(\beta)$ is irreducible and $\lambda|_{\Phi(\beta)}$ is an integral weight on $\Phi(\beta)$. Moreover,
\[
\rank(\Phi_I\cap\Phi(\beta))=\rank(\Phi_\lambda\cap\Phi(\beta))=\rank\ \Phi(\beta)-1.
\]
\end{lemma}

\begin{remark}
One might ask whether we can get a different $\Phi(\beta)$ if we choose another process of reduction. We leave it to the reader since we do not need this result in the present paper.
\end{remark}

\begin{definition}\label{basdef}
Let $\Phi$ be an irreducible root system and $\lambda\in\Lambda_I^+$ be integral. We say $M_I(\lambda)$ is a \emph{basic generalized Verma module} if $\rank\ \Phi_I=\rank\ \Phi_\lambda=\rank\ \Phi-1$. The weight $\lambda$ is called a \emph{basic weight} associated with $(\Phi_I, \Phi)$.
\end{definition}

Lemma \ref{redslem1} shows that $M(\lambda|_{\Phi(\beta)}, \Phi_I\cap \Phi(\beta), \Phi(\beta))$ is a basic generalized Verma module and $\lambda|_{\Phi(\beta)}$ is a basic weight associated with $(\Phi_I\cap \Phi(\beta), \Phi(\beta))$. Moreover, we have the following simplicity criteria of generalized Verma modules.

\begin{theorem}\label{redsthm1}
Let $\lambda\in\Lambda_I^+$. The following three conditions are equivalent:
\begin{itemize}
\item [(\rmnum{1})] $M_I(\lambda)$ is simple;
\item [(\rmnum{2})] $M(\lambda|_{\Phi(\beta)}, \Phi_I\cap \Phi(\beta), \Phi(\beta))$ is simple for all $\beta\in\Psi_\lambda^+$;
\item [(\rmnum{3})] $M(\lambda|_{\Phi(\beta)}, \Phi_I\cap \Phi(\beta), \Phi(\beta))$ is simple for all $\beta\in\Psi_\lambda^{++}$.
\end{itemize}
\end{theorem}

\begin{remark}
The above theorem makes the classification of basic generalized Verma modules necessary for our argument. This will be done in section 5.
\end{remark}

We need several lemmas to prove Theorem \ref{redsthm1}.

\begin{lemma}\label{redslem2}
Let $\lambda\in\Lambda_I^+$ and $\beta, \gamma\in\Psi_\lambda^+$. if $\gamma\in\Phi_{\beta, i}$ for some $i\in\{0, 1, 2\}$, then $\Phi_{\beta, i}=\Phi_{\gamma, i}$.
\end{lemma}
\begin{proof}
If $i=0$, $\Phi_{\beta, 0}$ is the irreducible component of $\Phi$ containing $\beta$. So $\gamma\in\Phi_{\beta, 0}$ implies $\beta$ and $\gamma$ are in the same irreducible component. One has $\Phi_{\gamma, 0}=\Phi_{\beta, 0}$. Now consider $i=1$ and $\gamma\in\Phi_{\beta, 1}=(\bbQ\beta+\bbQ\Phi_I)\cap\Phi$. With $\gamma\in\Psi_\lambda^+$, we can find $0\neq c\in\bbQ$ so that $\gamma\in c\beta+\bbQ\Phi_I$. Thus
\[
\Phi_{\gamma, 1}=(\bbQ\gamma+\bbQ\Phi_I)\cap\Phi=(\bbQ\beta+\bbQ\Phi_I)\cap\Phi=\Phi_{\beta, 1}.
\]
The proof is similar for $i=2$.
\end{proof}

\begin{lemma}\label{redslem3}
Let $\lambda\in\Lambda_I^+$ and $\beta, \gamma\in\Psi_\lambda^+$. If $\gamma\in\Phi_k(\beta)$ for some $k\geq0$, then $\Phi_k(\gamma)=\Phi_k(\beta)$. In particular, if $\gamma\in\Phi(\beta)$, then $\Phi(\gamma)=\Phi(\beta)$.
\end{lemma}
\begin{proof}
We prove the lemma by induction on $k$. The case $k=0$ is evident. Assume that $\Phi_{k-1}(\gamma)=\Phi_{k-1}(\beta)$ for any $\gamma\in\Psi_\lambda^+\cap\Phi_{k-1}(\beta)$. Now suppose $\gamma\in\Psi_\lambda^+\cap\Phi_{k}(\beta)$. Evidently, $\gamma\in\Psi_\lambda^+\cap\Phi_{k-1}(\beta)$. The induction hypothesis implies $\Phi_{k-1}(\gamma)=\Phi_{k-1}(\beta)$. With $\gamma\in\Phi_k(\beta)=(\Phi_{k-1}(\beta))_{\beta,i}$ for some $i\in\{0, 1, 2\}$. It follows from Lemma \ref{redslem2} that $(\Phi_{k-1}(\beta))_{\gamma,i}=(\Phi_{k-1}(\beta))_{\beta,i}$. Hence
\[
\Phi_k(\gamma)=(\Phi_{k-1}(\gamma))_{\gamma,i}=(\Phi_{k-1}(\beta))_{\gamma,i}=(\Phi_{k-1}(\beta))_{\beta,i}
=\Phi_k(\beta).
\]

The second assertion is an immediate consequence of the first one.
\end{proof}

\begin{lemma}\label{redslem4}
Let $\lambda\in\Lambda_I^+$, $\beta\in\Phi\backslash\Phi_\lambda$ and $\gamma\in \Phi_{\beta, 2}$. Then $s_\gamma\lambda$ is $\Phi_I$-regular if and only if $s_\gamma(\lambda|_{\Phi_{\beta, 2}})$ is $\Phi_I\cap \Phi_{\beta, 2}$-regular.
\end{lemma}
\begin{proof}
Since $s_\gamma\alpha\in \Phi_{\beta, 2}$ for $\alpha\in\Phi_{\beta, 2}$, one has
\begin{equation}\label{redsl4eq1}
\langle s_\gamma(\lambda|_{\Phi_{\beta, 2}}), \alpha\rangle=\langle\lambda|_{\Phi_{\beta, 2}}, s_\gamma\alpha\rangle=\langle\lambda, s_\gamma\alpha\rangle=\langle s_\gamma\lambda,\alpha\rangle.
\end{equation}
If $s_\gamma\lambda$ is $\Phi_I$-regular, then (\ref{redsl4eq1}) implies $s_\gamma(\lambda|_{\Phi_{\beta, 2}})$ is $\Phi_I\cap \Phi_{\beta, 2}$-regular. Conversely, assume that $s_\gamma(\lambda|_{\Phi_{\beta, 2}})$ is $\Phi_I\cap \Phi_{\beta, 2}$-regular. Then $\langle s_\gamma\lambda, \alpha\rangle\neq0$ for all $\alpha\in\Phi_I\cap \Phi_{\beta, 2}$ by (\ref{redsl4eq1}). If $s_\gamma\lambda$ is not $\Phi_I$-regular, there is $\alpha\in\Phi_I\backslash \Phi_{\beta, 2}$ such that $\langle s_\gamma\lambda,\alpha\rangle=0$. One has $\langle\lambda, s_\gamma\alpha\rangle=0$ and thus $s_\gamma\alpha\in \Phi_\lambda$. Therefore $\alpha\in\bbQ \Phi_\lambda+\bbQ\gamma$. With $\gamma\in \Phi_{\beta, 2}=(\bbQ\Phi_\lambda+\bbQ\beta)\cap\Phi$, we obtain $\alpha\in \Phi_{\beta, 2}$, a contradiction.
\end{proof}

Write $\Psi^{+}(\lambda, \Phi_I, \Phi)=\Psi_\lambda^{+}$ and $\Psi^{++}(\lambda, \Phi_I, \Phi)=\Psi_\lambda^{++}$ as in 2.4.

\begin{lemma}\label{redslem5}
Let $\lambda\in\Lambda_I^+$.
\begin{itemize}
  \item [(1)] If $\Phi'$ is any subsystem of $\Phi$, then
  \[
  \Psi_\lambda^{+}\cap\Phi'=\Psi^{+}(\lambda|_{\Phi'}, \Phi_I\cap\Phi', \Phi').
  \]
  \item [(2)] If $\beta\in\Phi\backslash\Phi_I$, then
  \[
  \Psi_\lambda^{++}\cap\Phi_{\beta, 1}=\Psi^{++}(\lambda|_{\Phi_{\beta, 1}}, \Phi_I, \Phi_{\beta, 1}).
  \]
  \item [(3)] If $\beta\in\Phi\backslash\Phi_\lambda$, then
  \[
  \Psi_\lambda^{++}\cap\Phi_{\beta, 2}=\Psi^{++}(\lambda|_{\Phi_{\beta, 2}}, \Phi_I\cap\Phi_{\beta, 2}, \Phi_{\beta, 2}).
  \]
\end{itemize}
\end{lemma}
\begin{proof}
The assertion (1) is an easy consequence of the definition, while (2) follows from $\Phi_I\subset\Phi_{\beta, 1}$ and (1). With Lemma \ref{redslem4} and (1), we can get (3).
\end{proof}

\textbf{Proof of Theorem \ref{redsthm1}}
First we prove the equivalence between (\rmnum{1}) and (\rmnum{2}). It suffices to show that the simplicity of $M_I(\lambda)$ is equivalent to the simplicity of $M(\lambda|_{\Phi_k(\beta)}, \Phi_I\cap\Phi_k(\beta), \Phi_k(\beta))$ for $k\geq0$ and $\beta\in\Psi_\lambda^+$. We use induction on $k$. The case $k=0$ follows from Lemma \ref{redllem1} (the integral reduction lemma). Suppose this is true for $k\geq0$. The induction hypothesis implies that $M_I(\lambda)$ is simple if and only if $M(\lambda|_{\Phi_{k}(\beta)}, \Phi_I\cap\Phi_{k}(\beta), \Phi_{k}(\beta))$ is simple for all $\beta\in\Psi_\lambda^+$. If $k\equiv0\ (\mathrm{mod}4)$, Lemma \ref{redllem2} implies that $M(\lambda|_{\Phi_{k}(\beta)}, \Phi_I\cap\Phi_{k}(\beta), \Phi_{k}(\beta))$ is simple if and only if $M(\lambda|_{(\Phi_{k}(\beta))_{\gamma, 0}}, \Phi_I\cap(\Phi_{k}(\beta))_{\gamma, 0}, (\Phi_{k}(\beta))_{\gamma, 0})$ is simple for any $\gamma\in\Psi_\lambda^+\cap\Phi_{k}(\beta)$. Here
\[
\Psi_\lambda^+\cap\Phi_{k}(\beta)=\Psi^+(\lambda|_{\Phi_k(\beta)}, \Phi_I\cap\Phi_k(\beta), \Phi_k(\beta))
\]
as a consequence of Lemma \ref{redslem5}. In view of Lemma \ref{redslem3},
\[
(\Phi_{k}(\beta))_{\gamma, 0}=(\Phi_{k}(\gamma))_{\gamma, 0}=\Phi_{k+1}(\gamma).
\]
We can prove similar results for $k\equiv1, 2, 3\ (\mathrm{mod}4)$. Therefore $M_I(\lambda)$ is simple if and only if $M(\lambda|_{\Phi_{k+1}(\gamma)}, \Phi_I\cap\Phi_{k+1}(\gamma), \Phi_{k+1}(\gamma))$ is simple for all $\gamma\in\Psi_\lambda^+$.

Obviously (\rmnum{2}) implies (\rmnum{3}). For the converse, consider the root $\beta\in\Psi_\lambda^+\backslash\Psi_\lambda^{++}$. If $\Psi^{++}(\lambda|_{\Phi(\beta)}, \Phi_I\cap\Phi(\beta), \Phi(\beta))=\emptyset$, then $M(\lambda|_{\Phi(\beta)}, \Phi_I\cap \Phi(\beta), \Phi(\beta))$ is simple by Corollary \ref{janscor1}. Otherwise we can choose $\gamma\in\Psi^{++}(\lambda|_{\Phi(\beta)}, \Phi_I\cap\Phi(\beta), \Phi(\beta))\subset\Phi(\beta)$. Lemma \ref{redslem5} implies $\gamma\in\Psi_\lambda^{++}\subset\Psi_\lambda^+$, while Lemma \ref{redslem3} yields $\Phi(\gamma)=\Phi(\beta)$. This forces $M(\lambda|_{\Phi(\beta)}, \Phi_I\cap \Phi(\beta), \Phi(\beta))$ to be simple since $M(\lambda|_{\Phi(\gamma)}, \Phi_I\cap \Phi(\gamma), \Phi(\gamma))$ is simple.

%
%

%
%
\section{Jantzen coefficients}
%
%

In this section, we first introduce the Jantzen coefficients for generalized Verma mdoules. Then we give several invariant properties of Jantzen coefficients. One of them can be used to prove the singular reduction lemma in the previous section.

\subsection{Jantzen coefficient and Sign function} For $\lambda, \mu\in\frh^*$, we write $\mu\leq\lambda$ if $\Hom_\caO(M(\mu), M(\lambda))\neq0$. This gives a partial ordering on $\frh^*$ which can be viewed as the Bruhat ordering on weights (\cite[\S2]{ES}). If $\lambda,\mu\in\Lambda_I^+$ and $\mu=ws_\beta\lambda$ with $w\in W_I$ and $\beta\in\Psi_\lambda^+$, then $w\beta\in\Psi_{w\lambda}^+$. The BGG Theorem \cite{BGG} implies $\mu=s_{w\beta}w\lambda<w\lambda<\lambda$.

\begin{definition}\label{sjdef1}
Let $\lambda\in\Lambda_I^+$. Write (see Proposition \ref{jansprop1})
\begin{equation}\label{sjeq1}
\begin{aligned}
\sum_{\beta\in\Psi_\lambda^+}\theta(s_\beta\lambda)=\sum_{\beta\in\Psi_\lambda^{++}}\theta(s_\beta\lambda)=\sum_{\lambda>\mu\in\Lambda_I^+}c(\lambda, \mu)[M_I(\mu)],
\end{aligned}
\end{equation}
where $c(\lambda, \mu)\in\bbZ$ is called the \emph{Jantzen coefficient} associated with $(\lambda, \mu)$. For convenience, we set $c(\lambda, \mu)=c(\mu, \lambda)$ when $\lambda<\mu$ and $c(\lambda, \mu)=0$ when $\lambda\nless\mu$ and $\lambda\ngtr\mu$.
\end{definition}

\begin{example}\label{sjex1}
Let $\Phi=A_2$ and $I=\{e_1-e_2\}$. Set $\lambda=(1, 0, -1)$, $\mu=(1, -1, 0)$ and $\nu=(0, -1, 1)$. Then $\Psi_\lambda^+=\{e_2-e_3, e_1-e_3\}$. So
\[
\sum_{\beta\in\Psi_\lambda^+}\theta(s_\beta\lambda)=\theta(\mu)+\theta(s_{e_1-e_2}\nu)=[M_I(\mu)]-[M_I(\nu)].
\]
Therefore $c(\lambda, \mu)=1$ and $c(\lambda, \nu)=-1$. Similarly, $c(\mu, \nu)=1$, while the other Jantzen coefficients are vanished.
\end{example}

\begin{remark}\label{sjrmk1}
Since $\{[M_I(\mu)]\mid \mu\in\Lambda_I^+\}$ is a basis of the Grothendieck group $K(\caO^\frp)$ of $\caO^\frp$, (\ref{sjeq1}) shows that Jantzen coefficients actually determine the simplicity of generalized Verma modules. In other words, a generalized Verma module is simple if and only if all its Jantzen coefficients are vanished.
\end{remark}

Define a symmetric binary \emph{sign function} on the set of $\Phi_I$-integral weights as follows:
\begin{definition}\label{sjdef2}
Let $\lambda, \mu$ be $\Phi_I$-integral weights. Then
\[
\sgn(\lambda, \mu)
=\left\{\begin{aligned}
&(-1)^{\ell(w)}\qquad \mbox{if}\ \lambda, \mu\ \mbox{are}\ \Phi_I\mbox{-regular and}\ \mu=w\lambda;\\
&0\qquad\qquad\quad\ \mbox{otherwise}.
\end{aligned}
\right.
\]
\end{definition}
The definition implies that if $\sgn(\lambda, \mu)\neq0$, then
\begin{equation}\label{sjeq2}
\sgn(\lambda, \nu)=\sgn(\lambda, \mu)\sgn(\mu, \nu)
\end{equation}
for any $\Phi_I$-integral weight $\nu$. When $\lambda, \mu\in\Lambda_I^+$, set
\begin{equation*}\label{sjeq3}
\Psi_{\lambda, \mu}^+:=\{\beta\in\Psi_\lambda^+\mid \mu=w_\beta s_\beta\lambda\ \mbox{for some}\ w_\beta\in W_I\}.
\end{equation*}

\begin{lemma}\label{sjlem1}
Let $\lambda, \mu\in\Lambda_I^+$ with $\lambda>\mu$. Then
\[
c(\lambda, \mu)=\sum_{\beta\in \Psi_{\lambda, \mu}^+}(-1)^{\ell(w_\beta)}=\sum_{\beta\in\Psi_\lambda^{++}}\sgn(s_\beta\lambda, \mu).
\]
\end{lemma}
\begin{proof}
The first equation is an immediate consequence of Proposition \ref{jansprop1}, while the second one follows from
\begin{equation*}
\begin{aligned}
\sum_{\beta\in\Psi_\lambda^{++}}\theta(s_\beta\lambda)=&\sum_{\beta\in\Psi_\lambda^{++}}\sum_{\lambda>\mu\in\Lambda_I^+}
\sgn(s_\beta\lambda, \mu)[M_I(\mu)]\\
=&\sum_{\lambda>\mu\in\Lambda_I^+}\sum_{\beta\in\Psi_\lambda^{++}}\sgn(s_\beta\lambda, \mu)[M_I(\mu)]\\
=&\sum_{\lambda>\mu\in\Lambda_I^+}c(\lambda, \mu)[M_I(\mu)].
\end{aligned}
\end{equation*}
\end{proof}

The following lemma is evident.

\begin{lemma}\label{sjlem2}
Let $\lambda, \mu\in\Lambda_I^+$. Then
\begin{itemize}
  \item [(1)] $c(\lambda, \mu)=0$ unless $\mu\in W\lambda$.
  \item [(2)] $c(k\lambda, k\mu)=c(\lambda, \mu)$ for $k\in\bbZ^{>0}$.
\end{itemize}
\end{lemma}

\subsection{Invariant properties of Jantzen coefficients} In this subsection, we will give four invariant properties of Jantzen coefficients. Each reduction lemma in the previous section corresponds to an invariant property of Jantzen coefficients. The singular reduction (Lemma \ref{redllem4}) will be proved at the end of this subsection. As in 2.4, we write $\Lambda_I^+=\Lambda^+(\Phi_I, \Phi)$, $\sgn(\lambda, \mu)=\sgn(\lambda, \mu, \Phi_I, \Phi)$ and $c(\lambda, \mu)=c(\lambda, \mu, \Phi_I, \Phi)$.

\begin{lemma}[Integral invariance]\label{invlem1}
Let $\lambda\in\Lambda_I^+$ and $\beta\in\Psi_\lambda^{+}$. Then $\sgn(s_\beta\lambda, \mu)\neq0$ for some $\mu\in\Lambda_I^+$ if and only if $\sgn(s_\beta\lambda|_{\Phi_{[\lambda]}}, \mu', \Phi_I, \Phi_{[\lambda]})\neq0$ for some $\mu'\in\Lambda^+(\Phi_I, \Phi_{[\lambda]})$. In particular, if $\sgn(s_\beta\lambda, \mu)\neq0$, then $\mu'=\mu|_{\Phi_{[\lambda]}}$ and
\[
c(\lambda, \mu)=c(\lambda|_{\Phi_{[\lambda]}}, \mu|_{\Phi_{[\lambda]}}, \Phi_I, \Phi_{[\lambda]}).
\]
\end{lemma}

\begin{lemma}[Irreducible invariance]\label{invlem2}
Let $\lambda\in\Lambda_I^+$ and $\beta\in\Psi_\lambda^{+}$. Then $\sgn(s_\beta\lambda, \mu)\neq0$ for some $\mu\in\Lambda_I^+$ if and only if $\sgn(s_\beta\lambda|_{\Phi_{\beta, 0}}, \mu', \Phi_I\cap\Phi_{\beta, 0}, \Phi_{\beta, 0})\neq0$ for some $\mu'\in\Lambda^+(\Phi_I\cap\Phi_{\beta, 0}, \Phi_{\beta, 0})$. In particular, if $\sgn(s_\beta\lambda, \mu)\neq0$, then $\mu'=\mu|_{\Phi_{\beta, 0}}$ and
\[
c(\lambda, \mu)=c(\lambda|_{\Phi_{\beta, 0}}, \mu|_{\Phi_{\beta, 0}}, \Phi_I\cap\Phi_{\beta, 0}, \Phi_{\beta, 0}).
\]
\end{lemma}

\begin{lemma}[Parabolic invariance]\label{invlem3}
Let $\lambda, \mu\in\Lambda_I^+$ and $\beta\in\Psi_\lambda^{+}$.  Then $\sgn(s_\beta\lambda, \mu)\neq0$ for some $\mu\in\Lambda_I^+$ if and only if $\sgn(s_\beta\lambda|_{\Phi_{\beta, 1}}, \mu', \Phi_I, \Phi_{\beta, 1})\neq0$ for some $\mu'\in\Lambda^+(\Phi_I, \Phi_{\beta, 1})$. In particular, if $\sgn(s_\beta\lambda, \mu)\neq0$, then $\mu'=\mu|_{\Phi_{\beta, 1}}$ and
\[
c(\lambda, \mu)=c(\lambda|_{\Phi_{\beta, 1}}, \mu|_{\Phi_{\beta, 1}}, \Phi_I, \Phi_{\beta, 1}).
\]
\end{lemma}

\begin{lemma}[Singular invariance]\label{invlem4}
Let $\lambda\in\Lambda_I^+$ and $\beta\in\Psi_\lambda^{+}$. Then $\sgn(s_\beta\lambda, \mu)\neq0$ for some $\mu\in\Lambda_I^+$, if and only if $\sgn((s_\beta\lambda)|_{\Phi_{\beta, 2}}, \mu', \Phi_I\cap\Phi_{\beta, 2}, \Phi_{\beta, 2})\neq0$ for some $\mu'\in\Lambda^+(\Phi_I\cap\Phi_{\beta, 2}, \Phi_{\beta, 2})$. In particular, if $\sgn(s_\beta\lambda, \mu)\neq0$, then
\[
c(\lambda, \mu)=(-1)^{\ell(w)+\ell(w')}c(\lambda|_{\Phi_{\beta, 2}}, \mu', \Phi_I\cap\Phi_{\beta, 2}, \Phi_{\beta, 2}),
\]
where $w\in W_I$ with $\mu=ws_\beta\lambda$ and $w'\in W(\Phi_I\cap\Phi_{\beta, 2})$ with $\mu'=w's_\beta(\lambda|_{\Phi_{\beta, 2}})$.
\end{lemma}

Given $\beta\in\Psi_\lambda^{++}$, note that we can always find $w\in W_I$ such that $\mu=ws_\beta\lambda\in\Lambda_I^+$ and $\sgn(s_\beta\lambda, \mu)=(-1)^{\ell(w)}\neq0$. The following result is an immediate consequence of the reduction process and the above lemmas.

\begin{lemma}\label{invjlem}
Let $\lambda\in\Lambda_I^+$ and $\beta\in\Psi_\lambda^{++}$. Denote $\lambda'=\lambda|_{\Phi(\beta)}$. Suppose that $\mu=ws_\beta\lambda\in\Lambda_I^+$ for $w\in W_I$. Then there exists $\mu'\in\Lambda^+(\Phi_I\cap\Phi(\beta), \Phi(\beta))$ with $\mu'=w's_\beta\lambda'$ for $w'\in W(\Phi_I\cap\Phi(\beta))$ such that
\[
|c(\lambda, \mu)|=|c(\lambda', \mu', \Phi_I\cap\Phi(\beta), \Phi(\beta))|.
\]
\end{lemma}

A considerable amount of effort is needed to prove Lemma \ref{invlem4}. This will be done in the next subsection. The proofs of the other three lemmas (Lemma \ref{invlem1}-\ref{invlem3}), which we leave to the reader, are similar and easier.

\begin{remark}\label{invrmk1}
With the above four lemmas and the classification of basic generalized Verma modules in the next section, One can deduce that Jantzen coefficients must be $\pm 1$ or $0$ except a few very special cases. This will be exhibited later.
\end{remark}

\textbf{Proof of Lemma \ref{redllem4} (Singular reduction)} If $M_I(\lambda)$ is not simple, Theorem \ref{jansthm1} and (\ref{sjeq1}) implies $c(\lambda, \mu)\neq0$ for some $\mu\in\Lambda_I^+$. Lemma \ref{sjlem1} yields $\beta\in\Psi_\lambda^{++}$ so that $\sgn(s_\beta\lambda, \mu)\neq0$. In view of Lemma \ref{invlem4}, $c(\lambda|_{\Phi_{\beta, 2}}, \mu', \Phi_I\cap\Phi_{\beta, 2}, \Phi_{\beta, 2})\neq0$ for some $\mu'\in\Lambda^+(\Phi_I\cap\Phi_{\beta, 2}, \Phi_{\beta, 2})$ and thus $M(\lambda|_{\Phi_{\beta, 2}}, \Phi_I\cap\Phi_{\beta, 2}, \Phi_{\beta, 2})$ is not simple.
Conversely, if $M(\lambda|_{\Phi_{\beta, 2}}, \Phi_I\cap\Phi_{\beta, 2}, \Phi_{\beta, 2})$ is not simple for some $\beta\in\Psi_\lambda^+$, there exists $\mu'\in\Lambda^+(\Phi_I\cap\Phi_{\beta, 2}, \Phi_{\beta, 2})$ such that $c(\lambda|_{\Phi_{\beta, 2}}, \mu', \Phi_I\cap\Phi_{\beta, 2}, \Phi_{\beta, 2})\neq0$. Therefore $\sgn(s_\gamma(\lambda|_{\Phi_{\beta, 2}}), \mu', \Phi_I\cap\Phi_{\beta, 2}, \Phi_{\beta, 2})\neq0$ for some $\gamma\in\Psi^+(\lambda|_{\Phi_{\beta, 2}}, \Phi_I\cap\Phi_{\beta, 2}, \Phi_{\beta, 2})=\Psi_\lambda^+\cap\Phi_{\beta, 2}$, keeping in mind Lemma \ref{redslem5}. We get $\Phi_{\beta, 2}=\Phi_{\gamma, 2}$ in view of Lemma \ref{redslem2}. Now Lemma \ref{invlem4} implies $\sgn(s_\gamma\lambda, \mu)\neq0$ and $c(\lambda, \mu)\neq0$ for some $\mu\in\Lambda_I^+$. Hence $M_I(\lambda)$ is not simple.

\subsection{Restriction of weights and singular invariance} Note that the reflection $s_\nu$ can be defined for any $\nu\in\frh^*$. In fact,
\[
s_\nu\lambda=\lambda-\frac{2\langle\lambda, \nu\rangle}{\langle\nu, \nu\rangle}\nu
\]
for $\lambda\in\frh^*$. Let $\Phi'$ be a subsystem of $\Phi$. For convenience, write $\lambda\perp\Phi'$ if $\lambda|_{\Phi'}=0$. Thus one obtains $(\lambda-\lambda|_{\Phi'})\perp\Phi'$ for any $\lambda\in\frh^*$.

\begin{lemma}\label{reslem1}
Let $\Phi'$ be a subsystem of $\Phi$. Choose $\nu\in\frh^*$ $($not necessarily a root$)$.
\begin{itemize}
  \item [(1)] If $\nu\perp\Phi'$, then $(s_\nu\lambda)|_{\Phi'}=\lambda|_{\Phi'}$ for any $\lambda\in\frh^*$.
  \item [(2)] If $\nu\in\bbC\Phi'$, then $(s_\nu\lambda)|_{\Phi'}=s_\nu(\lambda|_{\Phi'})$ for any $\lambda\in\frh^*$.
  \item [(3)] Denote $\lambda'=\lambda|_{\Phi'}$. Then $\Phi'_{\lambda'}=\Phi_\lambda\cap\Phi'$.
\end{itemize}
\end{lemma}
\begin{proof}
They are easy consequences of the definition.
\end{proof}

The following result can be used to prove Lemma \ref{invlem3}.
\begin{lemma}\label{reslem2}
Let $\lambda\in\Lambda_I^+$ and $\beta\in\Phi\backslash\Phi_I$. If $\gamma, \nu\in \Phi_{\beta, 1}$, then
\begin{equation}\label{resl2eq1}
\sgn(s_\gamma\lambda, s_\nu\lambda)=\sgn(s_\gamma(\lambda|_{\Phi_{\beta, 1}}), s_\nu(\lambda|_{\Phi_{\beta, 1}}), \Phi_I, \Phi_{\beta, 1}).
\end{equation}
\end{lemma}
\begin{proof}
If $\sgn(s_\gamma\lambda, s_\nu\lambda)\neq0$, then $s_\gamma\lambda$ and $s_\nu\lambda$ are $\Phi_I$-regular. Moreover, there exists $w\in W(\Phi_I)=W_I$ such that $ws_\gamma\lambda=s_\nu\lambda$. Let $\lambda'=\lambda|_{\Phi_{\beta, 1}}$. Lemma \ref{reslem1} yields
\[
ws_\gamma\lambda'=(ws_\gamma\lambda)|_{\Phi_{\beta, 1}}=(s_\nu\lambda)|_{\Phi_{\beta, 1}}=s_\nu\lambda'.
\]
Since $s_\gamma\lambda', s_\nu\lambda'$ are also $\Phi_I$-regular, it follows that
\[
\sgn(s_\gamma\lambda, s_\nu\lambda)=(-1)^{\ell(w)}=\sgn(s_\gamma\lambda', s_\nu\lambda', \Phi_I, \Phi_{\beta, 1}).
\]
Conversely, if $\sgn(s_\gamma\lambda', s_\nu\lambda', \Phi_I, \Phi_{\beta, 1})\neq0$, there exists $w\in W_I$ such that $ws_\gamma\lambda'=s_\nu\lambda'$. Lemma \ref{reslem1} yields
\[
ws_\gamma\lambda=ws_\gamma(\lambda-\lambda'+\lambda')=\lambda-\lambda'+ws_\gamma\lambda'=\lambda-\lambda'+s_\nu\lambda'=s_\nu\lambda.
\]
Obviously $s_\gamma\lambda, s_\nu\lambda$ are $\Phi_I$-regular since $s_\gamma\lambda', s_\nu\lambda'$ are $\Phi_I$-regular in this case. Then both sides of (\ref{resl2eq1}) are equal to $(-1)^{\ell(w)}$.
\end{proof}

There are similar results for $\Phi_{[\lambda]}$ and $\Phi_{\beta, 0}$, with easier proofs. Our plan is to prove a similar lemma for $\Phi_{\beta, 2}$. This turns out to be much more complicated than its counterparts. We need several related results.

\begin{lemma}\label{weyllem1}
Let $\lambda\in\frh^*$. Then $W(\Phi_\lambda)$ is the isotropy group of $\lambda$.
\end{lemma}
\begin{proof}
If $\lambda\in\bbR\Phi$, this is Theorem 1.12(c) in \cite{H2}. The general case for $\lambda\in\frh^*=\bbC\Phi$ is an easy consequence. In fact, denote $\Phi'=\Phi_{[\lambda]}$ and $\lambda':=\lambda|_{\Phi'}$. Evidently $\Phi_\lambda\subset\Phi'$. Lemma \ref{reslem1} gives $\Phi'_{\lambda'}=\Phi_\lambda\cap\Phi'=\Phi_\lambda$. Since $\lambda'$ is integral on $\Phi'$, \cite[Theorem 1.12]{H2} shows that the isotropy group of $\lambda'$ under the action of $W'=W_{[\lambda]}\subset W$ is $W'(\Phi'_{\lambda'})=W'(\Phi_\lambda)=W(\Phi_\lambda)$.

If $w\in W(\Phi_\lambda)$, obviously we have $w\lambda=\lambda$. Conversely, if $w\lambda=\lambda$, Theorem \ref{redlthm1} yields $w\in W(\Phi')=W'$. Applying Lemma \ref{reslem1}, one obtains $w\lambda'=\lambda'$. In other words, $w\in W'(\Phi'_{\lambda'})=W(\Phi_\lambda)$.
\end{proof}

\begin{lemma}\label{weyllem2}
Choose $I\subset\Delta$. Let $\Phi'$ be a subsystem of $\Phi$. Then $W_I\cap W(\Phi')=W(\Phi_I\cap\Phi')$.
\end{lemma}
\begin{proof}
One inclusion $W(\Phi_I\cap\Phi')\subset W_I\cap W(\Phi')$ is obvious.

Choose an integral weight $\lambda\in\frh^*$ so that $\langle\lambda, \alpha\rangle=0$ for $\alpha\in I$ and $\langle\lambda, \alpha\rangle>0$ for $\alpha\in\Delta\backslash I$. Then $\Phi_\lambda=\Phi_I$. Given $w\in W_I\cap W(\Phi')$, one has $w\lambda=\lambda$. Denote $\lambda'=\lambda|_{\Phi'}$. It follows from Lemma \ref{reslem1} that $\lambda'=w\lambda'$ and $\Phi'_{\lambda'}=\Phi_I\cap\Phi'$. \cite[Theorem 1.12]{H2} implies $W'(\Phi'_{\lambda'})$ is the isotropic group of $\lambda'$ under the action of $W'=W(\Phi')$. This means $w\in W'(\Phi'_{\lambda'})=W'(\Phi_I\cap\Phi')=W(\Phi_I\cap\Phi')$.
\end{proof}

\begin{lemma}\label{reslem3}
Let $\lambda\in\Lambda_I^+$. Give $\beta, \gamma\in \Psi_\lambda^+$ so that $\sgn(s_\beta\lambda, s_\gamma\lambda)\neq0$. Then $\gamma\in \Phi_{\beta, i}$ for $i=1, 2$.
\end{lemma}
\begin{proof}
We can assume that $ws_\beta\lambda=s_\gamma\lambda$ for some $w\in W_I$. Then $ws_\beta\lambda=\lambda-\langle\lambda, \gamma^\vee\rangle\gamma$, where $\langle\lambda, \gamma^\vee\rangle\in\bbZ^{>0}$ and $s_\gamma ws_\beta\lambda=\lambda$.

For $i=1$, since $\lambda-ws_\beta\lambda\in\bbQ\Phi_I+\bbQ\beta$, one has $\gamma\in\Phi_{\beta, 1}$. For $i=2$, Lemma \ref{weyllem1} implies $s_\gamma ws_\beta\in W(\Phi_\lambda)\subset W(\Phi_{\beta, 2})$. Denote $w_1=s_\gamma ws_\beta$. Choose a basis $\{\beta_1, \cdots, \beta_k\}\subset \Phi_{\beta, 2}$ of $\bbR\Phi_{\beta, 2}$, where $k=\rank\ \Phi_{\beta, 2}$. Then extend to a basis $\{\beta_1, \cdots, \beta_n\}\subset \Phi$ of $\bbR\Phi$. We can write $\gamma=\sum_{i=1}^nc_i\beta_i$ with $c_i\in\bbQ$. If $\gamma\not\in \Phi_{\beta, 2}$, there is $k+1\leq i\leq n$ such that $c_i\neq0$. Thus $\gamma\not\in\bbR\Phi_{\beta, 2}$. There exists $\mu\in\bbR\Phi$ so that $\langle\mu, \alpha\rangle=0$ for any $\alpha\in \Phi_{\beta, 2}$ and $\langle\mu, \gamma^\vee\rangle=1$. We get
\[
w\mu=s_\gamma w_1 s_\beta\mu=s_\gamma w_1\mu=s_\gamma \mu=\mu-\gamma,
\]
that is, $\gamma=\mu-w\mu\in\bbR\Phi_I$. This forces $\gamma\not\in\Psi_\lambda^+$, a contradiction.
\end{proof}

\begin{lemma}\label{reslem4}
Let $\lambda\in\Lambda_I^+$ and $\beta\in\Phi\backslash\Phi_\lambda$. If $\gamma, \nu\in \Phi_{\beta, 2}$, then
\begin{equation}\label{resl5eq1}
\sgn(s_\gamma\lambda, s_\nu\lambda)=\sgn(s_\gamma(\lambda|_{\Phi_{\beta, 2}}), s_\nu(\lambda|_{\Phi_{\beta, 2}}), \Phi_I\cap\Phi_{\beta, 2}, \Phi_{\beta, 2}).
\end{equation}
\end{lemma}
\begin{proof}
If $\sgn(s_\gamma\lambda, s_\nu\lambda)\neq0$, then $s_\gamma\lambda, s_\nu\lambda$ are $\Phi_I$-regular. There is $w\in W_I$ such that $ws_\gamma\lambda=s_\nu\lambda$. Lemma \ref{weyllem1} gives $w_1\in W(\Phi_\lambda)\subset W(\Phi_{\beta, 2})$ with $w_1=s_\nu ws_\gamma$. Therefore $w=s_\nu w_1s_\gamma\in W(\Phi_{\beta, 2})\cap W_I=W(\Phi_I\cap\Phi_{\beta, 2})$ in view of Lemma \ref{weyllem2}. Thus
\[
ws_\gamma\lambda'=(ws_\gamma\lambda)|_{\Phi_{\beta, 2}}=(s_\nu w_1\lambda)|_{\Phi_{\beta, 2}}=(s_\nu\lambda)|_{\Phi_{\beta, 2}}=s_\nu\lambda',
\]
where $\lambda'=\lambda|_{\Phi_{\beta, 2}}$. In view of Lemma \ref{redslem4}, both $s_\gamma\lambda'$ and $s_\nu\lambda'$ are $\Phi_I\cap \Phi_{\beta, 2}$-regular. Let $\ell'(-)$ be the length function on $W(\Phi_{\beta, 2})$. Then
\[
\sgn(s_\gamma(\lambda|_{\Phi_{\beta, 2}}), s_\nu(\lambda|_{\Phi_{\beta, 2}}), \Phi_I\cap\Phi_{\beta, 2}, \Phi_{\beta, 2})=(-1)^{\ell'(w)}.
\]
Since any reflection in $W(\Phi_I\cap\Phi_{\beta, 2})$ is a product of odd number of simple reflections in $W$, One has $\ell(w)\equiv \ell'(w)\ (\mathrm{mod} 2)$ and $(-1)^{\ell'(w)}=(-1)^{\ell(w)}$. This gives (\ref{resl5eq1}).

Now assume that $\sgn(s_\gamma\lambda, s_\nu\lambda)=0$. If $\sgn(s_\gamma\lambda', s_\nu\lambda', \Phi_I\cap\Phi_{\beta, 2}, \Phi_{\beta, 2})\neq0$, there exists $w\in W(\Phi_I\cap\Phi_{\beta, 2})\subset W(\Phi_I)$ such that $ws_\gamma\lambda'=s_\nu\lambda'$. Lemma \ref{reslem1} yields
\[
ws_\gamma\lambda=ws_\gamma(\lambda-\lambda'+\lambda')=\lambda-\lambda'+ws_\gamma\lambda'=\lambda-\lambda'+s_\nu\lambda'=s_\nu\lambda.
\]
On the other hand, Lemma \ref{redslem4} implies that both $s_\gamma\lambda$ and $s_\nu\lambda$ are $\Phi_I$-regular. Thus $\sgn(s_\gamma\lambda, s_\nu\lambda)=(-1)^{\ell(w)}\neq0$, a contradiction.
\end{proof}

\textbf{Proof of Lemma \ref{invlem4} (Singular invariance)} Denote $\lambda'=\lambda|_{\Phi_{\beta, 2}}$. First suppose that $\sgn(s_\beta\lambda, \mu)\neq0$. Then $s_\beta\lambda$ is $\Phi_I$-regular. With Lemma \ref{redslem4}. The weight $s_\beta\lambda$ is $\Phi_I$-regular if and only if $s_\beta\lambda'$ is $\Phi_I\cap \Phi_{\beta, 2}$-regular. There exists $\mu'\in\Lambda^+(\Phi_I\cap\Phi_{\beta, 2}, \Phi_{\beta, 2})$ with $\sgn(s_\beta\lambda', \mu', \Phi_I\cap\Phi_{\beta, 2}, \Phi_{\beta, 2})\neq0$. Obviously the converse also holds.

Now consider the Jantzen coefficients. Still assume that $\sgn(s_\beta\lambda, \mu)\neq0$. In this case, $\mu=ws_\beta\lambda$ for $w\in W_I$ and $\mu'=w's_\beta \lambda'$ for $w'\in W(\Phi_I\cap\Phi_{\beta, 2})$. Let $\ell'(-)$ be the length function on $W(\Phi_{\beta, 2})$. One has $\ell(w')\equiv \ell'(w')\ (\mathrm{mod} 2)$. With (\ref{sjeq2}), one has
\begin{equation*}\label{invl2eq1}
\begin{aligned}
&c(\lambda', \mu', \Phi_I\cap\Phi_{\beta, 2}, \Phi_{\beta, 2})\\
=&\sum_{\gamma\in\Psi_\lambda^+\cap \Phi_{\beta, 2}}\sgn(s_\gamma\lambda', \mu', \Phi_I\cap\Phi_{\beta, 2}, \Phi_{\beta, 2})\\
=&\sum_{\gamma\in\Psi_\lambda^+\cap \Phi_{\beta, 2}}\sgn(s_\beta\lambda', \mu', \Phi_I\cap\Phi_{\beta, 2}, \Phi_{\beta, 2})\sgn(s_\gamma\lambda', s_\beta\lambda', \Phi_I\cap\Phi_{\beta, 2}, \Phi_{\beta, 2})\\
=&(-1)^{\ell'(w')}\sum_{\gamma\in\Psi_\lambda^+\cap \Phi_{\beta, 2}}\sgn(s_\gamma\lambda', s_\beta\lambda', \Phi_I\cap\Phi_{\beta, 2}, \Phi_{\beta, 2})\\
=&(-1)^{\ell(w)+\ell'(w')}\sum_{\gamma\in\Psi_\lambda^+\cap \Phi_{\beta, 2}}\sgn(s_\beta\lambda, \mu)\sgn(s_\gamma\lambda, s_\beta\lambda)\\
=&(-1)^{\ell(w)+\ell(w')}\sum_{\gamma\in\Psi_\lambda^+}\sgn(s_\beta\lambda, \mu)\sgn(s_\gamma\lambda, s_\beta\lambda)\\
=&(-1)^{\ell(w)+\ell(w')}c(\lambda, \mu),\\
\end{aligned}
\end{equation*}
where the first equality follows from Lemma \ref{redslem5} and Lemma \ref{sjlem1}, the fourth from Lemma \ref{reslem4}, and the fifth from Lemma \ref{reslem3}.

\subsection{More invariant properties of Jantzen coefficients}

The following invariant property of Jantzen coefficients exhibit a kind of parabolic-singular duality. It can be used to determine the blocks of category $\caO^\frp$ (\cite{HXZ, X2}).

\begin{lemma}[Dual invariance]\label{invlem5}
Let $\lambda$ $($resp. $\mu)$ be a dominant integral weight with $\Phi_\lambda=\Phi_J$ $($resp. $\Phi_\mu=\Phi_I)$. Choose $x, y\in{}^IW^J$. Then
\[
c(x\lambda, y\lambda, \Phi_I, \Phi)=(-1)^{\ell(x)+\ell(y)+1}c(x^{-1}\mu, y^{-1}\mu, \Phi_J, \Phi).
\]

\end{lemma}
\begin{proof}
It suffices to consider the case $c(x\lambda, y\lambda, \Phi_I, \Phi)\neq0$ with $x\lambda>y\lambda$. Choose $\beta\in\Psi^+(x\lambda, \Phi_I, \Phi)$ so that $\sgn(s_\beta x\lambda, y\lambda, \Phi_I, \Phi)\neq0$, that is, $s_\beta x\lambda=wy\lambda$ for some $w\in W_I$. In view of Lemma \ref{weyllem1}, there exists $w'\in W_J$ so that $s_\beta xw'=wy$. We claim that $\gamma\in\Psi^+(x\lambda, \Phi_I, \Phi)$ if and only if $x^{-1}\gamma\in\Psi^+(x^{-1}\mu, \Phi_J, \Phi)$. In fact, if $\gamma\in\Psi^+(x\lambda, \Phi_I, \Phi)$, then $\gamma\in\Phi^+\backslash\Phi_I$ and $\langle x\lambda, \gamma\rangle=\langle\lambda, x^{-1}\gamma\rangle>0$. We obtain $\langle x^{-1}\mu, x^{-1}\gamma\rangle=\langle\mu, \gamma\rangle>0$ and $x^{-1}\gamma\in\Phi^+\backslash\Phi_J$. In other words, $x^{-1}\gamma\in\Psi^+(x^{-1}\mu, \Phi_J, \Phi)$. The claim then follows by symmetry. It yields $x^{-1}\beta\in\Psi^+(x^{-1}\mu, \Phi_J, \Phi)$. We get
\[
s_{x^{-1}\beta}(x^{-1}\mu)=x^{-1}s_{\beta}\mu=w'y^{-1}w^{-1}\mu=w'(y^{-1}\mu)
\]
and thus $\sgn(s_{x^{-1}\beta}x^{-1}\mu, y^{-1}\mu, \Phi_J, \Phi)=(-1)^{\ell(w')}$. On the other hand, $s_\beta xw'=wy$ implies $1+\ell(x)+\ell(w')\equiv \ell(w)+\ell(y)(\mathrm{mod}\ 2)$. Thus whenever $\sgn(s_\beta x\lambda, y\lambda, \Phi_I, \Phi)$ contributes $(-1)^{\ell(w)}$ to $c(x\lambda, y\lambda, \Phi_I, \Phi)$, the formula $\sgn(s_{x^{-1}\beta}x^{-1}\mu, y^{-1}\mu, \Phi_J, \Phi)$ will also contribute $(-1)^{\ell(w')}=(-1)^{\ell(y)+\ell(x)+1}(-1)^{\ell(w)}$ to $c(x^{-1}\mu, y^{-1}\mu, \Phi_J, \Phi)$. This yields the lemma.
\end{proof}

The last lemma expresses invariance of Jantzen coefficients under conjugation. For this the following notation will be useful: Note that any weight $\mu\in\frh^*$ can be uniquely written as $\sum_{\alpha\in\Delta}c_\alpha\alpha$ with $c_\alpha\in\bbC$. For any subset $I$ of $\Delta$, the $I$-\emph{height} of $\mu$ is defined by
\[
\htt_I\mu=\sum_{\alpha\in\Delta\backslash I}c_\alpha.
\]
In particular, if $I=\emptyset$, then $\htt\mu=\htt_I\mu$ is the (ordinary) \emph{height} of $\mu$.

\begin{lemma}[Conjugate invariance]\label{invlem6}
Let $I, J\subset\Delta$. Suppose that $\Phi_I$ and $\Phi_{J}$ are $W$-conjugate. Choose $w\in W$ with $\Phi_{J}^+=w\Phi_I^+$. Let $\lambda, \mu\in\Lambda_I^+$. Then
\[
c(\lambda, \mu, \Phi_I, \Phi)=c(w\lambda, w\mu, \Phi_{J}, \Phi).
\]
\end{lemma}
\begin{proof}
It suffices to consider the case $\lambda>\mu$. Note that
\begin{equation}\label{invl6eq1}
\langle w\nu, (w\gamma)^\vee\rangle=\langle w\nu, w\gamma^\vee\rangle=\langle \nu, \gamma^\vee\rangle
\end{equation}
for any $\nu\in\frh^*$ and $\gamma\in\Phi$. With $\Phi_{J}^+=w\Phi_I^+$, we have $J=wI$ and $W_{J}=wW_Iw^{-1}$. Since $\lambda, \mu\in\Lambda_I^+$, one gets $w\lambda, w\mu\in\Lambda_{J}^+$ by (\ref{invl6eq1}).

Let $\gamma_1, \gamma_2, \cdots, \gamma_k$ be all the positive roots contained in $\Psi^+(\lambda, \Phi_I, \Phi)$ so that $\sgn(s_{\gamma_i}\lambda, \mu)\neq0$ ($1\leq i\leq k$). There exist $x_1, \cdots, x_k\in W_I$ so that $\mu=x_1s_{\gamma_1}\lambda=\cdots=x_ks_{\gamma_k}\lambda$. Thus Lemma \ref{sjlem1} gives
\[
c(\lambda, \mu, \Phi_I, \Phi)=\sum_{\gamma\in\Psi^+(\lambda, \Phi_I, \Phi)}\sgn(s_\gamma\lambda, \mu, \Phi_I, \Phi)=\sum_{i=1}^k(-1)^{\ell(x_i)}.
\]
Note that $w\mu=wx_is_{\gamma_i}\lambda=wx_iw^{-1}s_{w\gamma_i}w\lambda$. Take the $J$-heights (keeping in mind that $wx_iw^{-1}\in W_{J}$), we obtain
\[
\htt_{J}w\mu=\htt_{J}w\lambda-\langle w\lambda, (w\gamma_i)^\vee\rangle \htt_{J}w\gamma_i=\htt_{J}w\lambda-\langle \lambda, \gamma_i^\vee\rangle \htt_{J}w\gamma_i.
\]
Choose $\beta=\gamma_1$. It follows that
\begin{equation}\label{invl6eq2}
\langle \lambda, \gamma_i^\vee\rangle \htt_{J}w\gamma_i=\langle \lambda, \beta^\vee\rangle \htt_{J}w\beta
\end{equation}
for $1\leq i\leq k$. With $\Phi_{J}=w\Phi_I$, we get $\Phi\backslash\Phi_{J}=w\Phi\backslash w\Phi_{I}=w(\Phi\backslash \Phi_{I})$. Thus $w\gamma_i\in\Phi\backslash\Phi_{J}$ in view of $\gamma_i\in\Phi\backslash \Phi_{I}$.

First assume that $w\beta>0$. We claim that $\Psi^+(w\lambda, \Phi_J, \Phi)=\{w\gamma_1, \cdots, w\gamma_k\}$. Indeed, with (\ref{invl6eq2}), one has $\htt_{J}w\beta>0$ and thus $\htt_{J}w\gamma_i>0$. So $w\gamma_i\in\Phi^+\backslash\Phi_{J}$ and $\langle w\lambda, (w\gamma_i)^\vee\rangle=\langle \lambda, \gamma_i^\vee\rangle\in\bbZ^{>0}$, that is, $w\gamma_i\in\Psi^+(w\lambda, \Phi_J, \Phi)$. This implies that $w\Psi^+(\lambda, \Phi_I, \Phi)\subset\Psi^+(w\lambda, \Phi_J, \Phi)$. The converse also holds by symmetry. Therefore
\[
c(w\lambda, w\mu, \Phi_J, \Phi)=\sum_{i=1}^k\sgn(s_{w\gamma_i}w\lambda, w\mu, \Phi_J, \Phi)=\sum_{i=1}^k(-1)^{\ell(wx_iw^{-1})}=\sum_{i=1}^k(-1)^{\ell(x_i)}.
\]

Now assume that $w\beta<0$. We show that $\Psi^+(w\mu, \Phi_J, \Phi)=\{-wx_1\gamma_1, \cdots, -wx_k\gamma_k\}$. In fact, one has $-w\gamma_i\in\Phi^+\backslash\Phi_J$ by (\ref{invl6eq2}). Thus $-wx_i\gamma_i=wx_iw^{-1}(-w\gamma_i)\in\Phi^+\backslash\Phi_J$. On the other hand,
\[
\langle w\mu, (-wx_i\gamma_i)^\vee\rangle=\langle\mu, (-x_i\gamma_i)^\vee\rangle=\langle x_is_{\gamma_i}\lambda, (-x_i\gamma_i)^\vee\rangle=\langle\lambda, \gamma_i^\vee\rangle\in\bbZ^{>0}.
\]
By symmetry, one must have $\Psi^+(w\mu, \Phi_J, \Phi)=\{-wx_1\gamma_1, \cdots, -wx_k\gamma_k\}$. With $w\lambda=ws_{\gamma_i}x_i^{-1}\mu=wx_i^{-1}w^{-1}s_{-wx_i\gamma_i}w\mu$, we obtain
\[
c(w\mu, w\lambda, \Phi_J, \Phi)=\sum_{i=1}^k\sgn(s_{-wx_i\gamma_i}w\mu, w\lambda, \Phi_J, \Phi)=\sum_{i=1}^k(-1)^{\ell(wx_i^{-1}w^{-1})}=\sum_{i=1}^k(-1)^{\ell(x_i)}.
\]
Hence $c(w\lambda, w\mu, \Phi_J, \Phi)=c(w\mu, w\lambda, \Phi_J, \Phi)=c(\lambda, \mu, \Phi_I, \Phi)$.

\end{proof}

\section{Basic weights and basic generalized Verma modules}
%
%

Since the basic generalized Verma modules play significant role in the reduction process described in section 3. We will give a full classification of basic generalized Verma modules in this section.

\subsection{Basic systems} First we want to find all the full subcategories $\caO_\lambda^\frp$ that contain at least one basic generalized Verma module. Let $\Delta=\{\alpha_1, \cdots, \alpha_n\}$ be the simple roots corresponding to the standard numbering of vertices in the Dynkin diagram of $\Phi$ (\cite{H1}, \S 11.4). Let $\varpi_1, \cdots, \varpi_n$ be the \emph{fundamental weights} which satisfy $\langle\varpi_i, \alpha^\vee_j\rangle=\delta_{ij}$. Fix a basic weight $\lambda$ of $(\Phi_I, \Phi)$. With $\rank \Phi_I=\rank \Phi_\lambda=\rank\Phi_J=\rank\Phi-1$, we can assume that $I=\Delta\backslash\{\alpha_i\}$ and $J=\Delta\backslash\{\alpha_j\}$ for some $i, j\in\{1, \cdots, n\}$. Recall that $\lambda=w\overline\lambda$ for some $w\in {}^IW^J$. Since $\Phi_{\overline\lambda}=\Phi_J$, we must have $\overline\lambda=k\varpi_j$ for some $k\in\bbZ^{>0}$. Hence $\lambda$ is determined by the triple $(\Phi, i, j)$ and $w\in {}^IW^J$ and $k\in\bbZ^{>0}$. In this situation, we say $(\Phi, i, j)$ or $(\Phi, \Phi_I, \Phi_J)$ is a \emph{basic system}. In view of Lemma \ref{redslem1}, we can eventually get a basic system by applying the reduction process.

\begin{lemma}\label{blem1}
Let $(\Phi, i, j)$ be a basic system. Then
\[
\{kw\varpi_j\mid k\in\bbZ^{>0}, w\in{}^IW^J\}
\]
is the set of all the basic weights of $(\Phi, i, j)$.
\end{lemma}

By sending $w\in {}^IW^J$ to $w^{-1}\in {}^JW^I$, we have the following result.

\begin{lemma}[Corollay 2.4.1, \cite{BN}]\label{bcor1}
${}^IW^J=({}^JW^I)^{-1}$.
\end{lemma}

This yields the dual relation for basic systems.

\begin{lemma}\label{blem2}
Let $(\Phi, i, j)$ be a basic system. Then $(\Phi, j, i)$ is also a basic system.
\end{lemma}

Recall the notation $\htt_I(-)$ of $I$-height defined in the previous section. For simplicity, we denote $\htt_j(-)=\htt_{\Delta\backslash\{\alpha_j\}}(-)$.

\begin{lemma}\label{blem3}
Let $(\Phi, i, j)$ be a basic system. Assume that $\lambda=w\varpi_j$ for some $w\in{}^IW^J$. Then
\begin{equation}\label{bl3eq1}
\begin{aligned}
\frac{\langle\alpha_j, \alpha_j\rangle}{2}\htt_j\beta_0&\geq\max\{\langle\lambda, \beta\rangle\mid\beta\in\Phi\}\\
&\geq\max\{0, \langle\rho, \beta\rangle\mid\beta\in\Phi_I^+\},
\end{aligned}
\end{equation}
where $\beta_0$ is the highest root of $\Phi$ $($see for example \cite{H1}, \S 12.2$)$.
If $\Phi$ is not simply laced $($i.e., with two root lengths$)$, we also have
\begin{equation}\label{bl3eq2}
\begin{aligned}
\frac{\langle\alpha_j, \alpha_j\rangle}{2}\htt_j\beta_0^s&\geq\max\{\langle\lambda, \beta\rangle\mid\beta\in\Phi\ \mbox{is short}\}\\
&\geq\max\{0, \langle\rho, \beta\rangle\mid\beta\in\Phi_I^+\ \mbox{is short}\}, \end{aligned}
\end{equation}
where $\beta_0^s$ is the highest short root of $\Phi$.
\end{lemma}
\begin{proof}
Lemma \ref{blem1} implies $\lambda$ is a basic weight of $(\Phi, i, j)$. With $\lambda\in\Lambda_I^+$, the second inequality of (\ref{bl3eq1}) follows from the facts that $\langle\lambda, \alpha^\vee\rangle\geq1=\langle\rho, \alpha^\vee\rangle$ for any $\alpha\in I=\Delta\backslash\{\alpha_i\}$ and (when $\Phi_I^+$ is empty)
\[
\max\{\langle\lambda, \beta\rangle\mid\beta\in\Phi\}=\max\{\langle\lambda, \beta\rangle, \langle\lambda, -\beta\rangle\mid\beta\in\Phi^+\}\geq0.
\]
A similar argument proves the second inequality of (\ref{bl3eq2}). With $J=\Delta\backslash\{\alpha_j\}$, one obtains $\langle\varpi_j, \beta\rangle=\langle\varpi_j, \alpha_j\rangle\htt_j{\beta}$ for any $\beta\in\Phi$. So we get
\[
\langle\lambda, \beta\rangle=\langle w^{-1}\lambda, w^{-1}\beta\rangle=\langle\varpi_j, w^{-1}\beta\rangle=\langle\varpi_j, \alpha_j\rangle\htt_j({w^{-1}\beta})\leq \langle\varpi_j, \alpha_j\rangle\htt_j\beta_0.
\]
Then the first inequality of (\ref{bl3eq1}) follows from $\langle\varpi_j, \alpha_j\rangle={\langle\alpha_j, \alpha_j\rangle}/{2}$. The proof of the first inequality of (\ref{bl3eq2}) is similar.
\end{proof}

Now we can present all the basic systems.

\begin{theorem}\label{bthm1}
Using the above notation, a basic system $(\Phi, i, j)$ must be one of the following cases.
\begin{itemize}
\item [(1)] $(A_1, 1, 1)$, $(A_2, 1, 1)$, $(A_2, 1, 2)$, $(A_2, 2, 1)$, $(A_2, 2, 2)$, $(A_3, 2, 2)$;

\item [(2)] $(B_2, 1, 1)$, $(B_2, 1, 2)$, $(B_2, 2, 1)$, $(B_2, 2, 2)$, $(B_3, 2, 2)$, $(B_3, 2, 3)$, $(B_3, 3, 2)$, $(B_4, 3, 3)$;

\item [(3)] $(C_2, 1, 1)$, $(C_2, 1, 2)$, $(C_2, 2, 1)$, $(C_2, 2, 2)$, $(C_3, 2, 2)$,  $(C_3, 2, 3)$, $(C_3, 3, 2)$, $(C_4, 3, 3)$;

\item [(4)] $(D_4, 2, 2)$, $(D_5, 3, 3)$;

\item [(5)] $(E_6, 4, 4)$, $(E_7, 4, 4)$, $(E_7, 4, 5)$, $(E_7, 5, 4)$, $(E_8, 3, 4)$, $(E_8, 4, 3)$, $(E_8, 4, 4)$, $(E_8, 4, 5)$, $(E_8, 5, 4)$, $(E_8, 5, 5)$;

\item [(6)] $(F_4, 2, 2)$, $(F_4, 2, 3)$, $(F_4, 3, 2)$, $(F_4, 3, 3)$;

\item [(7)] $(G_2, 1, 1)$, $(G_2, 1, 2)$, $(G_2, 2, 1)$, $(G_2, 2, 2)$.
\end{itemize}
\end{theorem}

\begin{remark}
Some of the above basic systems are isomorphic (e.g., $(A_2, 1, 1)$ and $(A_2, 2, 2)$). For symmetry and convenience, we keep all of them in our argument.

In the next subsection, we will show that each triple described in Theorem \ref{bthm1} is indeed a basic system. In other words, it contains at least a basic weight.
\end{remark}

We use the standard realization of $\Phi$ (see \cite{H1} \S 12.1), that is, $\Phi$ is a subset of a real vector space with orthonormal basis $e_i$. Denote $\lambda_i=\langle\lambda, e_i\rangle$ for a given $\lambda\in\frh^*$ and write
\begin{equation}\label{realization}
\lambda=(\lambda_1, \lambda_2, \cdots, \lambda_k)=\lambda_1e_1+\lambda_2e_2+\cdots+\lambda_ke_k,
\end{equation}
where $k=n$ or $n+1$ depending on $\Phi$.

From now on in this section, set $a_j=\frac{\langle\alpha_j, \alpha_j\rangle}{2}\htt_j\beta_0$ (resp. $a_j^s=\frac{\langle\alpha_j, \alpha_j\rangle}{2}\htt_j\beta_0^s$) and $b_i=\max\{0, \langle\rho, \beta\rangle\mid\beta\in\Phi_I^+\}$ (resp. $b_i^s=\max\{0, \langle\rho, \beta\rangle\mid\beta\in\Phi_I^+\ \mbox{is short}\}$), where $I=\Delta\backslash\{\alpha_i\}$. Set $a=(a_1, a_2, \cdots, a_n)$ and $b=(b_1, b_2, \cdots, b_n)$ (resp. $a^s=(a_1^s, a_2^s, \cdots, a_n^s)$ and $b^s=(b_1^s, b_2^s, \cdots, b_n^s)$). Lemma \ref{blem3} yields $a_j\geq b_i$ (resp. $a_j^s\geq b_i^s$) when $(\Phi, i, j)$ is a basic system. With Lemma \ref{blem2}, we also have $a_i\geq b_j$ (resp. $a_i^s\geq b_j^s$).

\begin{proof}
(1) $A_n (n\geq1)$. The highest root $\beta_0$ is $e_1-e_{n+1}$. One has $\htt_j(e_1-e_{n+1})=1$ and $\langle\alpha_j, \alpha_j\rangle=2$ for any $j\in\{1, \cdots, n\}$. So $a=(1, \cdots, 1)$. On the other hand, since $\Phi_I\simeq A_{i-1}\times A_{n-i}$,
\[
b_i=\max\{\langle\rho, e_1-e_i\rangle, \langle\rho, e_{i+1}-e_{n+1}\rangle\}=\max\{i-1, n-i\}.
\]
With $a_j\geq b_i$, we obtain $1\geq\max\{i-1, n-i\}$. It follows that $i\leq2$ and $n\leq3$. There is nothing to prove for $n=1, 2$ since the list exhausts all the possible cases. If $n=3$, we must have $i=2$. This forces $j=2$ since $a_i\geq b_j$.

(2) $B_n (n\geq2)$. With $\beta_0=e_1+e_2$ and $\beta_0^s=e_1$, one obtains $a=(1, 2, \cdots, 2, 1)$ and $a^s=(1, \cdots, 1, 1/2)$. On the other hand, if $1\leq i\leq n-2$, then $\Phi_I\simeq A_{i-1}\times B_{n-i}$ and $e_1-e_i$, $e_{i+1}+e_{i+2}$ are the corresponding highest roots, while $e_{i+1}$ is the highest short root. It follows that
\[
b_i=\max\{\langle\rho, e_1-e_i\rangle, \langle\rho, e_{i+1}+e_{i+2}\rangle\}=\max\{i-1, 2(n-i-1)\}
\]
and $b_i^s=\langle\rho, e_{i+1}\rangle=n-i-1/2$. Moreover,
\[
b_{n-1}=\max\{\langle\rho, e_1-e_{n-1}\rangle, \langle\rho, e_{n}\rangle\}=\max\{n-2, 1/2\}
\]
and $b_{n-1}^s=\langle\rho, e_{n}\rangle=1/2$, while $b_{n}=n-1$ and $b_n^s=0$.

With $a_j\geq b_i$ and $a_j^s\geq b_i^s$, one has $2\geq b_i$ and $1\geq b_i^s$. If $n\geq 5$, then $i\neq n-1, n$. It follows that $2\geq \max\{i-1, 2(n-i-1)\}$ for some $1\leq i\leq n-2$. This forces $n=5$ and $i=3$. However, it yields $1\geq b_3^s=3/2$, a contradiction. If $n=4$, we get $i=3$ since $b_1, b_4>2$ and $b_2^s>1$. By symmetry, we also get $j=3$. If $n=3$, then $(B_3, 3, 3)$ is not a basic system in view of $a_3=1<b_3=2$. We obtain $i\neq1$ since $b_1^s=3/2>1$. By symmetry, this means $j\neq 1$. There is nothing to prove for $n=2$.

(3) $C_n (n\geq2)$. The argument is similar to that of $B_n$. With $\beta_0=2e_1$, one has $a=(2, \cdots, 2)$. On the other hand,
\[
b_i=\max\{\langle\rho, e_1-e_i\rangle, \langle\rho, 2e_{i+1}\rangle\}=\max\{i-1, 2(n-i)\}.
\]
for $1\leq i\leq n-1$ and $b_n=\langle\rho, e_1-e_n\rangle=n-1$. If $n\geq 4$, then $2=a_j\geq b_i$ implies $i\neq n$. So $2\geq \max\{i-1, 2n-2i\}$ for $1\leq i\leq n-1$. This forces $n=4$ and $i=3$. One has $j=3$ by symmetry. Now consider the case $n=3$. Since $b_1=4$, we get $i\neq 1$ and $j\neq1$. If $(C_3, 3, 3)$ is a basic system, there exists basic weight $\lambda=w\varpi_3=w(1, 1, 1)$ for some $w\in {}^IW^J$. Thus $|\lambda_1|=|\lambda_2|=|\lambda_3|=1$. With $I=\{e_1-e_2, e_2-e_3\}$, one has $\lambda_1-\lambda_2, \lambda_2-\lambda_3\in\bbZ^{>0}$. This forces $\lambda_1=1>\lambda_2=-1>\lambda_3$, a contradiction. There is nothing to prove for $n=2$.

(4) $D_n (n\geq4)$. With $\beta_0=e_1+e_2$, we get $a=(1, 2, 2, \dots, 2, 1, 1)$. On the other hand, if $1\leq i\leq n-3$,
\[
b_i=\max\{\langle\rho, e_1-e_{i}\rangle, \langle\rho, e_{i+1}+e_{i+2}\rangle\}=\max\{i-1, 2n-2i-3\}
\]
Moreover,
\[
b_{n-2}=\max\{\langle\rho, e_1-e_{n-2}\rangle, \langle\rho, e_{n-1}-e_{n}\rangle, \langle\rho, e_{n-1}+e_{n}\rangle\}=n-3
\]
and $b_{n-1}=b_n=n-1$. It follows from $2\geq a_j\geq \max\{i-1, 2n-2i-3\}$ for $i\leq n-3$ that $n<6$. If $n=5$, we must have $i=3$ in view of $2\geq b_i$. By symmetry, we obtain $j=3$. Now consider $n=4$. First $i=2$ since $b_1=b_3=b_4=3>2$. Then $j=2$ by symmetry.

(5) $E_n (n=6, 7, 8)$. For $E_6$, the highest root is $\frac{1}{2}(e_1+e_2+e_3+e_4+e_5-e_6-e_7+e_8)=\alpha_1+2\alpha_2+2\alpha_3+3\alpha_4+2\alpha_5+\alpha_6$. So $a=(1, 2, 2, 3, 2, 1)$. Note that $\Phi_1\simeq D_5$, $\Phi_2\simeq A_5$, $\Phi_3\simeq A_1\times A_4$, $\Phi_4\simeq A_2\times A_1\times A_2$, $\Phi_5\simeq A_4\times A_1$ and $\Phi_6\simeq D_5$. Since $\langle\rho, \beta\rangle=\htt\beta$ for $\beta\in E_n$, we get $b=(7, 5, 4, 2, 4, 7)$. The only pair $(i, j)$ satisfying $a_j\geq b_i$ and $a_i\geq b_j$ is $(4, 4)$. For $E_7$, the highest root is $e_8-e_7=2\alpha_1+2\alpha_2+3\alpha_3+4\alpha_4+3\alpha_5+2\alpha_6+\alpha_7$. So $a=(2, 2, 3, 4, 3, 2, 1)$. On the other hand, we can get $b=(9, 6, 5, 3, 4, 7, 11)$. The pairs $(i, j)$ satisfying $a_j\geq b_i$ and $a_i\geq b_j$ are $(4, 4)$, $(4, 5)$ and $(5, 4)$.  For $E_8$, the highest root is $e_8+e_7=2\alpha_1+3\alpha_2+4\alpha_3+6\alpha_4+5\alpha_5+4\alpha_6+3\alpha_7+2\alpha_8$. So $a=(2, 3, 4, 6, 5, 4, 3, 2)$. In this case, $b=(11, 7, 6, 4, 4, 7, 11, 17)$. The pairs $(i, j)$ satisfying $a_j\geq b_i$ and $a_i\geq b_j$ are $(3, 4)$, $(4, 3)$, $(4, 4)$, $(4, 5)$, $(5, 4)$ and $(5, 5)$.

(5) $F_4$. With $\beta_0=e_1+e_2=2\alpha_1+3\alpha_2+4\alpha_3+2\alpha_4$ and $\beta_0^s=e_1=\alpha_1+2\alpha_2+3\alpha_3+2\alpha_4$, one obtains $a=(2, 3, 2, 1)$ and $a^s=(1, 2, 3/2, 1)$. Moreover, since $\rho=\frac{11}{2}e_1+\frac{5}{2}e_2+\frac{3}{2}e_3+\frac{1}{2}e_4$, we get $b_1=\langle\rho, e_1-e_2\rangle=3$, $b_4=\langle\rho, e_2+e_3\rangle=4$,
\[
b_2=\max\{\langle\rho, e_2-e_3\rangle, \langle\rho, \frac{1}{2}(e_1-e_2-e_3+e_4)\rangle\}=1,
\]
and
\[
b_3=\max\{\langle\rho, e_2-e_4\rangle, \langle\rho, \frac{1}{2}(e_1-e_2-e_3-e_4)\rangle\}=2.
\]
So $b=(3, 1, 2, 4)$. Similarly we can get $b^s=(5/2, 1, 1/2, 5/2)$. The pairs $(i, j)$ satisfying $a_j\geq b_i$, $a_i\geq b_j$, $a_j^s\geq b_i^s$ and $a_i^s\geq b_j^s$ are $(2, 2)$, $(2, 3)$, $(3, 2)$ and $(3, 3)$.

(5) $G_2$. There is nothing to prove.
\end{proof}

\subsection{Basic weights} Now we give all the basic weights in a case-by-case fashion. As in the previous subsection, we still use the form (\ref{realization}) for a weight $\lambda$. Let $(\Phi, i, j)$ be a basic system. With Lemma \ref{blem1}, it suffices to consider basic weights $w\varpi_j$ for $w\in {}^IW^J$.

\subsubsection{$A_n$}
\begin{theorem}\label{bwthm1}
Using the above notation, any basic weight of type $A$ must be one of the following cases $($up to a positive integer$)$.
$(1)$ $(A_1, 1, 1)$. $(\frac{1}{2}, -\frac{1}{2})$, $(-\frac{1}{2}, \frac{1}{2})$; $(2)$ $(A_2, 1, 1)$. $(-\frac{1}{3}, \frac{2}{3}, -\frac{1}{3})$; $(3)$ $(A_2, 1, 2)$. $(\frac{1}{3}, \frac{1}{3}, -\frac{2}{3})$; $(4)$ $(A_2, 2, 1)$. $(\frac{2}{3}, -\frac{1}{3}, -\frac{1}{3})$; $(5)$ $(A_2, 2, 2)$. $(\frac{1}{3}, -\frac{2}{3}, \frac{1}{3})$; $(6)$ $(A_3, 2, 2)$. $(\frac{1}{2}, -\frac{1}{2}, \frac{1}{2}, -\frac{1}{2})$;
\end{theorem}
\begin{proof}
Let $\lambda=w\varpi_j$ be a basic weight. Then $\lambda$ is integral and $\lambda\in\Lambda_I^+$.

(1) $(A_1, 1, 1)$. In this case, $I=J=\emptyset$. So ${}^IW^J=W=\{1, s_{\alpha_1}\}$. Therefore $\lambda=\varpi_1$ or $s_{\alpha_1}\varpi_1$, where $\varpi_1=(\frac{1}{2}, -\frac{1}{2})$;

$(2)$ $(A_2, 1, 1)$. Since $I=\Delta\backslash\{\alpha_1\}=\{\alpha_2\}=J$, we get ${}^IW^J=\{s_{\alpha_1}\}$ and $\lambda=s_{\alpha_1}\varpi_1=(-\frac{1}{3}, \frac{2}{3}, -\frac{1}{3})$, where $\varpi_1=(\frac{2}{3}, -\frac{1}{3}, -\frac{1}{3})$;

$(3)$ $(A_2, 1, 2)$.
Since $I=\{\alpha_2\}$ and $J=\{\alpha_1\}$, one gets ${}^IW^J=\{1\}$ and $\lambda=\varpi_2=(\frac{1}{3}, \frac{1}{3}, -\frac{2}{3})$;

$(4)$ $(A_2, 2, 1)$. ${}^IW^J=\{1\}$ and $\lambda=\varpi_1=(\frac{2}{3}, -\frac{1}{3}, -\frac{1}{3})$;

$(5)$ $(A_2, 2, 2)$. ${}^IW^J=\{s_{\alpha_2}\}$ and $\lambda=s_{\alpha_2}\varpi_2=(\frac{1}{3}, -\frac{2}{3}, \frac{1}{3})$;

$(6)$ $(A_3, 2, 2)$. Since $I=J=\{e_1-e_2, e_3-e_4\}$ and $\lambda\in\Lambda_I^+$, we get $\lambda=w\varpi_2=w(\frac{1}{2}, \frac{1}{2}, -\frac{1}{2}, -\frac{1}{2})$ and $\lambda_1-\lambda_2, \lambda_3-\lambda_4\in\bbZ^{>0}$. This forces $\lambda_1=\lambda_3=\frac{1}{2}$ and $\lambda_2=\lambda_4=-\frac{1}{2}$, that is, $\lambda=(\frac{1}{2}, -\frac{1}{2}, \frac{1}{2}, -\frac{1}{2})$ and $w=s_{\alpha_2}$.
\end{proof}

\subsubsection{$B_n$}

\begin{theorem}\label{bwthm2}
Any basic weight of type $B$ must be one of the following cases $($up to a positive integer$)$.
$(1)$ $(B_2, 1, 1)$. $(0, 1)$; $(2)$ $(B_2, 1, 2)$. $(\frac{1}{2}, \frac{1}{2})$, $(-\frac{1}{2}, \frac{1}{2})$; $(3)$ $(B_2, 2, 1)$. $(1, 0)$, $(0, -1)$; $(4)$ $(B_2, 2, 2)$. $(\frac{1}{2}, -\frac{1}{2})$; $(5)$ $(B_3, 2, 2)$. $(1, 0, 1)$, $(0, -1, 1)$; $(6)$ $(B_3, 2, 3)$. $(\frac{1}{2}, -\frac{1}{2}, \frac{1}{2})$; $(7)$ $(B_3, 3, 2)$. $(1, 0, -1)$; $(8)$ $(B_4, 3, 3)$. $(1, 0, -1, 1)$.
\end{theorem}
\begin{proof}
Let $\lambda=w\varpi_j$ be a basic weight with $w\in{}^IW^J$.

$(1)$ $(B_2, 1, 1)$. In this case, $I=J=\Delta\backslash\{\alpha_1\}=\{\alpha_2\}$, we get ${}^IW^J=\{s_{\alpha_1}\}$ and $\lambda=s_{\alpha_1}\varpi_1=(0, 1)$, where $\varpi_1=(1, 0)$;

$(2)$ $(B_2, 1, 2)$. With $I=\{\alpha_2\}$ and $J=\{\alpha_1\}$, we have ${}^IW^J=\{1, s_{\alpha_1}s_{\alpha_2}\}$ and $\lambda=\varpi_2=(\frac{1}{2}, \frac{1}{2})$ or $\lambda=s_{\alpha_1}s_{\alpha_2}\varpi_2=(-\frac{1}{2}, \frac{1}{2})$;

$(3)$ $(B_2, 2, 1)$.  With $I=\{\alpha_1\}$ and $J=\{\alpha_2\}$, we have ${}^IW^J=\{1,s_{\alpha_2}s_{\alpha_1}\}$ and $\lambda=\varpi_1=(1, 0)$ or $\lambda=s_{\alpha_2}s_{\alpha_1}\varpi_2=(0, -1)$;

$(4)$ $(B_2, 2, 2)$.  With $I=\{\alpha_1\}$ and $J=\{\alpha_1\}$, we have ${}^IW^J=\{s_{\alpha_2}\}$ and $\lambda=s_{\alpha_2}\varpi_2=(\frac{1}{2}, -\frac{1}{2})$;

$(5)$ $(B_3, 2, 2)$. With $I=\{e_1-e_2, e_3\}=J$, we get $\lambda=w\varpi_2=w(1, 1, 0)$ and $\lambda_1-\lambda_2, 2\lambda_3\in\bbZ^{>0}$. Thus $\lambda=(1, 0, 1)$ or $(0, -1, 1)$.

$(6)$ $(B_3, 2, 3)$. With $I=\{e_1-e_2, e_3\}$ and $J=\{e_1-e_2, e_2-e_3\}$. One has $\lambda=w\varpi_3=w(\frac{1}{2}, \frac{1}{2}, \frac{1}{2})$ and $\lambda_1-\lambda_2, 2\lambda_3\in\bbZ^{>0}$. We must have $\lambda=(\frac{1}{2}, -\frac{1}{2}, \frac{1}{2})$.

$(7)$ $(B_3, 3, 2)$. In this case, $I=\{e_1-e_2, e_2-e_3\}$ and $J=\{e_1-e_2, e_3\}$. We get $\lambda=w(1, 1, 0)$ and $\lambda_1-\lambda_2, \lambda_2-\lambda_3\in\bbZ^{>0}$. This forces $\lambda=(1, 0, -1)$.

$(8)$ $(B_4, 3, 3)$. Now $I=\{e_1-e_2, e_2-e_3, e_4\}=J$. One obtains $\lambda=w(1, 1, 1, 0)$ and $\lambda_1-\lambda_2, \lambda_2-\lambda_3, 2\lambda_4\in\bbZ^{>0}$. This yields $\lambda=(1, 0, -1, 1)$.
\end{proof}

\subsubsection{$C_n$}

\begin{theorem}\label{bwthm3}
Any basic weight of type $C$ must be one of the following cases $($up to a positive integer$)$.
$(1)$ $(C_2, 1, 1)$. $(0, 1)$; $(2)$ $(C_2, 1, 2)$. $(1, 1)$, $(-1, 1)$; $(3)$ $(C_2, 2, 1)$. $(1, 0)$, $(0, -1)$; $(4)$ $(C_2, 2, 2)$. $(1, -1)$; $(5)$ $(C_3, 2, 2)$. $(1, 0, 1)$, $(0, -1, 1)$; $(6)$ $(C_3, 2, 3)$. $(1, -1, 1)$; $(7)$ $(C_3, 3, 2)$. $(1, 0, -1)$; $(8)$ $(C_4, 3, 3)$. $(1, 0, -1, 1)$.
\end{theorem}
\begin{proof}
The argument is similar to that of Theorem \ref{bwthm2}.
\end{proof}

\subsubsection{$D_n$}

\begin{theorem}\label{bwthm4}
Any basic weight of type $D$ must be one of the following cases $($up to a positive integer$)$.
$(1)$ $(D_4, 2, 2)$. $(1, 0, 1, 0)$, $(0, -1, 1, 0)$; $(2)$ $(D_5, 3, 3)$. $(1, 0, -1, 1, 0)$.
\end{theorem}
\begin{proof}
Let $\lambda=w\varpi_j$ be a basic weight with $w\in{}^IW^J$.
$(1)$ $(D_4, 2, 2)$. In this case, $I=\{e_1-e_2, e_3-e_4, e_3+e_4\}=J$. So $\lambda=w(1, 1, 0, 0)$ and $\lambda_1-\lambda_2$, $\lambda_3-\lambda_4, \lambda_3+\lambda_4\in\bbZ^{>0}$. This forces $\lambda=(1, 0, 1, 0)$ or $(0, -1, 1, 0)$.

$(2)$ $(D_5, 3, 3)$. With $I=\{e_1-e_2, e_2-e_3, e_4\pm e_5\}=J$, we get $\lambda=w(1, 1, 1, 0, 0)$ and $\lambda_1-\lambda_2, \lambda_2-\lambda_3, \lambda_4\pm\lambda_5\in\bbZ^{>0}$. Hence $\lambda=(1, 0, -1, 1, 0)$.
\end{proof}


\subsubsection{$E_6, E_7, E_8$}. The cases of type $E$ are much more complicated. We have to run computer programs to calculate all the possible basic weights. The algorithm is summarized as follows (take $E_8$ as an example).

Suppose the basic system is $(E_8, i, j)$. Assume that $\lambda=w\varpi_j=\sum_{k=1}^8x_k\varpi_k$ with $x_k\in\bbZ$. In particular, $x_k\in \bbZ^{>0}$ when $k\neq i$ (since $\lambda\in\Lambda_I^+$). Changing basis,
\[
\begin{aligned}
\lambda=&\frac{x_2-x_3}{2}e_1+\frac{x_2+x_3}{2}e_2+(\frac{x_2+x_3}{2}+x_4)e_3+(\frac{x_2+x_3}{2}+x_4+x_5)e_4\\
&+(\frac{x_2+x_3}{2}+x_4+x_5+x_6)e_5+(\frac{x_2+x_3}{2}+x_4+x_5+x_6+x_7)e_6\\
&+(\frac{x_2+x_3}{2}+x_4+x_5+x_6+x_7+x_8)e_7\\
&+(2x_1+\frac{5x_2+7x_3}{2}+5x_4+4x_5+3x_6+2x_7+x_8)e_8.
\end{aligned}
\]
Note that $\langle\lambda, \lambda\rangle=\langle w\varpi_j, w\varpi_j\rangle=\langle \varpi_j, \varpi_j\rangle$. Thus the first step is to find all the integer solutions $x_k$ (with $x_k>0$ for $k\neq i$) of the equation
\[
\begin{aligned}
\langle \varpi_j, \varpi_j\rangle=&(\frac{x_2-x_3}{2})^2+(\frac{x_2+x_3}{2})^2+(\frac{x_2+x_3}{2}+x_4)^2+(\frac{x_2+x_3}{2}+x_4+x_5)^2\\
&+(\frac{x_2+x_3}{2}+x_4+x_5+x_6)^2+(\frac{x_2+x_3}{2}+x_4+x_5+x_6+x_7)^2\\
&+(\frac{x_2+x_3}{2}+x_4+x_5+x_6+x_7+x_8)^2\\
&+(2x_1+\frac{5x_2+7x_3}{2}+5x_4+4x_5+3x_6+2x_7+x_8)^2.
\end{aligned}
\]

For each solution $(x_1, \cdots, x_8)$ from the first step, write $\mu=\lambda=\sum_{k=1}^8x_k\varpi_k$. If $\langle\lambda, \alpha\rangle<0$ for any $\alpha\in \Delta$, then $\htt (s_\alpha\mu)-\htt_\mu=-\langle\mu, \alpha^\vee\rangle>0$. Replace $\mu$ by $s_\alpha\mu$. Raising $\htt\mu$ stepwise in this fashion, we arrive at a dominant weight $\mu$. If $\mu=\varpi_j$, then $\lambda=\sum_{k=1}^8x_k\varpi_k$ is a basic weight of $(\Phi, i, j)$. Otherwise we discard this solution.

Similar algorithms give all the basic weights of type $E$, which we list in the following tables (up to a positive integer). For any given basic system of type $E$, a number $n(\lambda)$ is assigned to each basic weight $\lambda$ of the system such that $n(\lambda)<n(\mu)$ whenever $\lambda>\mu$, where $\mu$ is another basic weight.

\begin{theorem}\label{bwthm5}
Any basic weight of type $E$ must be one of the weights $($up to a positive integer$)$ listed in Table 1-10.
\end{theorem}

\renewcommand\arraystretch{1.3}

\begin{table}[H]
\begin{tabular}{|l|c|l|c|}
\hline
$1$ & $(0, 1, -1, 0, 1, -1, -1, 1)$ & $3$ & $(0, 1, -2, -1, 0, 0, 0, 0)$ \\
\hline
$2$ & $(\frac{1}{2}, \frac{3}{2}, -\frac{3}{2}, -\frac{1}{2}, \frac{1}{2}, -\frac{1}{2}, -\frac{1}{2}, \frac{1}{2})$ & &\\
\hline
\end{tabular}
\bigskip
\caption{Basic weights of $(E_6, 4, 4)$}
\end{table}

\begin{table}[H]
\begin{tabular}{|l|c|l|c|}
\hline
$1$ & $(\frac{1}{2}, \frac{3}{2}, -\frac{3}{2}, -\frac{1}{2}, \frac{1}{2}, \frac{3}{2}, -\frac{3}{2}, \frac{3}{2})$ & $4$ & $(\frac{1}{2}, \frac{3}{2}, -\frac{5}{2}, -\frac{3}{2}, -\frac{1}{2}, \frac{1}{2}, -\frac{1}{2}, \frac{1}{2})$ \\
\hline
$2$ & $(0, 1, -2, -1, 0, 2, -1, 1)$ & $5$ & $(0, 1, -3, -1, 0, 1, 0, 0)$ \\
\hline
$3$ & $(0, 2, -2, -1, 0, 1, -1, 1)$ & $6$ & $(\frac{1}{2}, \frac{3}{2}, -\frac{5}{2}, -\frac{3}{2}, -\frac{1}{2}, \frac{1}{2}, \frac{1}{2}, -\frac{1}{2})$ \\
\hline
\end{tabular}
\bigskip
\caption{Basic weights of $(E_7, 4, 4)$}
\end{table}

\begin{table}[H]
\begin{tabular}{|l|c|}
\hline
$1$ & $(0, 1, -2, -1, 0, 1, -\frac{1}{2}, \frac{1}{2})$ \\
\hline
\end{tabular}
\bigskip
\caption{Basic weights of $(E_7, 4, 5)$}
\end{table}

\begin{table}[H]
\begin{tabular}{|l|c|}
\hline
$1$ & $(0, 1, 2, -2, -1, 0, -1, 1)$ \\
\hline
\end{tabular}
\bigskip
\caption{Basic weights of $(E_7, 5, 4)$}
\end{table}

\begin{table}[H]
\begin{tabular}{|l|c|}
\hline
$1$ & $(\frac{7}{2}, -\frac{5}{2}, -\frac{3}{2}, -\frac{1}{2}, \frac{1}{2}, \frac{3}{2}, \frac{5}{2}, \frac{1}{2})$ \\
\hline
\end{tabular}
\bigskip
\caption{Basic weights of $(E_8, 3, 4)$}
\end{table}

\begin{table}[H]
\begin{tabular}{|l|c|}
\hline
$1$ & $(\frac{1}{2}, \frac{3}{2}, -\frac{5}{2}, -\frac{3}{2}, -\frac{1}{2}, \frac{1}{2}, \frac{3}{2}, \frac{1}{2})$ \\
\hline
\end{tabular}
\bigskip
\caption{Basic weights of $(E_8, 4, 3)$}
\end{table}

\begin{table}[H]
\begin{tabular}{|l|c|l|c|}
\hline
$1$ & $(0, 2, -2, -1, 0, 1, 2, 4)$ & $25$ & $(\frac{1}{2}, \frac{3}{2}, -\frac{7}{2}, -\frac{5}{2}, -\frac{1}{2}, \frac{3}{2}, \frac{5}{2}, \frac{1}{2})$ \\
\hline
$2$ & $(\frac{1}{2}, \frac{3}{2}, -\frac{5}{2}, -\frac{3}{2}, -\frac{1}{2}, \frac{1}{2}, \frac{5}{2}, \frac{7}{2})$ & $26$ & $(0, 2, -4, -2, -1, 0, 2, 1)$\\
\hline
$3$ & $(0, 1, -3, -1, 0, 1, 3, 3)$ & $27$ & $(\frac{3}{2}, \frac{5}{2}, -\frac{7}{2}, -\frac{5}{2}, -\frac{1}{2}, \frac{1}{2}, \frac{3}{2}, \frac{1}{2})$\\
\hline
$4$ & $(\frac{1}{2}, \frac{5}{2}, -\frac{5}{2}, -\frac{3}{2}, -\frac{1}{2}, \frac{1}{2}, \frac{3}{2}, \frac{7}{2})$ & $28$ & $(\frac{1}{2}, \frac{3}{2}, -\frac{7}{2}, -\frac{5}{2}, -\frac{3}{2}, \frac{1}{2}, \frac{5}{2}, \frac{1}{2})$\\
\hline
$5$ & $(\frac{1}{2}, \frac{3}{2}, -\frac{5}{2}, -\frac{3}{2}, -\frac{1}{2}, \frac{1}{2}, \frac{7}{2}, \frac{5}{2})$ & $29$ & $(1, 2, -4, -2, 0, 1, 2, 0)$\\
\hline
$6$ & $(\frac{1}{2}, \frac{7}{2}, -\frac{5}{2}, -\frac{3}{2}, -\frac{1}{2}, \frac{1}{2}, \frac{3}{2}, \frac{5}{2})$ & $30$ & $(\frac{1}{2}, \frac{5}{2}, -\frac{7}{2}, -\frac{5}{2}, -\frac{3}{2}, \frac{1}{2}, \frac{3}{2}, \frac{1}{2})$\\
\hline
$7$ & $(\frac{1}{2}, \frac{3}{2}, -\frac{7}{2}, -\frac{3}{2}, -\frac{1}{2}, \frac{1}{2}, \frac{5}{2}, \frac{5}{2})$ & $31$ & $(-\frac{1}{2}, \frac{3}{2}, -\frac{7}{2}, -\frac{5}{2}, -\frac{3}{2}, -\frac{1}{2}, \frac{5}{2}, \frac{1}{2})$\\
\hline
$8$ & $(\frac{1}{2}, \frac{5}{2}, -\frac{7}{2}, -\frac{3}{2}, -\frac{1}{2}, \frac{1}{2}, \frac{3}{2}, \frac{5}{2})$ & $32$ & $(0, 2, -4, -2, -1, 1, 2, 0)$\\
\hline
$9$ & $(0, 2, -4, -1, 0, 1, 2, 2)$ & $33$ & $(-\frac{1}{2}, \frac{5}{2}, -\frac{7}{2}, -\frac{5}{2}, -\frac{3}{2}, -\frac{1}{2}, \frac{3}{2}, \frac{1}{2})$\\
\hline
$10$ & $(\frac{1}{2}, \frac{3}{2}, -\frac{7}{2}, -\frac{5}{2}, -\frac{1}{2}, \frac{1}{2}, \frac{3}{2}, \frac{5}{2})$ & $34$ & $(1, 2, -4, -2, -1, 0, 2, 0)$\\
\hline
$11$ & $(0, 1, -4, -2, 0, 1, 2, 2)$ & $35$ & $(-1, 2, -4, -2, -1, 0, 2, 0)$\\
\hline
$12$ & $(\frac{1}{2}, \frac{5}{2}, -\frac{7}{2}, -\frac{3}{2}, -\frac{1}{2}, \frac{1}{2}, \frac{5}{2}, \frac{3}{2})$ & $36$ & $(-\frac{3}{2}, \frac{5}{2}, -\frac{7}{2}, -\frac{5}{2}, -\frac{3}{2}, -\frac{1}{2}, \frac{1}{2}, \frac{1}{2})$\\
\hline
$13$ & $(0, 1, -4, -2, -1, 0, 2, 2)$ & $37$ & $(-\frac{1}{2}, \frac{3}{2}, -\frac{7}{2}, -\frac{5}{2}, -\frac{3}{2}, \frac{1}{2}, \frac{5}{2}, -\frac{1}{2})$\\
\hline
$14$ & $(-\frac{1}{2}, \frac{3}{2}, -\frac{7}{2}, -\frac{5}{2}, -\frac{3}{2}, -\frac{1}{2}, \frac{1}{2}, \frac{5}{2})$ & $38$ & $(\frac{1}{2}, \frac{3}{2}, -\frac{7}{2}, -\frac{5}{2}, -\frac{3}{2}, -\frac{1}{2}, \frac{5}{2}, -\frac{1}{2})$\\
\hline
$15$ & $(\frac{1}{2}, \frac{3}{2}, -\frac{7}{2}, -\frac{5}{2}, -\frac{1}{2}, \frac{1}{2}, \frac{5}{2}, \frac{3}{2})$ & $39$ & $(-\frac{1}{2}, \frac{5}{2}, -\frac{7}{2}, -\frac{5}{2}, -\frac{3}{2}, \frac{1}{2}, \frac{3}{2}, -\frac{1}{2})$\\
\hline
$16$ & $(0, 2, -4, -2, -1, 0, 1, 2)$ & $40$ & $(\frac{1}{2}, \frac{5}{2}, -\frac{7}{2}, -\frac{5}{2}, -\frac{3}{2}, -\frac{1}{2}, \frac{3}{2}, -\frac{1}{2})$\\
\hline
$17$ & $(1, 2, -3, -2, -1, 1, 3, 1)$ & $41$ & $(\frac{3}{2}, \frac{5}{2}, -\frac{7}{2}, -\frac{5}{2}, -\frac{3}{2}, -\frac{1}{2}, \frac{1}{2}, -\frac{1}{2})$\\
\hline
$18$ & $(\frac{1}{2}, \frac{5}{2}, -\frac{7}{2}, -\frac{5}{2}, -\frac{1}{2}, \frac{1}{2}, \frac{3}{2}, \frac{3}{2})$ & $42$ & $(0, 2, -4, -2, -1, 0, 2, -1)$\\
\hline
$19$ & $(-\frac{1}{2}, \frac{3}{2}, -\frac{9}{2}, -\frac{3}{2}, -\frac{1}{2}, \frac{1}{2}, \frac{3}{2}, \frac{3}{2})$ & $43$ & $(\frac{1}{2}, \frac{3}{2}, -\frac{9}{2}, -\frac{3}{2}, -\frac{1}{2}, \frac{1}{2}, \frac{3}{2}, -\frac{3}{2})$\\
\hline
$20$ & $(1, 3, -3, -2, -1, 1, 2, 1)$ & $44$ & $(\frac{1}{2}, \frac{5}{2}, -\frac{7}{2}, -\frac{5}{2}, -\frac{3}{2}, -\frac{1}{2}, \frac{1}{2}, -\frac{3}{2})$\\
\hline
$21$ & $(0, 2, -4, -2, 0, 1, 2, 1)$ & $45$ & $(0, 1, -4, -2, -1, 0, 2, -2)$\\
\hline
$22$ & $(\frac{5}{2}, \frac{7}{2}, -\frac{5}{2}, -\frac{3}{2}, -\frac{1}{2}, \frac{1}{2}, \frac{3}{2}, \frac{1}{2})$ & $46$ & $(0, 2, -4, -2, -1, 0, 1, -2)$\\
\hline
$23$ & $(\frac{3}{2}, \frac{5}{2}, -\frac{7}{2}, -\frac{3}{2}, -\frac{1}{2}, \frac{1}{2}, \frac{5}{2}, \frac{1}{2})$ & $47$ & $(\frac{1}{2}, \frac{3}{2}, -\frac{7}{2}, -\frac{5}{2}, -\frac{3}{2}, -\frac{1}{2}, \frac{1}{2}, -\frac{5}{2})$\\
\hline
$24$ & $(-\frac{1}{2}, \frac{5}{2}, -\frac{7}{2}, -\frac{5}{2}, -\frac{3}{2}, -\frac{1}{2}, \frac{1}{2}, \frac{3}{2})$ &  & \\
\hline
\end{tabular}
\bigskip
\caption{Basic weights of $(E_8, 4, 4)$}
\end{table}

\begin{table}[H]
\begin{tabular}{|l|c|l|c|}
\hline
$1$ & $(0, 1, 2, -2, -1, 0, 1, 3)$ & $5$ & $(0, 1, 2, -3, -2, -1, 0, 1)$ \\
\hline
$2$ & $(0, 1, 2, -3, -1, 0, 1, 2)$ & $6$ & $(0, 1, 2, -3, -2, -1, 1, 0)$\\
\hline
$3$ & $(\frac{1}{2}, \frac{3}{2}, \frac{5}{2}, -\frac{5}{2}, -\frac{3}{2}, -\frac{1}{2}, \frac{1}{2}, \frac{3}{2})$ & $7$ & $(0, 1, 2, -3, -2, -1, 0, -1)$\\
\hline
$4$ & $(0, 1, 2, -3, -2, 0, 1, 1)$ &  & \\
\hline
\end{tabular}
\bigskip
\caption{Basic weights of $(E_8, 5, 5)$}
\end{table}

\begin{table}[H]
\begin{tabular}{|l|c|l|c|}
\hline
$1$ & $(0, 1, -2, -1, 0, 1, 2, 3)$ & $10$ & $(0, 2, -3, -2, -1, 0, 1, 1)$ \\
\hline
$2$ & $(\frac{1}{2}, \frac{3}{2}, -\frac{5}{2}, -\frac{3}{2}, -\frac{1}{2}, \frac{1}{2}, \frac{3}{2}, \frac{5}{2})$ & $11$ & $(-\frac{1}{2}, \frac{3}{2}, -\frac{7}{2}, -\frac{3}{2}, -\frac{1}{2}, \frac{1}{2}, \frac{3}{2}, \frac{1}{2})$\\
\hline
$3$ & $(0, 1, -3, -1, 0, 1, 2, 2)$ & $12$ & $(0, 1, -3, -2, -1, 1, 2, 0)$\\
\hline
$4$ & $(\frac{1}{2}, \frac{3}{2}, -\frac{5}{2}, -\frac{3}{2}, -\frac{1}{2}, \frac{1}{2}, \frac{5}{2}, \frac{3}{2})$ & $13$ & $(1, 2, -3, -2, -1, 0, 1, 0)$\\
\hline
$5$ & $(\frac{1}{2}, \frac{5}{2}, -\frac{5}{2}, -\frac{3}{2}, -\frac{1}{2}, \frac{1}{2}, \frac{3}{2}, \frac{3}{2})$ & $14$ & $(-1, 2, -3, -2, -1, 0, 1, 0)$\\
\hline
$6$ & $(0, 1, -3, -2, -1, 0, 1, 2)$ & $15$ & $(\frac{1}{2}, \frac{3}{2}, -\frac{7}{2}, -\frac{3}{2}, -\frac{1}{2}, \frac{1}{2}, \frac{3}{2}, -\frac{1}{2})$\\
\hline
$7$ & $(0, 1, -3, -2, 0, 1, 2, 1)$ & $16$ & $(0, 1, -3, -2, -1, 0, 2, -1)$\\
\hline
$8$ & $(0, 1, -3, -2, -1, 0, 2, 1)$ & $17$ & $(0, 2, -3, -2, -1, 0, 1, -1)$\\
\hline
$9$ & $(\frac{3}{2}, \frac{5}{2}, -\frac{5}{2}, -\frac{3}{2}, -\frac{1}{2}, \frac{1}{2}, \frac{3}{2}, \frac{1}{2})$ & $18$ & $(0, 1, -3, -2, -1, 0, 1, -2)$\\
\hline
\end{tabular}
\bigskip
\caption{Basic weights of $(E_8, 4, 5)$}
\end{table}

\begin{table}[H]
\begin{tabular}{|l|c|l|c|}
\hline
$1$ & $(0, 1, 2, -2, -1, 0, 2, 4)$ & $10$ & $(\frac{1}{2}, \frac{3}{2}, \frac{5}{2}, -\frac{7}{2}, -\frac{5}{2}, -\frac{1}{2}, \frac{1}{2}, \frac{3}{2})$ \\
\hline
$2$ & $(\frac{1}{2}, \frac{3}{2}, \frac{5}{2}, -\frac{5}{2}, -\frac{3}{2}, -\frac{1}{2}, \frac{1}{2}, \frac{7}{2})$ & $11$ & $(1, 2, 3, -3, -2, -1, 1, 1)$\\
\hline
$3$ & $(0, 1, 3, -3, -1, 0, 1, 3)$ & $12$ & $(\frac{1}{2}, \frac{3}{2}, \frac{5}{2}, -\frac{7}{2}, -\frac{5}{2}, -\frac{1}{2}, \frac{3}{2}, \frac{1}{2})$\\
\hline
$4$ & $(\frac{1}{2}, \frac{3}{2}, \frac{7}{2}, -\frac{5}{2}, -\frac{3}{2}, -\frac{1}{2}, \frac{1}{2}, \frac{5}{2})$ & $13$ & $(\frac{1}{2}, \frac{3}{2}, \frac{5}{2}, -\frac{7}{2}, -\frac{5}{2}, -\frac{3}{2}, \frac{1}{2}, \frac{1}{2})$\\
\hline
$5$ & $(\frac{1}{2}, \frac{3}{2}, \frac{5}{2}, -\frac{7}{2}, -\frac{3}{2}, -\frac{1}{2}, \frac{1}{2}, \frac{5}{2})$ & $14$ & $(0, 1, 2, -4, -2, -1, 2, 0)$\\
\hline
$6$ & $(0, 1, 2, -4, -1, 0, 2, 2)$ & $15$ & $(-\frac{1}{2}, \frac{3}{2}, \frac{5}{2}, -\frac{7}{2}, -\frac{5}{2}, -\frac{3}{2}, -\frac{1}{2}, \frac{1}{2})$\\
\hline
$7$ & $(0, 1, 2, -4, -2, 0, 1, 2)$ & $16$ & $(-\frac{1}{2}, \frac{3}{2}, \frac{5}{2}, -\frac{7}{2}, -\frac{5}{2}, -\frac{3}{2}, \frac{1}{2}, -\frac{1}{2})$\\
\hline
$8$ & $(0, 1, 2, -4, -2, -1, 0, 2)$ & $17$ & $(\frac{1}{2}, \frac{3}{2}, \frac{5}{2}, -\frac{7}{2}, -\frac{5}{2}, -\frac{3}{2}, -\frac{1}{2}, -\frac{1}{2})$\\
\hline
$9$ & $(0, 1, 2, -4, -2, 0, 2, 1)$ & $18$ & $(0, 1, 2, -4, -2, -1, 0, -2)$\\
\hline
\end{tabular}
\bigskip
\caption{Basic weights of $(E_8, 5, 4)$}
\end{table}

\subsubsection{$F_4$}
\begin{theorem}\label{bwthm6}
Any basic weight of type $F_4$ must be one of the following cases $($up to a positive integer$)$.
\begin{itemize}
\item [(1)] $(F_4, 2, 2)$. $(2, 0, -1, 1)$, $(1, 0, -2, 1)$, $(0, -1, -2, 1)$;

\item [(2)] $(F_4, 2, 3)$. $(\frac{3}{2}, \frac{1}{2}, -\frac{1}{2}, \frac{1}{2})$, $(1, 0, -1, 1)$, $(\frac{1}{2}, \frac{1}{2}, -\frac{3}{2}, \frac{1}{2})$, $(\frac{1}{2}, -\frac{1}{2}, -\frac{3}{2}, \frac{1}{2})$, $(-\frac{1}{2}, -\frac{1}{2}, -\frac{3}{2}, \frac{1}{2})$;

\item [(3)] $(F_4, 3, 2)$. $(2, 1, 0, -1)$, $(1, 1, 0, -2)$, $(1, 0, -1, -2)$, $(0, 1, -1, -2)$, $(-1, 0, -1, -2)$;

\item [(4)] $(F_4, 3, 3)$. $(1, 1, 0, -1)$, $(\frac{1}{2}, \frac{1}{2}, -\frac{1}{2}, -\frac{3}{2})$, $(-\frac{1}{2}, \frac{1}{2}, -\frac{1}{2}, -\frac{3}{2})$.
\end{itemize}
\end{theorem}
\begin{proof}
It can be verified that all the weights in the theorem are basic weights contained in $W\varpi_j$ for the corresponding basic systems $(F_4, i, j)$. In view of Table 2 in \cite{P1}, these weights exhaust all the possibilities. We can also prove this theorem in a case-by-case fashion (only a few cases) using Lemma \ref{blem3} (details are omitted).
\end{proof}

\subsubsection{$G_2$}

\begin{theorem}\label{bwthm7}
Any basic weight of type $G_2$ must be one of the following cases $($up to a positive integer$)$.
$(1)$ $(G_2, 1, 1)$. $(-1, 1, 0)$, $(-1, 0, 1)$; $(2)$ $(G_2, 1, 2)$. $(-1, 2, -1)$, $(-1, -1, 2)$, $(-2, 1, 1)$; $(3)$ $(G_2, 2, 1)$. $(1, 0, -1)$, $(1, -1, 0)$, $(0, -1, 1)$; $(4)$ $(G_2, 2, 2)$. $(1, -2, 1)$, $(2, -1, -1)$.
\end{theorem}
\begin{proof}
Recall that $\alpha_1=e_1-e_2$ and $\alpha_2=-2e_1+e_2+e_3$ for $G_2$ (see \S 12.1, \cite{H1}).

$(1)$ $(G_2, 1, 1)$. In this case, $I=J=\{\alpha_2\}$, we have ${}^IW^J=\{s_{\alpha_1}, s_{\alpha_1}s_{\alpha_2}s_{\alpha_1}\}$ and $\lambda=s_{\alpha_1}\varpi_2=(-1, 0, 1)$ or $s_{\alpha_1}s_{\alpha_2}s_{\alpha_1}\varpi_2=(-1, 1, 0)$, where $\varpi_2=(0, -1, 1)$;

$(2)$ $(G_2, 1, 2)$. With $I=\{\alpha_2\}$ and $J=\{\alpha_1\}$, we have ${}^IW^J=\{1, s_{\alpha_1}s_{\alpha_2}, s_{\alpha_1}s_{\alpha_2}s_{\alpha_1}s_{\alpha_2}\}$ and $\lambda=\varpi_1=(-1, -1, 2)$ or $s_{\alpha_1}s_{\alpha_2}\varpi_1=(-2, 1, 1)$ or $s_{\alpha_1}s_{\alpha_2}s_{\alpha_1}s_{\alpha_2}\varpi_1=(-1, 2, -1)$;

$(3)$ $(G_2, 2, 1)$. With $I=\{\alpha_1\}$ and $J=\{\alpha_2\}$, we have ${}^IW^J=\{1, s_{\alpha_2}s_{\alpha_1}, s_{\alpha_2}s_{\alpha_1}s_{\alpha_2}s_{\alpha_1}\}$ and $\lambda=\varpi_2=(0, -1, 1)$ or $s_{\alpha_2}s_{\alpha_1}\varpi_2=(1, -1, 0)$ or $s_{\alpha_2}s_{\alpha_1}s_{\alpha_2}s_{\alpha_1}\varpi_2=(1, 0, -1)$;

$(4)$ $(G_2, 2, 2)$. With $I=J=\{\alpha_1\}$, we get ${}^IW^J=\{s_{\alpha_2}, s_{\alpha_2}s_{\alpha_1}s_{\alpha_2}\}$ and $\lambda=s_{\alpha_2}\varpi_1=(1, -2, 1)$ or $s_{\alpha_2}s_{\alpha_1}s_{\alpha_2}\varpi_1=(2, -1, -1)$.
\end{proof}

%
%
\section{Jantzen coefficients of basic generalized Verma modules}
%
%

\subsection{Nonzero Jantzen coefficients}In this subsection, we give all the nonzero Jantzen coefficients of basic generalized Verma modules. A basic weight $\lambda$ is called \emph{standard} if $\lambda\in W\varpi_j$ for some $1\leq j\leq n$. In view of Lemma \ref{sjlem2} and Lemma \ref{blem1}, it suffices to consider Jantzen coefficients associated with standard basic weights. Let $\lambda^1, \cdots, \lambda^l$ be all the standard basic weights of a basic system (we adopt the ordering in \S5.2, with $i<j$ whenever $\lambda^i>\lambda^j$). Write $c_{i, j}=c(\lambda^i, \lambda^j)$.

\begin{theorem}\label{nzj}
The Jantzen coefficients of a basic system $(\Phi, i, j)$ are vanished unless $(\Phi, i, j)=(A_1, 1, 1)$, $(B_3, 2, 2)$ and $(C_3, 2, 2)$, $(E_7, 4, 4)$, $(E_8, 4, 5)$, $(E_8, 5, 4)$ and $(E_8, 4, 4)$. In these exceptional cases, all the nonzero Jantzen coefficients are given in Table \ref{njtb1}-\ref{njtb5}.
\end{theorem}

For type $E$, we have to run computer programs to calculate the corresponding Jantzen coefficients. If all the Jantzen coefficients of a system $(\Phi, \Phi_I, \Phi_J)$ are vanished, then the subcategory associated with the system is semisimple).

\begin{table}[htbp]
\begin{tabular}{|c|c|}
\hline
$i$ & $\{j\mid c_{i, j}=1\}$ \\
\hline
$1$ & $2$ \\
\hline
\end{tabular}
\bigskip
\caption{Nonzero Jantzen coefficients of $(A_1, 1, 1)$, $(B_3, 2, 2)$ and $(C_3, 2, 2)$}
\label{njtb1}
\end{table}

\renewcommand\arraystretch{1.3}
\begin{table}[htbp]
\begin{tabular}{|c|c|c|c|}
\hline
$i$ & $\{j\mid c_{i, j}=1\}$ & $\{j\mid c_{i, j}=-1\}$ & $\{j\mid c_{i, j}=2\}$ \\
\hline
$1$ & $3$ & $5$ & $6$ \\
\hline
$2$ & $3$ & $6$ & $5$ \\
\hline
$3$ &  &  & $4$ \\
\hline
$4$ & $5, 6$ &  & \\
\hline
\end{tabular}
\bigskip
\caption{Nonzero Jantzen coefficients of $(E_7, 4, 4)$}
\label{njtb2}
\end{table}

\renewcommand\arraystretch{1.3}
\begin{table}[htbp]
\begin{tabular}{|c|c|c|c|c|c|}
\hline
$i$ & $\{j\mid c_{i, j}=1\}$ & $\{j\mid c_{i, j}=-1\}$ & $i$ & $\{j\mid c_{i, j}=1\}$ & $\{j\mid c_{i, j}=-1\}$\\
\hline
$1$ & $2, 3, 10, 11, 14, 16$ & $4, 6, 12, 13, 18$ & $10$ & $11, 12, 15, 18$ & $14, 17$\\
\hline
$2$ & $3, 4, 8, 9, 13$ & $5, 10, 14, 15$ & $11$ & $13, 17$ & $15, 16$\\
\hline
$3$ & $5, 6, 12, 15, 18$ & $8, 9, 11, 16$ & $12$ & $13, 14$ & $15$\\
\hline
$4$ & $6, 9, 10, 16$ & $7, 8, 12, 17$ & $13$ & $15, 16$ & $18$\\
\hline
$5$ & $7, 9, 11, 13, 18$ & $10, 12, 17$ & $14$ & $16$ & $17$\\
\hline
$6$ & $7, 8, 11, 14, 17$ & $9, 13, 18$ & $15$ & $17$ & $18$\\
\hline
$7$ & $9, 10, 16$ & $11, 14, 15, 18$ & $16$ & $17, 18$ & \\
\hline
$8$ & $10, 13, 14, 18$ & $11, 12, 16$ & $17$ & $18$ & \\
\hline
$9$ & $12, 14, 15, 17$ & $13, 16$ &  &  & \\
\hline
\end{tabular}
\bigskip
\caption{Nonzero Jantzen coefficients of $(E_8, 5, 4)$}
\label{njtb3}
\end{table}

\renewcommand\arraystretch{1.3}
\begin{table}[htbp]
\begin{tabular}{|c|c|c|c|c|c|}
\hline
$i$ & $\{j\mid c_{i, j}=1\}$ & $\{j\mid c_{i, j}=-1\}$ & $i$ & $\{j\mid c_{i, j}=1\}$ & $\{j\mid c_{i, j}=-1\}$\\
\hline
$1$ & $2, 4, 10, 12, 14, 18$ & $3, 6, 11, 13, 16$ & $10$ & $11, 12, 14, 15, 17, 18$ & \\
\hline
$2$ & $3, 4, 5, 8, 10, 14$ & $9, 13, 15$ & $11$ & $13, 15, 16, 17$ & \\
\hline
$3$ & $5, 6, 8, 11, 16$ & $9, 12, 15, 18$ & $12$ & $13, 14, 15$ & \\
\hline
$4$ & $6, 7, 8, 9, 10, 12$ & $16, 17$ & $13$ & $15, 16$ & $18$\\
\hline
$5$ & $7, 10, 11, 17, 18$ & $9, 12, 13$ & $14$ & $16, 17$ & \\
\hline
$6$ & $7, 8, 9, 11, 13$ & $14, 17, 18$ & $15$ & $17, 18$ & \\
\hline
$7$ & $9, 10, 11, 15, 18$ & $14, 16$ & $16$ & $17$ & $18$\\
\hline
$8$ & $10, 11, 12, 13, 14, 16$ & $18$ & $17$ & $18$ & \\
\hline
$9$ & $12, 13, 15$ & $14, 16, 17$ &  &  & \\
\hline
\end{tabular}
\bigskip
\caption{Nonzero Jantzen coefficients of $(E_8, 4, 5)$}
\label{njtb4}
\end{table}

\begin{landscape}
\renewcommand\arraystretch{1.3}
\begin{table}[htbp]\footnotesize
\begin{tabular}{|c|c|c|c|c|c|}
\hline
$i$ & $\{j\mid c_{i, j}=1\}$ & $\{j\mid c_{i, j}=-1\}$ & $i$ & $\{j\mid c_{i, j}=1\}$ & $\{j\mid c_{i, j}=-1\}$\\
\hline
$1$ & $12, 15, 16, 22, 32, 33, 46$ & $6, 7, 14, 20, 21, 26, 41, 42$ & $25$ & $27, 28, 29, 32, 33, 46$ & $31, 38, 39$\\
\hline
$2$ & $3, 4, 5, 7, 10, 24, 25, 28, 33, 39, 47$ & $6, 9, 11, 27, 32, 38, 40$ & $26$ & $28, 30, 31, 32, 33, 34, 35, 42, 44, 45$ & $43, 47$\\
\hline
$3$ & $6, 7, 14, 17, 26, 36$ & $12, 19, 22, 34$ & $27$ & $29, 30, 34, 36, 37, 41$ & $33, 38, 39, 47$\\
\hline
$4$ & $6, 7, 14, 17, 26, 36$ & $12, 19, 22, 34$ & $28$ & $30, 31, 32, 34, 36, 37, 38, 39$ & $42, 47$\\
\hline
$5$ & $6, 7, 15, 16, 22, 31, 32, 33$ & $9, 10, 26, 29, 38, 39$ & $29$ & $32, 34, 39, 44, 45$ & $35, 37, 40$\\
\hline
$6$ & $8, 11, 13, 14, 20, 22, 24, 35, 41$ & $9, 10, 16, 23, 36, 39, 47$ & $30$ & $32, 33, 34, 38, 39$ & $36$\\
\hline
$7$ & $8, 9, 10, 11, 12, 13, 15, 24, 26, 35, 38$ & $21, 23, 34, 42, 47$ & $31$ & $33, 35, 37, 38, 40$ & $36, 44, 45$\\
\hline
$8$ & $9, 10, 12, 15, 16, 19, 31$ & $14, 17, 29, 46$ & $32$ & $34, 35, 37, 40, 42, 43, 47$ & $38$\\
\hline
$9$ & $11, 12, 13, 18, 19, 21, 28, 42$ & $15, 20, 24, 25, 32, 43, 46$ & $33$ & $35, 36, 37, 39, 40, 43, 47$ & $41$\\
\hline
$10$ & $11, 13, 14, 16, 17, 18, 21, 33, 46$ & $20, 25, 24, 27, 41, 43$ & $34$ & $38, 40, 42, 47$ & $44, 45$\\
\hline
$11$ & $15, 16, 19, 20, 21, 31, 34$ & $17, 27, 28, 29, 36$ & $35$ & $38, 39, 41, 42$ & \\
\hline
$12$ & $15, 17, 18, 20, 21, 23, 26, 34, 37$ & $24, 35, 42, 44, 45$ & $36$ & $39, 41, 44, 45$ & $40, 47$\\
\hline
$13$ & $15, 16, 19, 26, 36, 41, 42$ & $17, 22, 34$ & $37$ & $38, 39, 41, 42$ & $46$\\
\hline
$14$ & $16, 20, 21, 24, 35, 44, 45$ & $18, 19, 22, 23, 36, 37, 41$ & $38$ & $40, 42$ & $41, 43, 46$\\
\hline
$15$ & $17, 18, 21, 22, 24, 25, 26, 27, 28, 43, 47$ & $38, 46$ & $39$ & $40, 42, 46$ & $41, 43$\\
\hline
$16$ & $18, 19, 22, 24, 25, 26, 39, 43, 46, 47$ & $20, 27, 28$ & $40$ & $41, 42$ & \\
\hline
$17$ & $20, 23, 25, 28, 29, 30$ & $22, 37, 40, 43$ & $41$ & $43, 44, 45$ & $46, 47$\\
\hline
$18$ & $20, 21, 26, 27, 28, 29$ & $22, 31$ & $42$ & $43, 44, 45, 46$ & $47$\\
\hline
$19$ & $21, 26, 27, 37, 40, 43$ & $23, 25, 30, 31$ & $43$ & $46$ & \\
\hline
$20$ & $22, 23, 25, 27, 30, 31, 32$ & $33, 37, 39, 46$ & $44$ & $46$ & \\
\hline
$21$ & $23, 25, 26, 29, 30, 32, 38, 46$ & $28, 33, 37$ & $45$ & $46$ & \\
\hline
$22$ & $27, 32, 33, 41, 43, 47$ & $29, 30, 35, 36, 44, 45$ & i & $\{j\mid c_{i, j}=2\}$ & \\
\hline
$23$ & $27, 28, 29, 34$ & $31, 36, 41, 42$ & $1$ & $2$ & \\
\hline
$24$ & $32, 33, 36, 41, 42, 46$ & $34, 38, 39$ & $46$ & $47$ & \\
\hline
\end{tabular}
\bigskip
\caption{Nonzero Jantzen coefficents of $(E_8, 4, 4)$}
\label{njtb5}
\end{table}
\end{landscape}

\begin{theorem}\label{njthm1}
Jantzen coefficients $|c(\lambda, \mu)|\leq1$ except the cases when $\Phi=E_7, E_8$. In these exceptional cases, $|c(\lambda, \mu)|\leq 2$.
\end{theorem}
\begin{proof}
If $c(\lambda, \mu)\neq0$ for $\lambda>\mu$, Lemma \ref{sjlem1} yields $\beta\in\Psi_\lambda^{++}$ with $\sgn(s_\beta\lambda, \mu)\neq0$ for some $\mu\in\Lambda_I^+$. The theorem then follows from Lemma \ref{invjlem} and Table \ref{njtb1}-\ref{njtb5}.
\end{proof}

\subsection{Posets of the basic systems} For any nonzero Jantzen coefficient $c(\lambda^i, \lambda^j)$, if there exists no sequence $i=i_0<i_1<\cdots<i_k=j$ such that $c(\lambda^{i_{t-1}}, \lambda^{i_{t}})\neq0$ for $1\leq t\leq k$ and $k>1$, we say $\lambda^i$ and $\lambda^j$ are \emph{adjacent}. Connect $i$ and $j$ when $\lambda^i$ and $\lambda^j$ are adjacent. This gives us the posets in Figure \ref{pbfig1}-\ref{pbfig4}.

\begin{figure}[htbp]
\setlength{\unitlength}{1.1mm}
\begin{center}
\begin{picture}(0,20)
\put(0,5){\circle{5}}
\put(0,15){\circle{5}}

\put(0,7.5){\line(0,1){5}}

\put(-0.8,4){$2$}
\put(-0.8,14){$1$}
\end{picture}
\end{center}
\caption{Poset for $(A_1, 1, 1)$, $(B_3, 2, 2)$ and $(C_3, 2, 2)$}
\label{pbfig1}
\end{figure}
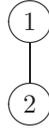

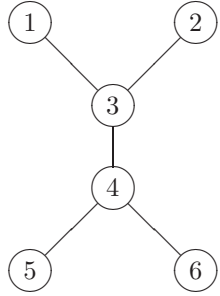
\begin{figure}[htbp]
\setlength{\unitlength}{1.1mm}
\begin{center}
\begin{picture}(0,40)

\put(-10,5){\circle{5}} \put(10,5){\circle{5}}
\put(0,15){\circle{5}}
\put(0,25){\circle{5}}
\put(-10,35){\circle{5}} \put(10,35){\circle{5}}

\put(-8.15,6.85){\line(1,1){6.3}} \put(8.15,6.85){\line(-1,1){6.3}}
\put(0,17.5){\line(0,1){5}}
\put(1.85,26.85){\line(1,1){6.3}} \put(-1.85,26.85){\line(-1,1){6.3}}

\put(-10.8, 4){$5$}\put(9.2, 4){$6$}
\put(-0.8,14){$4$}
\put(-0.8,24){$3$}
\put(-10.8, 34){$1$}\put(9.2, 34){$2$}

\end{picture}
\end{center}
\caption{Poset for $(E_7, 4, 4)$}
\label{pbfig2}
\end{figure}

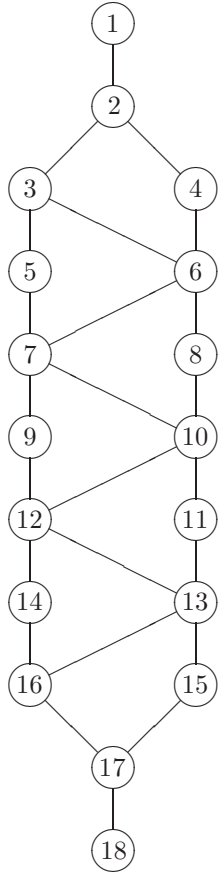
\begin{figure}[htbp]
\setlength{\unitlength}{1.1mm}
\begin{center}
\begin{picture}(0,110)

\put(0,5){\circle{5}}
\put(0,15){\circle{5}}
\put(-10,25){\circle{5}} \put(10,25){\circle{5}}
\put(-10,35){\circle{5}} \put(10,35){\circle{5}}
\put(-10,45){\circle{5}} \put(10,45){\circle{5}}
\put(-10,55){\circle{5}} \put(10,55){\circle{5}}
\put(-10,65){\circle{5}} \put(10,65){\circle{5}}
\put(-10,75){\circle{5}} \put(10,75){\circle{5}}
\put(-10,85){\circle{5}} \put(10,85){\circle{5}}
\put(0,95){\circle{5}}
\put(0,105){\circle{5}}

\put(0,7.5){\line(0,1){5}}
\put(1.85,16.85){\line(1,1){6.3}} \put(-1.85,16.85){\line(-1,1){6.3}}
\put(-10,27.5){\line(0,1){5}} \put(-7.7,26.15){\line(2,1){15.4}}\put(10,27.5){\line(0,1){5}}
\put(-10,37.5){\line(0,1){5}}\put(7.7,36.15){\line(-2,1){15.4}} \put(10,37.5){\line(0,1){5}}
\put(-10,47.5){\line(0,1){5}} \put(-7.7,46.15){\line(2,1){15.4}}\put(10,47.5){\line(0,1){5}}
\put(-10,57.5){\line(0,1){5}}\put(7.7,56.15){\line(-2,1){15.4}} \put(10,57.5){\line(0,1){5}}
\put(-10,67.5){\line(0,1){5}} \put(-7.7,66.15){\line(2,1){15.4}}\put(10,67.5){\line(0,1){5}}
\put(-10,77.5){\line(0,1){5}}\put(7.7,76.15){\line(-2,1){15.4}} \put(10,77.5){\line(0,1){5}}
\put(-8.15,86.85){\line(1,1){6.3}} \put(8.15,86.85){\line(-1,1){6.3}}
\put(0,97.5){\line(0,1){5}}

\put(-1.5,4.0){$18$}
\put(-1.7,14){$17$}
\put(-11.7,24){$16$}\put(8.3,24){$15$}
\put(-11.7,34){$14$}\put(8.3,34){$13$}
\put(-11.7,44){$12$}\put(8.3,44){$11$}
\put(-10.8,54){$9$}\put(8.3,54){$10$}
\put(-10.8,64){$7$}\put(9.2,64){$8$}
\put(-10.8,74){$5$}\put(9.2,74){$6$}
\put(-10.8,84){$3$}\put(9.2,84){$4$}
\put(-0.6,94.0){$2$}
\put(-0.8,104){$1$}

\end{picture}
\end{center}
\caption{Poset for $(E_8, 4, 5)$ and $(E_8, 5, 4)$}
\label{pbfig3}
\end{figure}

\begin{figure}[htbp]
\setlength{\unitlength}{1.1mm}
\begin{center}
\begin{picture}(0,210)

\put(0,5){\circle{5}}
\put(0,15){\circle{5}}
\put(-20,25){\circle{5}} \put(0,25){\circle{5}} \put(20,25){\circle{5}}
\put(-10,35){\circle{5}} \put(10,35){\circle{5}}
\put(0,45){\circle{5}}
\put(-10,55){\circle{5}} \put(10,55){\circle{5}}
\put(-30,65){\circle{5}}\put(-10,65){\circle{5}} \put(10,65){\circle{5}}\put(30,65){\circle{5}}
\put(-10,75){\circle{5}} \put(10,75){\circle{5}}
\put(-20,85){\circle{5}} \put(0,85){\circle{5}} \put(20,85){\circle{5}}
\put(-10,95){\circle{5}} \put(10,95){\circle{5}}
\put(-30,105){\circle{5}}\put(-10,105){\circle{5}} \put(10,105){\circle{5}}\put(30,105){\circle{5}}\put(60,105){\circle{5}}
\put(-10,115){\circle{5}} \put(10,115){\circle{5}}
\put(-20,125){\circle{5}} \put(0,125){\circle{5}} \put(20,125){\circle{5}}
\put(-10,135){\circle{5}} \put(10,135){\circle{5}}
\put(-30,145){\circle{5}}\put(-10,145){\circle{5}} \put(10,145){\circle{5}}\put(30,145){\circle{5}}
\put(-10,155){\circle{5}} \put(10,155){\circle{5}}
\put(0,165){\circle{5}}
\put(-10,175){\circle{5}} \put(10,175){\circle{5}}
\put(-20,185){\circle{5}} \put(0,185){\circle{5}} \put(20,185){\circle{5}}
\put(0,195){\circle{5}}
\put(0,205){\circle{5}}

\put(0,7.5){\line(0,1){5}}
\put(-2.3,16.15){\line(-2,1){15.4}} \put(0,17.5){\line(0,1){5}} \put(2.3,16.15){\line(2,1){15.4}}
\put(-18.15,26.85){\line(1,1){6.3}} \put(-17.54,25.82){\line(3,1){25.1}}
\put(1.85,26.85){\line(1,1){6.3}} \put(-1.85,26.85){\line(-1,1){6.3}}
\put(17.54,25.82){\line(-3,1){25.1}} \put(18.15,26.85){\line(-1,1){6.3}}
\put(-8.15,36.85){\line(1,1){6.3}} \put(8.15,36.85){\line(-1,1){6.3}}
\put(1.85,46.85){\line(1,1){6.3}} \put(-1.85,46.85){\line(-1,1){6.3}}
\put(-12.3,56.15){\line(-2,1){15.4}} \put(-10,57.5){\line(0,1){5}} \put(-7.7,56.15){\line(2,1){15.4}}
\put(7.7,56.15){\line(-2,1){15.4}} \put(10,57.5){\line(0,1){5}} \put(12.3,56.15){\line(2,1){15.4}}
\put(-27.7,66.15){\line(2,1){15.4}} \put(-10,67.5){\line(0,1){5}} \put(-7.7,66.15){\line(2,1){15.4}}
\put(7.7,66.15){\line(-2,1){15.4}} \put(10,67.5){\line(0,1){5}} \put(27.7,66.15){\line(-2,1){15.4}}
\put(-11.85,76.85){\line(-1,1){6.3}} \put(-8.15,76.85){\line(1,1){6.3}}
\put(8.15,76.85){\line(-1,1){6.3}} \put(11.85,76.85){\line(1,1){6.3}}
\put(-18.15,86.85){\line(1,1){6.3}}
\put(1.85,86.85){\line(1,1){6.3}} \put(-1.85,86.85){\line(-1,1){6.3}}
\put(18.15,86.85){\line(-1,1){6.3}}
\put(-12.3,96.15){\line(-2,1){15.4}} \put(-10,97.5){\line(0,1){5}} \put(-7.7,96.15){\line(2,1){15.4}}
\put(7.7,96.15){\line(-2,1){15.4}} \put(10,97.5){\line(0,1){5}} \put(12.3,96.15){\line(2,1){15.4}}
\put(-27.7,106.15){\line(2,1){15.4}} \put(-10,107.5){\line(0,1){5}} \put(-7.7,106.15){\line(2,1){15.4}}
\put(7.7,106.15){\line(-2,1){15.4}} \put(10,107.5){\line(0,1){5}} \put(27.7,106.15){\line(-2,1){15.4}}
\put(-11.85,116.85){\line(-1,1){6.3}} \put(-8.15,116.85){\line(1,1){6.3}}
\put(8.15,116.85){\line(-1,1){6.3}} \put(11.85,116.85){\line(1,1){6.3}}
\put(-18.15,126.85){\line(1,1){6.3}}
\put(1.85,126.85){\line(1,1){6.3}} \put(-1.85,126.85){\line(-1,1){6.3}}
\put(18.15,126.85){\line(-1,1){6.3}}
\put(-12.3,136.15){\line(-2,1){15.4}} \put(-10,137.5){\line(0,1){5}} \put(-7.7,136.15){\line(2,1){15.4}}
\put(7.7,136.15){\line(-2,1){15.4}} \put(10,137.5){\line(0,1){5}} \put(12.3,136.15){\line(2,1){15.4}}
\put(-27.7,146.15){\line(2,1){15.4}} \put(-10,147.5){\line(0,1){5}} \put(-7.7,146.15){\line(2,1){15.4}}
\put(7.7,146.15){\line(-2,1){15.4}} \put(10,147.5){\line(0,1){5}} \put(27.7,146.15){\line(-2,1){15.4}}
\put(-8.15,156.85){\line(1,1){6.3}} \put(8.15,156.85){\line(-1,1){6.3}}
\put(1.85,166.85){\line(1,1){6.3}} \put(-1.85,166.85){\line(-1,1){6.3}}
\put(-11.85,176.85){\line(-1,1){6.3}} \put(-8.15,176.85){\line(1,1){6.3}}\put(-7.54,175.82){\line(3,1){25.1}}
\put(8.15,176.85){\line(-1,1){6.3}} \put(11.85,176.85){\line(1,1){6.3}}\put(7.54,175.82){\line(-3,1){25.1}}
\put(17.7,186.15){\line(-2,1){15.4}} \put(0,187.5){\line(0,1){5}} \put(-17.7,186.15){\line(2,1){15.4}}
\put(0,197.5){\line(0,1){5}}

\qbezier(12.7,75)(29,79)(58,103.5)
\qbezier(-7.3,75)(29,79)(58,103.5)
\qbezier(12.7,135)(29,131)(58,106.5)
\qbezier(-7.3,135)(29,131)(58,106.5)

\put(-1.5,4.0){$47$}
\put(-1.7,14){$46$}
\put(-21.7,24){$43$} \put(-1.7,24){$44$} \put(18.3,24){$45$}
\put(-11.7,34){$41$}\put(8.3,34){$42$}
\put(-1.7,44){$40$}
\put(-11.7,54){$38$}\put(8.3,54){$39$}
\put(-31.7,64){$34$}\put(-11.7,64){$35$}\put(8.3,64){$37$}\put(28.3,64){$36$}
\put(-11.7,74){$32$}\put(8.3,74){$33$}
\put(-21.7,84){$29$} \put(-1.7,84){$30$} \put(18.3,84){$31$}
\put(-11.7,94){$27$}\put(8.3,94){$28$}
\put(-31.7,104){$22$}\put(-11.7,104){$23$}\put(8.3,104){$25$}\put(28.3,104){$26$}\put(58.3,104){$24$}
\put(-11.7,114){$20$}\put(8.3,114){$21$}
\put(-21.7,124){$17$} \put(-1.7,124){$18$} \put(18.3,124){$19$}
\put(-11.7,134){$15$}\put(8.3,134){$16$}
\put(-31.7,144){$12$}\put(-11.7,144){$11$}\put(8.3,144){$13$}\put(28.3,144){$14$}
\put(-10.8,154){$9$}\put(8.3,154){$10$}
\put(-0.8,164){$8$}
\put(-10.8,174){$6$}\put(9.2,174){$7$}
\put(-20.8,184){$3$} \put(-0.8,184){$4$} \put(19.2,184){$5$}
\put(-0.8,194){$2$}\put(-0.8,204){$1$}

\end{picture}
\end{center}
\caption{Poset for $(E_8, 4, 4)$}
\label{pbfig4}
\end{figure}

\begin{remark}\label{sbrmk1}
It will be proved in \cite{XZ} that the above posets are actually the $\Ext^1$-posets (see \cite{BH}) of the corresponding basic systems.
\end{remark}

\subsection{Simplicity of basic generalized Verma modules} By Jantzen's simplicity criterion, one gets the following result from the data of Jantzen coefficients.

\begin{theorem}\label{sbthm1}
Let $\lambda$ be a basic weight of a basic system $(\Phi, i, j)$. Then $M_I(\lambda)$ is not simple if and only if one of the following conditions are satisfied $(k\in\bbZ^{>0})$:
\begin{itemize}
\item [(1)] $(A_1, 1, 1)$, with $\lambda=k(\frac{1}{2}, -\frac{1}{2})$;

\item [(2)] $(B_3, 2, 2)$, with $\lambda=k(1, 0, 1)$;

\item [(3)] $(C_3, 2, 2)$, with $\lambda=k(1, 0, 1)$;

\item [(4)] $(E_7, 4, 4)$, with $\lambda\neq k(0, 1, -3, -1, 0, 1, 0, 0)$, $k(\frac{1}{2}, \frac{3}{2}, -\frac{5}{2}, -\frac{3}{2}, -\frac{1}{2}, \frac{1}{2}, \frac{1}{2}, -\frac{1}{2})$;

\item [(5)] $(E_8, 4, 4)$, with $\lambda\neq k(\frac{1}{2}, \frac{3}{2}, -\frac{7}{2}, -\frac{5}{2}, -\frac{3}{2}, -\frac{1}{2}, \frac{1}{2}, -\frac{5}{2})$;

\item [(6)] $(E_8, 4, 5)$, with $\lambda\neq k(0, 1, -3, -2, -1, 0, 1, -2)$;

\item [(7)] $(E_8, 5, 4)$, with $\lambda\neq k(0, 1, 2, -4, -2, -1, 0, -2)$.
\end{itemize}
\end{theorem}

\begin{remark}\label{sbrmk2}
If $\lambda\in\Lambda_I^+$ and $\beta\in\Psi_\lambda^{++}$, then $s_\beta\lambda$ is $\Phi_I$-regular. There exists $w\in W_I$ so that $ws_\beta\lambda\in\Lambda_I^+$. The category $\caO_\lambda^{\frp_I}$ has at least two highest weights $\lambda-\rho$ and $ws_\beta\lambda-\rho$. If $\Phi$ is of classical type, the category $\caO_{\varpi_j}^{\frp_I}$ of $(\Phi, i, j)$ is semisimple unless $(\Phi, i, j)=(A_1, 1, 1)$, $(B_3, 2, 2)$ or $(C_3, 2, 2)$. In view of Theorem \ref{bwthm1}-\ref{bwthm4}, we give all the basic weights (up to a positive integer) with $\Psi_\lambda^{++}\neq\emptyset$ in Table \ref{sbtb1}.
\begin{table}[htbp]\footnotesize
\begin{tabular}{|c|c|c|c|c|}
\hline
basic system & basic weight $\lambda$ & \centering $\Psi_\lambda^{++}$ & \centering $\sum_{\beta\in\Psi_\lambda^{++}}\theta(s_\beta\lambda)$ & simple  \\
\hline
$(A_1, 1, 1)$ & $(\frac{1}{2}, -\frac{1}{2})$ & $e_1-e_2$ & $\neq0$ & No \\
\hline
$(B_2, 1, 2)$ & $(\frac{1}{2}, \frac{1}{2})$ & $e_1$, $e_1+e_2$ & 0 & Yes \\
\hline
$(B_2, 2, 1)$ & $(1, 0)$ & $e_1$, $e_1+e_2$ & 0 & Yes \\
\hline
$(B_3, 2, 2)$ & $(1, 0, 1)$ & $e_1$, $e_1+e_2$, $e_1+e_3$ & $\neq0$ & No \\
\hline
$(C_2, 1, 2)$ & $(1, 1)$ & $2e_1$, $e_1+e_2$ & 0 & Yes \\
\hline
$(C_2, 2, 1)$ & $(1, 0)$ & $2e_1$, $e_1+e_2$ & 0 & Yes \\
\hline
$(C_3, 2, 2)$ & $(1, 0, 1)$ & $2e_1$, $e_1+e_2$, $e_1+e_3$ & $\neq0$ & No \\
\hline
$(D_4, 2, 2)$ & $(1, 0, 1, 0)$ & $e_1+e_2$, $e_1+e_3$ & 0 & Yes \\
\hline
\end{tabular}
\bigskip
\caption{Classical basic systems with nonempty $\Psi_\lambda^{++}$}
\label{sbtb1}
\end{table}
\end{remark}

When $\lambda\in\Lambda_I^+$ and $\beta\in\Psi_\lambda^+$ are fixed, denote by $(\Phi(\beta), i_\beta, j_\beta)$ the basic system corresponding to the basic generalized Verma module $M(\lambda|_{\Phi(\beta)}, \Phi_I\cap \Phi(\beta), \Phi(\beta))$.

\begin{cor}\label{sbcor1}
Let $\Phi$ be a classical root system. Choose $\lambda\in\Lambda_I^+$ and $\beta\in\Psi_\lambda^{+}$.
\begin{itemize}
  \item [(1)] If $(\Phi(\beta), i_\beta, j_\beta)\simeq(B_2, 1, 2)$, $(B_2, 2, 1)$, $(C_2, 1, 2)$, $(C_2, 2, 1)$ or $(D_4, 2, 2)$, then $M(\lambda|_{\Phi(\beta)}, \Phi_I\cap \Phi(\beta), \Phi(\beta))$ is always simple and $|\Psi_\lambda^{++}\cap\Phi(\beta)|=2$.
  \item [(2)] If $(\Phi(\beta), i_\beta, j_\beta)\simeq(A_1, 1, 1)$, $(B_3, 2, 2)$ or $(C_3, 2, 2)$, then $M(\lambda|_{\Phi(\beta)}, \Phi_I\cap \Phi(\beta), \Phi(\beta))$ is not simple and $|\Psi_\lambda^{++}\cap\Phi(\beta)|=1, 3$.
  \item [(3)] In all other cases, $M(\lambda|_{\Phi(\beta)}, \Phi_I\cap \Phi(\beta), \Phi(\beta))$ is simple and $|\Psi_\lambda^{++}\cap\Phi(\beta)|=0$. In particular, $\beta\not\in\Psi_\lambda^{++}$.
\end{itemize}
\end{cor}
\begin{proof}
It follows from Lemma \ref{redslem5} that $\Psi_\lambda^{++}\cap\Phi(\beta)=\Psi^{++}(\lambda|_{\Phi(\beta)}, \Phi_I\cap\Phi(\beta), \Phi(\beta))$. Then the corollary follows from Remark \ref{sbrmk2} and Table \ref{sbtb1}.
\end{proof}

\begin{remark}\label{sbrmk3}
If $\Phi$ is of classical type, the above results show that $\Psi_\lambda^{++}=0$ serves well as a simplicity criterion if we can explicitly rules out the exceptional cases (that is, when $(\Phi(\beta), i_\beta, j_\beta)\simeq(B_2, 1, 2)$, $(B_2, 2, 1)$, $(C_2, 1, 2)$, $(C_2, 2, 1)$ or $(D_4, 2, 2)$). This will be achieved in the next section.
\end{remark}

%
%
\section{Jantzen coefficients for classical Lie algebras}
%
%

In this section, we will determine all the Jantzen coefficients and give a refinement of Jantzen's simplicity criteria for classical Lie algebras.

\subsection{Jantzen coefficients for classical types}

Let $\Phi$ be one of the classical root systems $B_n$, $C_n$ and $D_n$ with simple roots $\Delta=\{\alpha_1, \cdots, \alpha_n\}$. Here $\alpha_i=e_i-e_{i+1}$ for $1\leq i\leq n-1$ and $\alpha_n=e_n$, $2e_n$ and $e_{n-1}+e_n$ respectively. We say the subset $I\subset\Delta$ is \emph{not standard} if $\Phi=D_n$ ($n\geq 4$), $\alpha_{n-1}=e_{n-1}-e_n\not\in I$ and $\alpha_{n}=e_{n-1}+e_n\in I$, otherwise we say $I$ is \emph{standard}. The following isomorphism of $\frh^*$ sends a nonstandard $I\subset\Delta$ for $\Phi=D_n$ to a standard one:
\begin{equation}\label{simceq1}
\vf:\frh^*\ra \frh^*\ \mbox{with}\ \vf(\lambda)=s_{e_n}\lambda.
\end{equation}

Write $\Delta\backslash I=\{\alpha_{q_1}, \dots, \alpha_{q_{m-1}}\}$, where $1\leq q_1<\cdots< q_{m-1}\leq n$. So $m=n+1-|I|$. Set $q_0=0$ and $q_m=n+1$. If $I$ is standard or $s<m-1$, set $I_s=\{\alpha_i\in I\ |\ q_{s-1}<i<q_s\}$. If $I$ is not standard, then $q_{m-1}=n-1$. Put $I_{m-1}=\{\alpha_i\in I\ |\ q_{m-2}<i<q_{m-1}\}\cup\{\alpha_n\}$ and $I_m=\emptyset$. Thus we obtain
\begin{equation*}
\Phi_I=\bigsqcup_{s=1}^m \Phi_{I_s}.
\end{equation*}
Put $n_s=|I_s|+1$ for $1\leq s<m$ and $n_m=|I_m|$. If $\Phi=D_n$, then $n_s$ is invariant under the map $I\ra\vf(I)$. If $I$ is standard, then $n_s=q_s-q_{s-1}$ for $1\leq s<m$ and $n_m=n-q_{m-1}$. 

Recall that for integral weight $\lambda\in\Lambda_I^+$, one has $\lambda=w\overline\lambda$ with $w\in{}^IW^J$. In addition, $\Phi_{\overline\lambda}=\Phi_J=w^{-1}\Phi_\lambda$. Similarly, write $\Delta\backslash J=\{\alpha_{\overline q_1}, \dots, \alpha_{\overline q_{\overline m-1}}\}$ for $1\leq \overline q_1<\cdots< \overline q_{\overline m-1}\leq n$ and $\overline m=n+1-|J|$. Put $\overline q_0=0$ and $\overline q_{\overline m}=n+1$. If $J$ is standard or $s<\overline m-1$, let $J_s=\{\alpha_i\in J\ |\ \overline q_{s-1}<i< \overline q_s\}$. If $J$ is not standard, let $J_{\overline m-1}=\{\alpha_i\in I\ |\ \overline q_{\overline m-2}<i<n-1\}\cup\{\alpha_n\}$ and $J_m=\emptyset$. So
\begin{equation*}
\Phi_J=\bigsqcup_{s=1}^{\overline m} \Phi_{J_s}.
\end{equation*}
Set $\overline n_s=|J_s|+1$ for $1\leq s<\overline m$ and $\overline n_{\overline m}=|J_{\overline m}|$. Put $a_s=\overline\lambda_{\overline q_s}$ and for $1\leq s< \overline m$ and $a_{\overline m}=0$. Let
\begin{equation}\label{simceq2}
\caA=\{a_1, a_2, \cdots, a_{\overline m}\}.
\end{equation}
Then $a_1>\cdots>a_{\overline m-1}\geq a_{\overline m}=0$ (this is true even when $J$ is not standard).


\begin{remark}\label{simcrmk1}
One might expect $\overline n_{\overline m}=|\{1\leq i\leq n\mid \lambda_i=0\}|$. This does hold in most cases except when $a_{\overline m-1}=0=a_{\overline m}$. In this case, $\Phi=D_n$, $\overline q_{\overline m-2}=n-1$ and $\overline q_{\overline m-1}=n$. For example, let $\Phi=D_4$ and $I=J=\emptyset$. Choose $\lambda=(4, 3, 2, 0)$. Then
\[
\overline n_{\overline m}=0\neq 1=|\{1\leq i\leq n\mid \lambda_i=0\}|.
\]
In particular, the categories $\caO^\frb_\lambda\simeq\caO^\frb_\mu$ for $\mu=(4, 3, 2, 1)$. In this situation, we say $\lambda$ has a \emph{nonessential} $0$-entry.
\end{remark}

Here are some useful facts.

\begin{lemma}\label{simclem1}
Use the above notation.
\begin{itemize}
  \item [(1)] All the root systems $\Phi_{I_s}$ $($resp. $\Phi_{J_s}$$)$ are irreducible except the case when $\Phi=D_n$ and $q_{m-1}=n-2$ $($resp. $\overline q_{\overline m-1}=n-2$$)$. In this case, $\Phi_{I_{m}}$ $($resp. $\Phi_{J_{\overline m}}$$)$ is isomorphic to $A_1\times A_1$.
  \item [(2)] $\overline n_s=|\{1\leq i\leq n\mid \lambda_i=|a_s|\}|$ for $1\leq s\leq \overline m$ except when $a_{\overline m-1}=a_{\overline m}=0$ and $s=\overline m$. In this case, $\Phi=D_n$, $\overline n_{\overline m-1}=1$ and $\overline n_{\overline m}=0$.
\end{itemize}

\end{lemma}

Now we can state the following result about Jantzen coefficients.

\begin{theorem}\label{rcthm1}
Let $\Phi$ be a classical root system and $\lambda, \mu\in\Lambda_I^+$. Suppose $\mu=ws_\beta\lambda$ for some $w\in W_I$ and $\beta\in\Psi_\lambda^{++}$. Then $c(\lambda, \mu)=0$ if and only if one of the following conditions is satisfied.
\begin{itemize}
  \item [(\rmnum{1})] $\Phi=B_n$ $($resp. $C_n)$, $\beta=e_i$ $($resp. $2e_i)$ or $e_i+e_j$ for $q_{s-1}<i\leq q_s\leq q_{m-1}<j\leq n$ and $1\leq s<m$. Moreover, $\lambda_i=\lambda_j\in\frac{1}{2}\bbZ^{>0}$ $($resp. $\bbZ^{>0}$$)$ and $\lambda_k\neq 0$, $-\lambda_i$ for $q_{s-1}<k\leq q_s$.
  \item [(\rmnum{2})] $\Phi=B_n$ $($resp. $C_n)$, $\beta=e_i$ $($resp. $2e_i)$ or $e_i+e_j$ for $q_{s-1}<i<j\leq q_s$ and $1\leq s<m$. Moreover, $\lambda_i\in\bbZ^{>0}$, $\lambda_j=0$, $\lambda_k\neq-\lambda_i$ for $q_{s-1}<k\leq q_s$ and $\lambda_l\neq\lambda_i$ for $q_{m-1}<l\leq n$.
  \item [(\rmnum{3})] $\Phi=D_n$, $\beta=e_i+e_j$ or $e_i+e_k$ for $q_{s-1}<i<j\leq q_s\leq q_{m-1}<k<n$ and $1\leq s<m$. Moreover, $\lambda_i=\lambda_k\in\bbZ^{>0}$, $\lambda_j=\lambda_n=0$ and $\lambda_l\neq-\lambda_i$ for $q_{s-1}<l\leq q_s$.
\end{itemize}
\end{theorem}

The theorem will be proved in a case-by-case fashion (Lemma \ref{rctlem1}, Lemma \ref{rctlem2}, Lemma \ref{rctlem3} and Lemma \ref{rctlem4}) in the next four subsections. The main idea is to find $\Phi(\beta)$ for each $\beta\in\Psi_{\lambda}^{++}$.

\subsection{The case of type $A$ and the reduction process}

The case of type $A$ is relatively easy since we always has $\Phi(\beta)\simeq A_1$ for $\beta\in\Psi_\lambda^{++}$ in this case.

\begin{lemma}\label{rclem0}
Let $\lambda\in\Lambda_I^+$ and $\mu=ws_\beta\lambda\in\Lambda_I^+$ with $w\in W_I$ and $\beta\in\Psi_\lambda^{++}$. If $\Phi(\beta)$ is of type $A$, then $|c(\lambda, \mu)|=1$.
\end{lemma}
\begin{proof}
With $\beta\in\Psi_\lambda^{++}$, we can find $w'\in W(\Phi_I\cap\Phi(\beta))$ such that $\mu':=w's_\beta\lambda'\in\Lambda^+(\Phi_I\cap\Phi(\beta), \Phi(\beta))$, where $\lambda'=\lambda|_{\Phi(\beta)}$. Thus the basic system $(\Phi(\beta), i_\beta, j_\beta)$ contains two different basic weights $\lambda', \mu'$ with $\mu'=W(\Phi(\beta))\lambda'$. In view of Theorem \ref{bwthm1}, this forces $(\Phi(\beta), i_\beta, j_\beta)\simeq(A_1, 1, 1)$ and $c(\lambda', \mu', \Phi_I\cap\Phi(\beta), \Phi(\beta))=1$. Then Lemma \ref{invjlem} implies that $|c(\lambda, \mu)|=1$.
\end{proof}

We get the following result immediately.

\begin{lemma}\label{rctlem1}
Theorem \ref{rcthm1} holds for $\Phi=A_n$.
\end{lemma}

We need more notations and several lemmas to investigate the reduction process for the other classical types.  First consider the parabolic reduction. Choose standard $I\subset\Delta$. For $\beta=\pm(e_i\pm e_j)$ ($i<j$), there exist $s, t\in\{1, \cdots,  m\}$ such that $q_{s-1}<i\leq q_s$ and $q_{t-1}<j\leq q_t$. Set $c_1^I(\beta)=s$. Put $c_2^I(\beta)=t$ when $t>s$ and $c_2^I(\beta)=m$ when $t=s$. For $\beta=e_i$ or $2e_i$ with $q_{s-1}<i\leq q_s$, set $c_1^I(\beta)=s$ and $c_2^I(\beta)=m$. If $I$ is not standard, set $c_i^I(\beta)=c_i^{\vf(I)}(\vf(\beta))$ for $i=1, 2$.

\begin{lemma}\label{rclem1}
Suppose $\Phi=B_n$ $($resp. $C_n$$)$. Choose $I\subset\Delta$. Fix $\beta\in\Phi\backslash\Phi_I$. Then
\begin{equation}\label{rcl1eq1}
(\Phi_{\beta, 1})_{\beta, 0}=(\bbQ\beta+\bbQ \Phi_{I_s}+\bbQ\Phi_{I_t})\cap\Phi,
\end{equation}
where $s=c_1^I(\beta)$ and $t=c_2^I(\beta)$. In particular, $(\Phi_{\beta, 1})_{\beta, 0}$ is of type $B$ $($resp. $C$$)$ when $t=m$ and $|I_s\cup I_m|\geq1$, otherwise it is of type $A$.
\end{lemma}

\begin{proof}
It suffices to consider the case $\Phi=B_n$. The proof for $\Phi=C_n$ is similar. Denote $\Phi'=(\Phi_{\beta, 1})_{\beta, 0}$. For any $\beta\in\Phi^+\backslash\Phi_I$, there exists exactly one positive root $\gamma\in\Phi_{\beta, 1}$ such that $I\cup\{\gamma\}$ is a basis for $\Phi_{\beta, 1}$ (see Remark \ref{redrmk1}). And $\Phi'$ is the irreducible component of $\Phi_{\beta, 1}$ containing $\beta$.

First consider the case $\beta=\pm(e_i-e_j)$ with $1\leq i<j\leq n$. Since $\beta\not\in\Phi_I$, we must have $q_{s-1}<i\leq q_s\leq q_{t-1}<j\leq q_t$ for $s=c_1^I(\beta)<t=c_2^I(\beta)$. It is evident that $\{e_{q_s}-e_{q_{t-1}+1}\}\cup I$ is a basis of $\Phi_{\beta, 1}$, that is, $\gamma=e_{q_s}-e_{q_{t-1}+1}$. In view of Lemma \ref{simclem1}, the irreducible component $\Phi'$ of $\Phi_{\beta, 1}$ containing $\beta$ has a basis $\{\gamma\}\cup I_s\cup I_t$. So (\ref{rcl1eq1}) follows. If $t<m$, the simple roots of $\Phi'$ are
\[
\{e_{q_{s-1}+1}-e_{q_{s-1}+2}, \cdots, e_{q_s-1}-e_{q_s}, e_{q_s}-e_{q_{t-1}+1}, e_{q_{t-1}+1}-e_{q_{t-1}+2}, \cdots,  e_{q_t-1}-e_{q_t}\}.
\]
So $\Phi'\simeq A_{n_s+n_t-1}$. If $t=m$, then $q_{m-1}<j\leq n$ and $n_m=n-q_{m-1}\geq1$. The simple roots of $\Phi'$ are
\[
\{e_{q_{s-1}+1}-e_{q_{s-1}+2}, \cdots, e_{q_s}-e_{q_{m-1}+1}, e_{q_{m-1}+1}-e_{q_{m-1}+2}, \cdots,  e_{n-1}-e_{n}, e_n\}.
\]
Therefore $\Phi'\simeq B_{n_s+n_m}$ (keeping in mind that $n_s+n_m\geq 2$).

Then assume that $\beta=\pm(e_i+e_j)$ with $q_{s-1}<i\leq q_s\leq q_{t-1}<j\leq q_t$. We get $\gamma=e_{q_s}+e_{q_{t}}$ for $t<m$ and $\gamma=e_{q_s}-e_{q_{m-1}+1}$ for $t=m$. In both cases (\ref{rcl1eq1}) holds. If $t<m$, the set of simple roots of $\Phi'$ is
\[
\{e_{q_{s-1}+1}-e_{q_{s-1}+2}, \cdots, e_{q_s-1}-e_{q_s}, e_{q_s}+e_{q_t},  e_{q_t-1}-e_{q_t}, \cdots, e_{q_{t-1}+1}-e_{q_{t-1}+2}\}.
\]
So $\Phi'\simeq A_{n_s+n_t-1}$. If $t=m$, then $\Phi'\simeq B_{n_s+n_m}$ ($n_s+n_m\geq 2$).

Next assume that $\beta=\pm(e_i+e_j)$ with $q_{s-1}<i<j\leq q_s$. So $n_s=q_s-q_{s-1}\geq2$ and $s<t=c_2^I(\beta)=m$. Since $e_i-e_j\in\Phi_{I_s}\subset\Phi_I\subset\Phi_{\beta, 1}$, we obtain $e_i, e_j\in\Phi_{\beta, 1}$. It follows that $e_k\in\Phi_{\beta, 1}$ for $q_{s-1}<k\leq q_s$. If  $q_{m-1}<n$, then $e_l\in\Phi_I\subset\Phi_{\beta, 1}$ for $q_{m-1}<l\leq n$. We obtain $\gamma=e_{q_s}-e_{q_{m-1}+1}$ and $\Phi'\simeq B_{n_s+n_m}$. If $q_{m-1}=n$, then $\gamma=e_{q_s}$ and $\Phi'\simeq B_{n_s}$. The equation (\ref{rcl1eq1}) holds in all these cases.

At last, assume that $\beta=\pm e_i$ for $q_{s-1}<i\leq q_s$. Then $s<t=c_2^I(\beta)=m$. We get $\gamma=e_{q_s}-e_{q_{m-1}+1}$ and $\Phi'\simeq B_{n_s+n_m}$ for $n_m\geq1$. Similarly, one obtains $\gamma=e_{q_s}$ for $n_m=0$, while $\Phi'\simeq B_{n_s}$ for $n_s\geq2$ and $\Phi'\simeq A_{1}$ for $n_s=1$. Moreover, the equation (\ref{rcl1eq1}) follows in either case.

In view of $n_s+n_m=|I_s|+|I_m|+1$, the second statement is an easy consequence of the above proof.
\end{proof}

\begin{lemma}\label{rclem2}
Suppose $\Phi=D_n$. Choose $I\subset\Delta$. Fix $\beta\in\Phi\backslash\Phi_I$. Then
\begin{equation}\label{rcl2eq1}
(\Phi_{\beta,1})_{\beta, 0}=\{\pm\beta\}\ \mbox{or}\ (\bbQ\beta+\bbQ \Phi_{I_s}+\bbQ\Phi_{I_t})\cap\Phi,
\end{equation}
where $s=c_1^I(\beta)$ and $t=c_2^I(\beta)$. In particular, $(\Phi_{\beta,1})_{\beta, 0}$ is of type $D$ when $t=m$ and $|I_s\cup I_m|\geq3$, otherwise it is of type $A$.
\end{lemma}
\begin{proof}
With (\ref{simceq1}), we only need to consider the case when $I$ is standard. Set $\Phi'=(\Phi_{\beta,1})_{\beta, 0}$. The simple system corresponding to $\Phi_{\beta, 1}^+:=\Phi_{\beta, 1}\cap\Phi^+$ is $I\cup\{\gamma\}$.

First assume that $\beta=\pm(e_i\pm e_j)$ with $q_{s-1}<i\leq q_s\leq q_{t-1}<j\leq q_t$. If $t<m$, we can follow the proof in Lemma \ref{rclem1} and get $\Phi'\simeq A_{n_s+n_t-1}$. If $t=m$, then $q_{m-1}<j\leq n$. We must have $q_{m-1}\leq n-2$ since $q_{m-1}\neq n-1$ for standard $I$. The set of simple roots in $\Phi'$ is $\{\gamma\}\cup I_s\cup I_m$, which is
\[
\{e_{q_{s-1}+1}-e_{q_{s-1}+2}, \cdots, e_{q_s}-e_{q_{m-1}+1}, e_{q_{m-1}+1}-e_{q_{m-1}+2}, \cdots,  e_{n-1}-e_{n}, e_{n-1}+e_{n}\}.
\]
Thus (\ref{rcl2eq1}) follows. In this case, $\Phi'\simeq D_{n_s+n_m}$ when $n_s+n_m\geq4$ and $\Phi'\simeq A_{3}$ when $n_s+n_m=3$ (that is, $n_s=1$ and $n_m=n-q_{m-1}=2$).

Now assume that $\beta=\pm(e_i+e_j)$ with $q_{s-1}<i<j\leq q_s$. Then $s<t=c_2^I(\beta)=m$ and $e_k\in\Phi_{\beta, 1}$ for $q_{s-1}<k\leq q_s$. If $q_{m-1}\leq n-2$, we get $e_l\pm e_{n}\in\Phi_I$ and thus $e_l, e_n\in\bbQ\Phi_I$ for $q_{m-1}<l<n$. So $e_k\pm e_l\in\Phi_{\beta, 1}$ and $\gamma=e_{q_s}-e_{q_{m-1}+1}$. Then $\Phi'\simeq D_{n_s+n_m}$ (keeping in mind that $n_s+n_m\geq 2+2=4$). If $q_{m-1}=n$, then $I_m=\emptyset$ and $\gamma=e_{q_s-1}+e_{q_s}$. The set of simple roots of $\Phi'$ is
\[
\{e_{q_{s-1}+1}-e_{q_{s-1}+2}, \cdots, e_{q_s-1}-e_{q_s}, e_{q_s-1}+e_{q_s}\}
\]
when $n_s\geq3$ and is $\{e_{q_s-1}+e_{q_s}\}$ when $n_s=2$ (in this case one has $\Phi_{\{\gamma\}\cup I_s\cup I_m}=\{\pm(e_{q_s-1}\pm e_{q_s})\}\simeq A_1\times A_1$ is not irreducible, the irreducible subsystem of $\Phi_{\beta, 1}$ containing $\beta=e_{q_s-1}+e_{q_s}$ is $\{\pm\beta\}$). Therefore, $\Phi'\simeq D_{n_s}$ when $n_s\geq 4$ and $\Phi'\simeq A_{3}$ when $n_s=3$ and $\Phi'=\{\pm\beta\}\simeq A_1$ when $n_s=2$. Then (\ref{rcl2eq1}) holds in all these cases.

In view of $n_s+n_m=|I_s|+|I_m|+1$, the second statement follows from the above proof.
\end{proof}

Now consider the singular reduction. Choose $\lambda\in\frh^*$. For $\beta=\pm(e_i\pm e_j)$ ($i<j$), set $d_1^\lambda(\beta)=\max\{|\lambda_i|, |\lambda_j|\}$. Put $d_2^\lambda(\beta)=\min\{|\lambda_i|, |\lambda_j|\}$ if $|\lambda_i|\neq|\lambda_j|$ and $d_2^\lambda(\beta)=0$ if $|\lambda_i|=|\lambda_j|$. For $\beta=\pm e_i$ or $\pm2e_i$. Set $d_1^\lambda(\beta)=|\lambda_i|$ and $d_2^\lambda(\beta)=0$. Therefore we always have $d_2^\lambda(\beta)=0$ for $\beta\in \Phi_\lambda$. For $a\in\bbR$, define the following subsets of $\Phi_\lambda$:
\[
\Phi_\lambda(a):=\{\gamma\in \Phi_\lambda\ |\ d_1^\lambda(\gamma)=a\}.
\]
Then $\Phi_\lambda(a)$ is a subsystem of $\Phi_\lambda$. Set $\Phi_\lambda^+(a)=\Phi_\lambda(a)\cap\Phi^+$. One has
\begin{equation*}
\Phi_\lambda=\bigsqcup_{a\in\bbR} \Phi_\lambda(a).
\end{equation*}

\begin{lemma}\label{rclem3}
Suppose that $\Phi=B_n$ $($resp. $C_n$$)$. Choose $I\subset\Delta$. Fix an integral weight $\lambda\in\Lambda_I^+$ and $\beta\in\Psi_\lambda^+$. Then
\[
(\Phi_{\beta, 2})_{\beta, 0}=(\bbQ\beta+\bbQ \Phi_\lambda(a)+\bbQ \Phi_\lambda(b))\cap\Phi,
\]
where $a=d_1^\lambda(\beta)$ and $b=d_2^\lambda(\beta)$. In particular, $(\Phi_{\beta, 2})_{\beta, 0}$ is of type $B$ $($resp. $C$$)$ when $b=0$ and $\rank(\Phi_\lambda(a)\cup \Phi_\lambda(0))\geq1$, otherwise it is of type $A$.
\end{lemma}

\begin{proof}
It suffices to consider the case $\Phi=B_n$. Denote $\Phi'=(\Phi_{\beta, 2})_{\beta, 0}$. Note that $\lambda=w\overline\lambda$ for $w\in{}^IW^J$. Since $\overline\lambda$ is dominant, one has $\overline\lambda_1\geq\cdots\geq\overline\lambda_n\geq0$.
With $\Phi_{J_s}\subset \Phi_{\overline\lambda}$, we obtain (recall \ref{simceq2})
\[
\overline\lambda_{\overline q_{s-1}+1}=\cdots=\overline\lambda_{\overline q_{s}}=a_s
\]
for $1\leq s<\overline m$ and $\overline\lambda_{\overline q_{\overline m-1}+1}=\cdots=\overline\lambda_{n}=0$. Then $a_1>\cdots>a_{\overline m}=0$. It can be easily checked that for $\gamma\in\Phi$,
\begin{equation}\label{simcleq41}
d_1^{\overline\lambda}(\gamma)=a_s\ \mbox{and}\ d_2^{\overline\lambda}(\gamma)=a_t
\end{equation}
where $s=c_1^J(\gamma)$ and $t=c_2^J(\gamma)$. In particular,
\begin{equation}\label{simcleq42}
\Phi_{J_s}=\{\gamma\in\Phi_J\ |\ d_1^{\overline\lambda}(\gamma)=a_s\}.
\end{equation}
Note that for each $w\in W$, there exists a permutation $\sigma$ on $\{1, \cdots, n\}$ such that $we_i=\pm e_{\sigma(i)}$. Thus
\[
|\overline\lambda_i|=|\langle\overline\lambda, e_i\rangle|=|\langle\lambda, we_i\rangle|=|\lambda_{\sigma(i)}|.
\]
One has $d_i^{\overline\lambda}(\gamma)=d_i^\lambda(w\gamma)$ for any $\gamma\in\Phi$ and $i=1, 2$. Then (\ref{simcleq42}) yields
\begin{equation}\label{simcleq43}
w\Phi_{J_s}=\{w\gamma\in w\Phi_J\ |\ d_1^{\overline\lambda}(\gamma)=a_s\}=\{w\gamma\in \Phi_\lambda\ |\ d_1^\lambda(w\gamma)=a_s\}=\Phi_\lambda(a_s).
\end{equation}
On the other hand, $w^{-1}\Phi'$ is the irreducible component containing $w^{-1}\beta$ of
\[
w^{-1}\Phi_{\beta, 2}=w^{-1}(\bbQ \beta+\bbQ \Phi_\lambda)\cap\Phi=(\bbQ w^{-1}\beta+\bbQ \Phi_J)\cap\Phi.
\]
It follows from the argument in Lemma \ref{rclem1} that
\[
w^{-1}\Phi'=(\bbQ w^{-1}\beta+\bbQ \Phi_{J_s}+\bbQ \Phi_{J_t})\cap\Phi
\]
for $s=c_1^J(w^{-1}\beta)$ and $t=c_2^J(w^{-1}\beta)$. In view of (\ref{simcleq43}), we get
\[
\Phi'=(\bbQ \beta+\bbQ w\Phi_{J_s}+\bbQ w\Phi_{J_t})\cap\Phi=(\bbQ \beta+\bbQ \Phi_\lambda(a_s)+\bbQ \Phi_\lambda(a_t))\cap\Phi.
\]
With $\gamma$ replaced by $w^{-1}\beta$ in (\ref{simcleq41}), one has $a_s=d_1^{\overline\lambda}(w^{-1}\beta)=d_1^\lambda(\beta)=a$ and $a_t=d_2^{\overline\lambda}(w^{-1}\beta)=d_2^\lambda(\beta)=b$. Keeping in mind Lemma \ref{rclem1}, $\Phi'$ is of type $B$ if and only if $t=c_2^J(w^{-1}\beta)=\overline m$ (that is, $b=a_{\overline m}=0$) and $|J_s\cup J_{\overline m}|\geq1$ (that is, $\rank(\Phi_\lambda(a)\cup \Phi_\lambda(0))\geq1$).

\end{proof}

\begin{lemma}\label{rclem4}
Suppose that $\Phi=D_n$. Choose $I\subset\Delta$. Fix integral weight $\lambda\in\Lambda_I^+$ and $\beta\in\Psi_\lambda^+$. Then
\begin{equation}\label{simcl5eq1}
(\Phi_{\beta, 2})_{\beta, 0}=\{\pm\beta\}\ \mbox{or}\ (\bbQ\beta+\bbQ \Phi_\lambda(a)+\bbQ \Phi_\lambda(b))\cap\Phi,
\end{equation}
where $a=d_1^\lambda(\beta)$ and $b=d_2^\lambda(\beta)$. In particular, $(\Phi_{\beta, 2})_{\beta, 0}$ is of type $D$ when $b=0$ and $\rank (\Phi_\lambda(a)\cup \Phi_\lambda(0))\geq3$, otherwise it is of type $A$.
\end{lemma}

\begin{proof}
Assume that $\lambda=w\overline\lambda$ for $w\in{}^IW^J$. The dominant weight $\overline\lambda$ satisfies $\overline\lambda_1\geq\cdots\geq\overline\lambda_{n-1}\geq\pm\overline\lambda_n$. By imitating closely the argument in Lemma \ref{rclem3} and Lemma \ref{rclem2}, we can eventually get (whether or not $J$ is standard)
\begin{equation}\label{simcl5eq2}
w\Phi_{J_s}=\{w\gamma\in \Phi_\lambda\ |\ d_1^\lambda(w\gamma)=a_s\}=\Phi_\lambda(a_s)
\end{equation}
and
\[
(\Phi_{\beta, 2})_{\beta, 0}=\{\pm\beta\}\ \mbox{or}\ (\bbQ\beta+\bbQ \Phi_\lambda(a)+\bbQ \Phi_\lambda(b))\cap\Phi
\]
for $a=d_1^\lambda(\beta)$ and $b=d_2^\lambda(\beta)$.

The second statement also follows from Lemma \ref{rclem2}.
\end{proof}

Let $\Phi=B_n$, $C_n$ or $D_n$. Fix $\lambda\in\frh^*$. For $z\in\bbC$, denote
\[
K_{(z)}:=\{1\leq i\leq n\mid\lambda_i-z\in\bbZ\}.
\]
Then $\{1, 2, \cdots, n\}=\bigsqcup_{0\leq \Ree(z)<1} K_{(z)}$. If $z\not\in\frac{1}{2}\bbZ$, set
\[
\Phi_{(z)}=\{\pm(e_i-e_j), \pm(e_j+e_k), \pm(e_k-e_l)\mid i, j\in K_{(z)}, k, l\in K_{(1-z)}, i\neq j\ \mbox{and}\ k\neq l\}.
\]
Then $\Phi_{(z)}$ is a subsystem of type $A$. If $\Phi=B_n$ and $z\in\frac{1}{2}\bbZ$, put
\[
\Phi_{(z)}=\{\pm(e_i\pm e_j), \pm e_k\mid i, j, k\in K_{(z)}\ \mbox{and}\ i<j\}.
\]
If $\Phi=C_n$ and $z\in\bbZ$, write
\[
\Phi_{(z)}=\{\pm(e_i\pm e_j), \pm 2e_k\mid i, j, k\in K_{(z)}\ \mbox{and}\ i<j\}.
\]
If $\Phi=C_n$ and $z\in\frac{1}{2}+\bbZ$ or $\Phi=D_n$ and $z\in\frac{1}{2}\bbZ$, write
\[
\Phi_{(z)}=\{\pm(e_i\pm e_j)\mid i, j\in K_{(z)}\ \mbox{and}\ i<j\}.
\]

The following result is an easy consequence.

\begin{lemma}\label{rclem5}
Let $\Phi=B_n, C_n$ or $D_n$ and $\lambda\in\frh^*$. Then $\Phi_{[\lambda]}=\bigsqcup_{0\leq\Ree(z)\leq 1/2}\Phi_{(z)}$.
\end{lemma}

Keeping in mind that $\Phi_I\subset\Phi_{[\lambda]}$ for $\lambda\in\Lambda_I^+$, we have the following result.

\begin{lemma}\label{rclem6}
Let $\Phi=B_n, C_n$ or $D_n$ and $I$ be standard. Choose $\lambda\in\Lambda_I^+$. Fix $z\in\bbC$.
\begin{itemize}
  \item [(1)] Choose $1\leq s<m$. Then $q_{s}\in K_{(z)}$ if and only if $i\in K_{(z)}$ for $q_{s-1}<i\leq q_s$.
  \item [(2)] Suppose $q_{m-1}<n$. Then $n\in K_{(z)}$ if and only if $i\in K_{(z)}$ for $q_{m-1}<i\leq n$.
\end{itemize}
\end{lemma}

\subsection{The case of type $B$} Let $\lambda\in\Lambda_I^+$. Lemma \ref{rclem0} shows that $c(\lambda, \mu)$ is nonzero whenever $\Phi(\beta)$ is of type $A$, where $\mu=ws_\beta\lambda\in\Lambda_I^+$ for $w\in W_I$ and $\beta\in\Psi_\lambda^{++}$. It suffices to consider $\beta\in\Psi_\lambda^{++}$ such that $\Phi(\beta)$ is not of type $A$. Start with $\beta\in\Psi_\lambda^{+}\supset\Psi_\lambda^{++}$. Lemma \ref{rclem5} shows that $\Phi_1(\beta)=(\Phi_{[\lambda]})_{\beta, 0}=\Phi_{(z)}$ for some $z\in\bbC$ with $0\leq\Ree(z)\leq1/2$. If $\Phi(\beta)\subset\Phi_{(z)}$ is not of type $A$, we get $z=0$ or $1/2$ and
\[
\Phi_1(\beta)=\Phi_{(z)}=\{\pm(e_i\pm e_j), \pm e_k\mid i, j, k\in K_{(z)}\ \mbox{and}\ i<j\}
\]
is of type $B$. Lemma \ref{rclem1} and Lemma \ref{rclem6} imply
\[
\Phi_3(\beta)=(\bbQ\beta+\bbQ \Phi_{I_{s}}+\bbQ\Phi_{I_{t}})\cap\Phi_1(\beta)
\]
for $s=c_1^I(\beta)<t=c_2^I(\beta)$. Since $\Phi(\beta)\subset\Phi_3(\beta)$ is not of type $A$, we obtain $t=m$ and $I_s\cup I_m\neq\emptyset$. In view of Lemma \ref{rclem6}, one has
\[
\Phi_3(\beta)=\left\{\pm(e_i\pm e_j), \pm e_k\ |\ i, j, k\in L\ \mbox{and}\ i<j\right\}.
\]
Here the set $L=\{q_{s-1}+1, \cdots, q_s, q_{m-1}+1, \cdots, n\}$ when $n\in K_{(z)}$ and $L=\{q_{s-1}+1, \cdots, q_s\}$ when $n\not\in K_{(z)}$ (the case $L=\{q_{m-1}+1, \cdots, n\}$ is not possible, otherwise $\beta\in\Phi_3(\beta)\subset\Phi_I$, a contradiction). Lemma \ref{rclem3} yields
\[
\Phi_5(\beta)=(\bbQ\beta+\bbQ (\Phi_\lambda(a)\cap\Phi_3(\beta))+\bbQ (\Phi_\lambda(b)\cap\Phi_3(\beta)))\cap\Phi_3(\beta),
\]
where $a=d_1^\lambda(\beta)>b=d_2^\lambda(\beta)$. Moreover, $\Phi_5(\beta)$ is of type $B$ only when $b=0$ and $(\Phi_\lambda(a)\cap\Phi_3(\beta))\cup (\Phi_\lambda(0)\cap\Phi_3(\beta))\neq\emptyset$. On the other hand, $\lambda\in\Lambda_I^+$ ($z\in\frac{1}{2}\bbZ$) implies
\begin{equation}\label{beq1}
\lambda_{q_{s-1}+1}>\lambda_{q_{s-1}+2}>\cdots>\lambda_{q_{s}},\ \lambda_{q_{m-1}+1}>\cdots>\lambda_{n}>0.
\end{equation}
Therefore $\Phi_\lambda^+(0)\cap\Phi_3(\beta)=\emptyset$ or $\{e_i\}$ for some $q_{s-1}< i\leq q_s$ and $\Phi_\lambda^+(a)\cap\Phi_3(\beta)$ must be one of the following cases: (1) $\emptyset$; (2) $\{e_j+e_k\}$ for some $q_{s-1}<j<k\leq q_{s}$; (3) $\{e_j-e_l\}$ for $q_{s-1}<j\leq q_{s}\leq q_{m-1}<l$; (4) $\{e_k+e_l\}$ for $q_{s-1}<k\leq q_{s}\leq q_{m-1}<l$; (5) $\{e_j+e_k, e_j-e_l, e_k+e_l\}$ for $q_{s-1}<j<k\leq q_{s}\leq q_{m-1}<l$.

For example, if $\Phi_\lambda^+(0)\cap\Phi_3(\beta)=\{e_i\}$ for $q_{s-1}<i\leq q_{s}$ and $\Phi_\lambda^+(a)\cap\Phi_3(\beta)=\{e_j+e_k\}$ for $q_{s-1}<j<k\leq q_{s}$, then $\lambda_j=-\lambda_k=a$ and $\lambda_i=0$. With $\beta\in\Psi_\lambda^+$, $d_1^\lambda(\beta)=a$ and $d_2^\lambda(\beta)=0$, we must have $\beta=e_j$ or $e_j+e_i$. So
\[
\Phi_5(\beta)=\{\pm(e_i\pm e_j), \pm(e_i\pm e_k), \pm(e_j\pm e_k), \pm e_i, \pm e_j, \pm e_k\}
\]
with $\Phi_I\cap\Phi_5(\beta)=\{\pm(e_i-e_j), \pm(e_i-e_k), \pm(e_j-e_k)\}$ and $\Phi_\lambda\cap\Phi_5(\beta)=\{\pm e_i, \pm(e_j+e_k)\}$. The corresponding basic system is $(B_3, 3, 2)$ and $\Phi(\beta)=\Phi_5(\beta)$. It can be verified that $\Psi_\lambda^+\cap\Phi(\beta)=\{e_j, e_j+e_i\}$. Note that $\langle s_{e_j}\lambda, e_j-e_k\rangle=0=\langle s_{e_j+e_i}\lambda, e_i-e_k\rangle$. Both $s_{e_i}\lambda$ and $s_{e_i+e_l}\lambda$ are $\Phi_I$-singular. It forces $\Psi_\lambda^{++}\cap\Phi(\beta)=\emptyset$. This is case (7) in Table \ref{simctab1}, while all the possible cases are listed in the table (the last column shows whether or not the generalized Verma module $M(\lambda|_{\Phi(\beta)}, \Phi_I\cap\Phi(\beta), \Phi(\beta))$ is simple). To summarize:

\begin{lemma}\label{alem1}
Let $\Phi=B_n$, $\lambda\in\Lambda_I^+$ and $\beta\in\Psi_\lambda^+$. Suppose that $\Phi(\beta)$ is not of type $A$. Then $\beta$ must satisfy one of the conditions listed in Table \ref{simctab1} (where $q_{s-1}<i, j, k\leq q_s$ and $q_{m-1}<l\leq n$).
\end{lemma}

\begin{table}[htbp]\footnotesize
\begin{tabular}{|c|c|p{2.5cm}|p{1.9cm}|p{1.5cm}|c|c|}
\hline
 & $\Phi_\lambda^+(0)\cap\Phi_3(\beta)$ & \centering $\Phi_\lambda^+(a)\cap\Phi_3(\beta)$ & \centering $\Psi_\lambda^+\cap\Phi(\beta)$ & \centering $\Psi_\lambda^{++}\cap\Phi(\beta)$ & basic sys. & simple \\
\hline
1 & \multirow{4}{*}{$\emptyset$} & \centering$e_j+e_k$ $(j<k)$ & \centering$e_j$ & \centering  $\emptyset$ & $(B_2, 2, 2)$ & Yes\\
\cline{1-1} \cline{3-7}
2 & & \centering$e_j-e_l$ &\centering $e_j$, $e_j+e_l$ & \centering  $e_j$, $e_j+e_l$ & $(B_2, 1, 2)$ & Yes\\
\cline{1-1} \cline{3-7}
3 & & \centering$e_k+e_l$ & \centering$\emptyset$ & \centering  $\emptyset$ & $-$ & $-$ \\
\cline{1-1} \cline{3-7}
4 & & \centering$e_j+e_k$, $e_j-e_l$, $e_k+e_l$ $(j<k)$ & \centering$e_j$, $e_j+e_l$ & \centering $\emptyset$ & $(B_3, 2, 3)$ & Yes\\
\hline
5 & \multirow{6}{*}{\centering $e_i$} & \centering\multirow{2}{*}{$\emptyset$} & \centering$e_j$, $e_j+e_i$, $(j<i)$ & \centering  $e_j$, $e_j+e_i$ & $(B_2, 2, 1)$ & Yes\\
\cline{1-1} \cline{4-7}
6 & &  & \centering$e_i+e_l$ & \centering  $\emptyset$ & $(B_2, 1, 1)$ & Yes\\
\cline{1-1} \cline{3-7}
7 & & \centering$e_j+e_k$ $(j<i<k)$ & \centering$e_j$, $e_j+e_i$ & \centering  $\emptyset$ & $(B_3, 3, 2)$ & Yes\\
\cline{1-1} \cline{3-7}
8 & & \centering$e_j-e_l$ $(j<i)$& \centering$e_j$, $e_j+e_i$, $e_j+e_l$, $e_i+e_l$ &  \centering $e_j$, $e_j+e_i$, $e_j+e_l$ & $(B_3, 2, 2)$ & No\\
\cline{1-1} \cline{3-7}
9 & & \centering$e_k+e_l$ $(i<k)$ & \centering$e_i+e_l$ &  \centering$\emptyset$ & $(B_3, 2, 2)$ & Yes\\
\cline{1-1} \cline{3-7}
10 & & \centering$e_j+e_k$, $e_j-e_l$, $e_k+e_l$ $(j<i<k)$ & \centering$e_j$, $e_j+e_i$, $e_j+e_l$, $e_i+e_l$ & \centering$\emptyset$  & $(B_4, 3, 3)$ & Yes\\
\hline
\end{tabular}
\bigskip
\caption{reduction of type $B$}
\label{simctab1}
\end{table}

\begin{lemma}\label{rctlem2}
Theorem \ref{rcthm1} holds for $\Phi=B_n$.
\end{lemma}
\begin{proof}
Choose $w'\in W(\Phi_I\cap\Phi(\beta))$ such that $\mu'=w's_\beta\lambda'\in\Lambda^+(\Phi_I\cap\Phi(\beta), \Phi(\beta))$, where $\lambda'=\lambda|_{\Phi(\beta)}$. Note that $\lambda', \mu'$ are two different basic weights of the basic system $(\Phi(\beta), i_\beta, j_\beta)$.

If $c(\lambda, \mu)=0$, then $c(\lambda', \mu', \Phi_I\cap\Phi(\beta), \Phi(\beta))=0$ in view of Lemma \ref{invjlem}. Lemma \ref{rclem0} and Lemma \ref{alem1} show that $\Phi(\beta)$ is of type $B$. Moreover, $(\Phi(\beta), i_\beta, j_\beta)\simeq(B_2, 1, 2)$, $(B_2, 2, 1)$ or $(B_3, 2, 2)$ (corresponding to cases (2) (5) and (7) respectively in Table \ref{simctab1}) since $\beta\in\Psi_\lambda^{++}\cap\Phi(\beta)$. If $(\Phi(\beta), i_\beta, j_\beta)\simeq(B_3, 2, 2)$, Theorem \ref{nzj} implies that $c(\lambda', \mu', \Phi_I\cap\Phi(\beta), \Phi(\beta))=1$, a contradiction. If $(\Phi(\beta), i_\beta, j_\beta)\simeq(B_2, 1, 2)$, we obtain $\lambda_j=\lambda_l=a$ in view of Lemma \ref{alem1}, where $q_{s-1}<j\leq q_s$ and $q_{m-1}<l\leq n$. Moreover, $\beta\in\{e_j, e_j+e_l\}\subset\Psi_\lambda^{++}$ and $\lambda_u\neq 0, -a$ for $q_{s-1}<u\leq q_s$. It follows from $2a=\langle\lambda, e_j^\vee\rangle\in\bbZ$ that $a\in\frac{1}{2}\bbZ$. If $(\Phi(\beta), i_\beta, j_\beta)\simeq(B_2, 2, 1)$, the argument is similar. Conversely, if (\rmnum{1}) or (\rmnum{2}) holds, The argument before Lemma \ref{alem1} yields $\Phi(\beta)=\Phi_5(\beta)\simeq B_2$ and $(\Phi(\beta), i_\beta, j_\beta)\simeq (B_2, 1, 2)$ or $(B_2, 2, 1)$. Therefore Theorem \ref{nzj} and Lemma \ref{invjlem} yield $c(\lambda', \mu', \Phi_I\cap\Phi(\beta), \Phi(\beta))=c(\lambda, \mu)=0$.
\end{proof}

\subsection{The case of type $D$} We consider $D_n$ before $C_n$ because we need the result of the former case in the proof of the latter one.

Let $\lambda\in\Lambda_I^+$. Choose $\beta\in\Psi_\lambda^{++}$. Suppose that $\Phi(\beta)$ is not of type $A$. Theorem \ref{bwthm4} shows that $(\Phi(\beta), i_\beta, j_\beta)=(D_4, 2, 2)$ or $(D_5, 3, 3)$. In these cases, $n_m=|I_m|\geq2$ and $q_{m-1}=n-n_m\leq n-2$. Thus $I$ is standard. Lemma \ref{rclem5} shows that
\[
\Phi_1(\beta)=\Phi_{(z)}=\{\pm(e_i\pm e_j)\mid i, j\in K_{(z)}\ \mbox{and}\ i<j\}.
\]
for some $z\in\frac{1}{2}\bbZ$. Denote $s=c_1^I(\beta)$. Lemma \ref{rclem2} and Lemma \ref{rclem6} give
\[
\Phi_3(\beta)=\left\{\pm(e_i\pm e_j)\ |\ i, j\in L\ \mbox{and}\ i<j\right\}.
\]
Here the set $L=\{q_{s-1}+1, \cdots, q_s, q_{m-1}+1, \cdots, n\}$ when $n\in K_{(z)}$ and $L=\{q_{s-1}+1, \cdots, q_s\}$ when $n\not\in K_{(z)}$. Moreover, $|L|\geq4$. Then Lemma \ref{rclem4} yields
\[
\Phi_5(\beta)=(\bbQ\beta+\bbQ (\Phi_\lambda(a)\cap\Phi_3(\beta))+\bbQ (\Phi_\lambda(0)\cap\Phi_3(\beta)))\cap\Phi_3(\beta),
\]
where $a=d_1^\lambda(\beta)>0$ and
\begin{equation}\label{simceq6}
\rank(\Phi_\lambda(a)\cap\Phi_3(\beta))+\rank(\Phi_\lambda(0)\cap\Phi_3(\beta))\geq3.
\end{equation}
Since $\lambda\in\Lambda_I^+$, we have
\begin{equation}\label{beq2}
\lambda_{q_{s-1}+1}>\lambda_{q_{s-1}+2}>\cdots>\lambda_{q_{s}},\ \lambda_{q_{m-1}+1}>\cdots>\lambda_{n-1}>|\lambda_{n}|.
\end{equation}
It can be easily checked that $\Phi_\lambda^+(0)\cap\Phi_3(\beta)=\emptyset$ or $\{e_i\pm e_n\}$ for some $q_{s-1}< i\leq q_s$. Moreover, $\Phi_\lambda^+(a)\cap\Phi_3(\beta)$ must be one of the following cases: (1) $\emptyset$; (2) $\{e_j+e_k\}$ for some $q_{s-1}<j<k\leq q_{s}$; (3) $\{e_j-e_l\}$ for $q_{s-1}<j\leq q_{s}\leq q_{m-1}<l$; (4) $\{e_k+e_l\}$ for $q_{s-1}<k\leq q_{s}\leq q_{m-1}<l$; (5) $\{e_j+e_k, e_j-e_l, e_k+e_l\}$ for $q_{s-1}<j<k\leq q_{s}\leq q_{m-1}<l$. Therefore $\rank(\Phi_\lambda(a)\cap\Phi_3(\beta))\leq 2$ for $a\in\bbZ^{>0}$. We must have $\rank(\Phi_\lambda(0)\cap\Phi_3(\beta))\geq1$ by (\ref{simceq6}). So $\Phi_\lambda^+(0)\cap\Phi_3(\beta)=\{e_i\pm e_n\}$, that is, $\lambda_i=\lambda_n=0$ for some $q_{s-1}< i\leq q_s$. This forces $n\in K_{(z)}$.

For example, if $\Phi_\lambda^+(a)\cap\Phi_3(\beta)=\{e_j+e_k\}$ for $q_{s-1}<j<k\leq q_{s}$, then $\lambda_j=-\lambda_k=a$ and $\lambda_i=\lambda_n=0$. Since $d_1^\lambda(\beta)=a$ and $d_2^\lambda(\beta)=0$, we must have $\beta=e_j+e_i$ or $e_j+e_n$ (keeping in mind that $\beta\in\Psi_\lambda^+$). So
\[
\Phi_5(\beta)=\{\pm(e_i\pm e_j), \pm(e_i\pm e_k), \pm(e_i\pm e_n), \pm(e_j\pm e_k), \pm(e_j\pm e_n), \pm(e_k\pm e_n)\}
\]
with $\Phi_I\cap\Phi_5(\beta)=\{\pm(e_i-e_j), \pm(e_i-e_k), \pm(e_j-e_k)\}$ and $\Phi_\lambda\cap\Phi_5(\beta)=\{\pm(e_j+e_k), \pm (e_i\pm e_n)\}$. In particular, if $\beta=e_j+e_i$, then
\[
\Phi_7(\beta)=(\bbQ\beta+\bbQ(\Phi_I\cap\Phi_5(\beta))\cap\Phi_5(\beta)=\{\pm(e_i\pm e_j), \pm(e_i\pm e_k), \pm(e_j\pm e_k)\}
\]
and $\Phi_\lambda\cap\Phi_7(\beta)=\{\pm(e_j+e_k)\}$. It follows that
\[
\Phi_9(\beta)=(\bbQ\beta+\bbQ(\Phi_\lambda\cap\Phi_7(\beta))\cap\Phi_7(\beta)=\{\pm(e_i+e_j), \pm(e_i-e_k), \pm(e_j+e_k)\}
\]
with $\Phi_I\cap\Phi_9(\beta)=\{\pm(e_i-e_k)\}$ and $\Phi_\lambda\cap\Phi_9(\beta)=\{\pm(e_j+e_k)\}$. The corresponding basic system is $(A_2, 1, 2)$ and $\Phi(\beta)=\Phi_9(\beta)\simeq A_2$ is not of type $D$, a contradiction. We arrive at a similar contradiction when $\beta=e_j+e_n$. We list all the other possible cases in Table \ref{simctab2} (the last column shows whether or not $M(\lambda|_{\Phi(\beta)}, \Phi_I\cap\Phi(\beta), \Phi(\beta))$ is simple).

\begin{lemma}\label{alem2}
Let $\Phi=D_n$, $\lambda\in\Lambda_I^+$ and $\beta\in\Psi_\lambda^+$. Suppose that $\Phi(\beta)$ is not of type $A$. Then $\beta$ must satisfy one of the conditions listed in Table \ref{simctab2} (where $q_{s-1}<i, j, k\leq q_s$ and $q_{m-1}<l<n$).
\end{lemma}

\begin{table}[htbp]\footnotesize
\begin{tabular}{|c|c|p{2.5cm}|p{2.0cm}|p{1.5cm}|c|c|}
\hline
 & $\Phi_\lambda^+(0)\cap\Phi_3(\beta)$ & \centering $\Phi_\lambda^+(a)\cap\Phi_3(\beta)$ & \centering $\Psi_\lambda^+\cap\Phi(\beta)$ & \centering $\Psi_\lambda^{++}\cap\Phi(\beta)$ & basic sys. & simple \\
\hline
1 & \multirow{3}{*}{\centering $e_i\pm e_n$} & \centering$e_j-e_l$ $(j<i)$ &\centering $e_j+e_i$, $e_j+e_l$, $e_j+e_n$, $e_i+e_l$ &\centering $e_j+e_i$, $e_j+e_l$ & $(D_4, 2, 2)$ & Yes\\
\cline{1-1} \cline{3-7}
2 & & \centering$e_k+e_l$ $(i<k)$ & \centering $e_i+e_l$ & \centering $\emptyset$ & $(D_4, 2, 2)$ & Yes\\
\cline{1-1} \cline{3-7}
3 & & \centering$e_j+e_k$, $e_j-e_l$, $e_k+e_l$ $(j<i<k)$ & \centering$e_j+e_i$, $e_j+e_l$, $e_j+e_n$, $e_i+e_l$ & \centering $\emptyset$  & $(D_5, 3, 3)$ & Yes\\
\hline
\end{tabular}
\bigskip
\caption{reduction of type $D$}
\label{simctab2}
\end{table}

\smallskip

\begin{lemma}\label{rctlem3}
Theorem \ref{rcthm1} holds for $\Phi=D_n$.
\end{lemma}

\begin{proof}
Choose $w'\in W(\Phi_I\cap\Phi(\beta))$ with $\mu'=w's_\beta\lambda'\in\Lambda^+(\Phi_I\cap\Phi(\beta), \Phi(\beta))$, where $\lambda'=\lambda|_{\Phi(\beta)}$.

If $c(\lambda, \mu)=0$, then $c(\lambda', \mu', \Phi_I\cap\Phi(\beta), \Phi(\beta))=0$ keeping in mind Lemma \ref{invjlem}. Lemma \ref{rclem0} and Lemma \ref{alem2} imply that $\Phi(\beta)$ is of type $D$. With $\beta\in\Psi_\lambda^{++}\cap\Phi(\beta)$, we have $(\Phi(\beta), i_\beta, j_\beta)\simeq(D_4, 2, 2)$ (corresponding to case (1) in Table \ref{simctab1}). This forces $\lambda_i=\lambda_n=0$ and $\lambda_j=\lambda_l=a\in\bbZ^{>0}$. Moreover, $\lambda_u\neq-a$ for $q_{s-1}<u\leq q_s$. Conversely, if (\rmnum{3}) holds, the reasoning before Lemma \ref{alem2} gives $(\Phi(\beta), i_\beta, j_\beta)\simeq(D_4, 2, 2)$. It follows from Theorem \ref{nzj} and Lemma \ref{invjlem} that $c(\lambda, \mu)=0$.
\end{proof}

\subsection{The case of type $C$} Let $\lambda\in\Lambda_I^+$ and $\beta\in\Psi_\lambda^+$. Lemma \ref{rclem5} implies that $\Phi_1(\beta)=\Phi_{(z)}$ is of type $A$, $C$ or $D$. We claim that $\Phi(\beta)$ is not of type $D$. Otherwise
\[
\Phi_1(\beta)=\Phi_{(\frac{1}{2})}=\{\pm(e_i\pm e_j)\mid i, j\in K_{(\frac{1}{2})}\ \mbox{and}\ i<j\}.
\]
Obviously, $\lambda_u\neq0$ for all $u\in K_{(\frac{1}{2})}$. On the other hand, Lemma \ref{alem2} shows that at least two of $\lambda_u$'s are $0$ ($\lambda_i=\lambda_n=0$ in Table \ref{simctab2}), a contradiction. Now the argument is similar to the case of type $B$. To summarize:

\begin{lemma}\label{alem3}
Let $\Phi=C_n$, $\lambda\in\Lambda_I^+$ and $\beta\in\Psi_\lambda^{+}$. Suppose that $\Phi(\beta)$ is not of type $A$. Then $\beta$ must satisfy one of the conditions listed in Table \ref{simctab3} (where $q_{s-1}<i, j, k\leq q_s$ and $q_{m-1}<l\leq n$).
\end{lemma}

\begin{table}[htbp]\footnotesize
\begin{tabular}{|c|c|p{2.5cm}|p{1.9cm}|p{1.6cm}|c|c|}
\hline
 & $\Phi_\lambda^+(0)\cap\Phi_3(\beta)$ & \centering $\Phi_\lambda^+(a)\cap\Phi_3(\beta)$ & \centering $\Psi_\lambda^+\cap\Phi(\beta)$ & \centering $\Psi_\lambda^{++}\cap\Phi(\beta)$ & basic sys. & simple \\
\hline
1 & \multirow{4}{*}{$\emptyset$} & \centering$e_j+e_k$ $(j<k)$ & \centering$2e_j$ & \centering  $\emptyset$ & $(C_2, 2, 2)$ & Yes\\
\cline{1-1} \cline{3-7}
2 & & \centering$e_j-e_l$ &\centering $2e_j$, $e_j+e_l$ & \centering  $2e_j$, $e_j+e_l$ & $(C_2, 1, 2)$ & Yes\\
\cline{1-1} \cline{3-7}
3 & & \centering$e_k+e_l$ & \centering$\emptyset$ & \centering  $\emptyset$ & $-$ & $-$ \\
\cline{1-1} \cline{3-7}
4 & & \centering$e_j+e_k$, $e_j-e_l$, $e_k+e_l$ $(j<k)$ & \centering$2e_j$, $e_j+e_l$ & \centering $\emptyset$ & $(C_3, 2, 3)$ & Yes\\
\hline
5 & \multirow{6}{*}{\centering $2e_i$} & \centering\multirow{2}{*}{$\emptyset$} & \centering$2e_j$, $e_j+e_i$, $(j<i)$ & \centering  $2e_j$, $e_j+e_i$ & $(C_2, 2, 1)$ & Yes\\
\cline{1-1} \cline{4-7}
6 & &  & \centering$e_i+e_l$ & \centering  $\emptyset$ & $(C_2, 1, 1)$ & Yes\\
\cline{1-1} \cline{3-7}
7 & & \centering$e_j+e_k$ $(j<i<k)$ & \centering$2e_j$, $e_j+e_i$ & \centering  $\emptyset$ & $(C_3, 3, 2)$ & Yes\\
\cline{1-1} \cline{3-7}
8 & & \centering$e_j-e_l$ $(j<i)$& \centering$2e_j$, $e_j+e_i$, $e_j+e_l$, $e_i+e_l$ &  \centering $2e_j$, $e_j+e_i$, $e_j+e_l$ & $(C_3, 2, 2)$ & No\\
\cline{1-1} \cline{3-7}
9 & & \centering$e_k+e_l$ $(i<k)$ & \centering$e_i+e_l$ &  \centering$\emptyset$ & $(C_3, 2, 2)$ & Yes\\
\cline{1-1} \cline{3-7}
10 & & \centering$e_j+e_k$, $e_j-e_l$, $e_k+e_l$ $(j<i<k)$ & \centering$2e_j$, $e_j+e_i$, $e_j+e_l$, $e_i+e_l$ & \centering$\emptyset$  & $(C_4, 3, 3)$ & Yes\\
\hline
\end{tabular}
\bigskip
\caption{reduction of type $C$}
\label{simctab3}
\end{table}

Similar to the proof of Lemma \ref{rctlem2}, we obtain the following result.

\begin{lemma}\label{rctlem4}
Theorem \ref{rcthm1} holds for $\Phi=C_n$.
\end{lemma}

\subsection{simplicity criteria for classical root systems}

In this subsection, we will give refinement of Jantzen's simplicity criteria for classical root systems. With Corollary \ref{janscor1} and Theorem \ref{rcthm1} in hand, we can first recover the following result of Jantzen.

\begin{theorem}[{\cite[Satz 4]{J}}] \label{jansthm3}
Let $\lambda\in\Lambda_I^+$. If all the irreducible component of $\Phi$ are of type $A$, then $M_I(\lambda)$ is simple if and only if $\Psi_\lambda^{++}=\emptyset$.
\end{theorem}

If $\Phi$ is not of type $A$, there are examples showing that $M_I(\lambda)$ could be simple even when $\Psi_\lambda^{++}\neq\emptyset$ (see the basic systems $(B_2, 1, 2)$, $(C_2, 1, 2)$ and $(D_4, 2, 2)$). Applying Corollary \ref{janscor1} and Theorem \ref{rcthm1} again, the simplicity criteria for the other classical types are given as follows.

\begin{theorem}\label{simcthm1}
Suppose that $\Phi=B_n$. Choose $I\subset\Delta$. Fix $\lambda\in\Lambda_I^+$. Then $M_I(\lambda)$ is simple if and only if $\Psi_\lambda^{++}$ contains only the following roots:
\begin{itemize}
\item [(1)] $e_i$, $e_i+e_j$ for $q_{s-1}<i\leq q_s\leq q_{m-1}<j\leq n$ and $1\leq s<m$. Moreover, $\lambda_i=\lambda_j\in\frac{1}{2}\bbZ^{>0}$ and $\lambda_k\neq 0$, $-\lambda_i$ for $q_{s-1}<k\leq q_s$.

\item [(2)] $e_i$, $e_i+e_j$ for $q_{s-1}<i<j\leq q_s$ and $1\leq s<m$. Moreover, $\lambda_i\in\bbZ^{>0}$, $\lambda_j=0$, $\lambda_k\neq-\lambda_i$ for $q_{s-1}<k\leq q_s$ and $\lambda_l\neq\lambda_i$ for $q_{m-1}<l\leq n$.
\end{itemize}
\end{theorem}

\begin{theorem}\label{simcthm2}
Suppose that $\Phi=C_n$. Choose $I\subset\Delta$. Fix $\lambda\in\Lambda_I^+$. Then $M_I(\lambda)$ is simple if and only if $\Psi_\lambda^{++}$ contains only the following roots:
\begin{itemize}
\item [(1)] $2e_i$, $e_i+e_j$ for $q_{s-1}<i\leq q_s\leq q_{m-1}<j\leq n$ and $1\leq s<m$. Moreover, $\lambda_i=\lambda_j\in\bbZ^{>0}$ and $\lambda_k\neq 0$, $-\lambda_i$ for $q_{s-1}<k\leq q_s$.

\item [(2)] $2e_i$, $e_i+e_j$ for $q_{s-1}<i<j\leq q_s$ and $1\leq s<m$. Moreover, $\lambda_i\in\bbZ^{>0}$, $\lambda_j=0$, $\lambda_k\neq-\lambda_i$ for $q_{s-1}<k\leq q_s$ and $\lambda_l\neq\lambda_i$ for $q_{m-1}<l\leq n$.
\end{itemize}
\end{theorem}

\begin{theorem}\label{simcthm3}
Suppose that $\Phi=D_n$. Choose $I\subset\Delta$. Fix $\lambda\in\Lambda_I^+$. Then $M_I(\lambda)$ is simple if and only if $\Psi_\lambda^{++}$ contains only the following roots: $e_i+e_j$, $e_i+e_k$ for $q_{s-1}<i<j\leq q_s\leq q_{m-1}<k<n$ and $1\leq s<m$. Moreover, $\lambda_i=\lambda_k\in\bbZ^{>0}$, $\lambda_j=\lambda_n=0$ and $\lambda_l\neq-\lambda_i$ for $q_{s-1}<l\leq q_s$.
\end{theorem}


\begin{cor}\label{simccor1}
Let $\Phi=D_n$ and $I\subset\Delta$. If $\lambda\in\Lambda_I^+$ contains at most one $0$-entry, then $M_I(\lambda)$ is simple if and only if $\Psi_\lambda^{++}=\emptyset$.
\end{cor}
\begin{proof}
In view of Corollary \ref{janscor1}, $\Psi_\lambda^{++}=\emptyset$ implies $M_I(\lambda)$ is simple. Conversely, If $M_I(\lambda)$ is simple and $\Psi_\lambda^{++}\neq\emptyset$, Theorem \ref{simcthm3} shows that $\lambda$ has at least two $0$-entries, a contradiction.
\end{proof}

\begin{remark}\label{simcrmk2}
We have $\theta(s_{e_i}\lambda)+\theta(s_{e_i+e_j}\lambda)=0$ in Theorem \ref{simcthm1}; while $\theta(s_{2e_i}\lambda)+\theta(s_{e_i+e_j}\lambda)=0$ in Theorem \ref{simcthm2} and $\theta(s_{e_i+e_j}\lambda)+\theta(s_{e_i+e_k}\lambda)=0$ in Theorem \ref{simcthm3}.
\end{remark}

\end{document}